\documentclass{article}
\usepackage[nottoc]{tocbibind}
\usepackage[titletoc,title]{appendix}
\usepackage[utf8]{inputenc}
\usepackage{times}

\usepackage{authblk}

\usepackage[pdftex]{graphicx}
\usepackage[pdftex,dvipsnames,usenames]{color}

\usepackage{multirow}
\usepackage{enumitem}
\usepackage{mathtools}

\usepackage{tikz}
\usetikzlibrary{arrows,shapes,patterns,calc,fadings,decorations.pathreplacing,decorations.markings,decorations.pathmorphing,backgrounds}

\usepackage{cancel}

\usepackage[english]{babel}

\usepackage[sc]{mathpazo}
\usepackage[hyperindex,breaklinks]{hyperref}

\hypersetup{
  colorlinks   = true, % Colours links instead of ugly boxes
  urlcolor     = blue, % Colour for external hyperlinks
  linkcolor    = blue, % Colour of internal links
  citecolor   = red % Colour of citations
}

\usepackage{amsmath,amssymb,amsfonts,mathrsfs, amsthm}

\usepackage{graphicx}
\usepackage[margin=1in]{geometry}

\usepackage{soul}

\usepackage{pdfpages}

\usepackage{lipsum}
\usepackage{csquotes}
\usepackage{thmtools}

\usepackage{amsfonts,amsmath,amstext,amssymb,amsthm}
\usepackage{enumitem}
\usepackage[english]{babel}
\usepackage{fancybox}
\usepackage{xspace}
\usepackage{amscd}
\usepackage[all]{xy}
\usepackage{calc}

\usepackage{esint}
\usepackage{authblk}

\usepackage{cases}
\usepackage{eqnarray}

\usepackage{accents}

\newcommand{\B}{\mathcal{B}}
\def\XXint#1#2#3{{\setbox0=\hbox{$#1{#2#3}{\int}$ }
\vcenter{\hbox{$#2#3$ }}\kern-.56\wd0}}

\newcommand{\N}{\mathbb{N}}

\newcommand{\R}{\mathbb{R}}

\newcommand{\ep}{\varepsilon}
\newcommand{\conv}[1]{\xrightarrow{\,#1\,}}

\newcommand{\dive}{\text{div}}
\newcommand{\s}{\hspace{7pt}}
\newcommand{\vsp}{\vspace{4pt}}
\newcommand{\ov}[1]{\widetilde{#1}}
\newcommand{\lp}{\langle}
\newcommand{\rp}{\rangle}

\newcommand{\inj}{\textnormal{inj}}

\newtheorem{theorem}{Theorem}[section]
\newtheorem{proposition}[theorem]{Proposition}
\newtheorem{lemma}[theorem]{Lemma}

\theoremstyle{definition}

\newtheorem{definition}[theorem]{Definition}

\newtheorem{remark}[theorem]{Remark}

\usepackage{url}
\usepackage{authblk}

\makeatletter
\newcommand{\addressa}[1]{\gdef\@addressa{#1}}
\newcommand{\emaila}[1]{\gdef\@emaila{\url{#1}}}

\newcommand{\addressb}[1]{\gdef\@addressb{#1}}
\newcommand{\emailb}[1]{\gdef\@emailb{\url{#1}}}

\newcommand{\addressc}[1]{\gdef\@addressc{#1}}
\newcommand{\emailc}[1]{\gdef\@emailc{\url{#1}}}

\newcommand{\@endstuff}{\par\vspace{\baselineskip}\noindent
\begin{tabular}{@{}l}\scshape\@addressa\\\textit{E-mail address:} \@emaila\end{tabular} 

\vspace{12pt} \noindent
\begin{tabular}{@{}l}\scshape\@addressb\\ \textit{E-mail address:} \@emailb\end{tabular}

\vspace{12pt} \noindent
\begin{tabular}{@{}l}\scshape\@addressc\\ \textit{E-mail address:} \@emailc\end{tabular}
}

\AtEndDocument{\@endstuff}
\makeatother

\begin{document}

%%% CHECK CITATIONS ------  
%  \nocite{*}

%%%%%%  TITLE %%%%%%%%%%%%%%%%%%%%%%%%%%%
\title{Fractional Sobolev spaces on Riemannian manifolds}
\date{\today}

%%%%%% Authors %%%%%%%%%%%%%%%%%%%%%%%%%%%

\author{Michele Caselli, Enric Florit-Simon and Joaquim Serra}

%%%%%% Info and addresses %%%%%%%%%%%%%%%%%%%%%%%%%%%

\addressa{Michele Caselli \\ Scuola Normale Superiore\\ Piazza dei Cavalieri 7, 56126 Pisa, Italy }
\emaila{michele.caselli@sns.it}

\addressb{Enric Florit-Simon \\ Department of Mathematics, ETH Z\"{u}rich \\ Rämistrasse 101, 8092 Zürich, Switzerland}
\emailb{enric.florit@math.ethz.ch}

\addressc{Joaquim Serra \\ Department of Mathematics, ETH Z\"{u}rich \\ Rämistrasse 101, 8092 Zürich, Switzerland }
\emailc{joaquim.serra@math.ethz.ch}

%%%%%%%%%%%%%%%%%%%%%%%%%%%%%%%%%%%%%%%%%%

\setlength\parindent{12pt}

\maketitle

\begin{abstract}
This article studies the canonical Hilbert energy $H^{s/2}(M)$ on a Riemannian manifold for $s\in(0,2)$, with particular focus on the case of closed manifolds. Several equivalent definitions for this energy and the fractional Laplacian on a manifold are given, and they are shown to be identical up to explicit multiplicative constants. Moreover, the precise behaviour of the kernel associated with the singular integral definition of the fractional Laplacian is obtained through an in-depth study of the heat kernel on a Riemannian manifold. Furthermore, a monotonicity formula for stationary points of functionals of the type $\mathcal E(v)=[v]^2_{H^{s/2}(M)}+\int_M F(v) \, dV$, with $F\geq 0$, is given, which includes in particular the case of nonlocal $s$-minimal surfaces. Finally, we prove some estimates for the Caffarelli-Silvestre extension problem, which are of general interest.

\vspace{2pt}
This work is motivated by \cite{CFS}, which defines nonlocal minimal surfaces on closed Riemannian manifolds and shows the existence of infinitely many of them for any metric on the manifold, ultimately proving the nonlocal version of a conjecture of Yau (\cite{Yau82}). Indeed, the definitions and results in the present work serve as an essential technical toolbox for the results in \cite{CFS}.
\end{abstract} 

\tableofcontents

\section{Introduction}

In recent years we have seen a great development of the theory of nonlocal equations. The simplest example of a nonlocal operator is the fractional Laplacian, $(-\Delta)^{\sigma}u=\int_{\R^n} (u(x)-u(y))\frac{1}{|x-y|^{n+2\sigma}}$, where $\sigma\in(0,1)$ and $u:\R^n\to\R$. Formally, it corresponds to the $\sigma$-th power of the usual Laplacian, and it is, therefore, an operator of order (of differentiation) $2\sigma$. Another way to look at it is as the operator arising from the Euler-Lagrange equation of the functional
$$
 \mathcal E(v)=[v]^2_{H^\sigma(\R^n)}=\iint_{\R^n\times\R^n} (u(x)-u(y))^2\frac{1}{|x-y|^{n+2\sigma}}\,,$$
 which involves a fractional Sobolev energy term. There are precise multiplicative constants one should put in front of these objects which will be given later, but we will omit them in the introduction for the sake of exposition.

 \vsp
The present work addresses how the fractional Sobolev energy $H^{\sigma}(M)=W^{\sigma,2}(M)$ and the associated fractional Laplacian on $M$ have a natural, canonical interpretation in the case where $M$ is a closed Riemannian manifold. We give several definitions for these objects and show them to be identical, which justifies their canonical nature. Moreover, we obtain fundamental properties for these objects thanks to a deeper study of their different definitions. We give a brief account in what follows:

\vsp
Let $(M^n, g)$ be an $n$-dimensional, closed Riemannian manifold, with $n\ge2$.
Let us start by giving the definition of the fractional Sobolev seminorm $H^{s/2}(M)$ ---for convenience and consistency with the notation in our paper \cite{CFS}, throughout the paper we put $\sigma = s/2$, where $s\in (0,2)$.

The $H^{s/2}(M)$ seminorm can be defined in at least three equivalent ways:
 \begin{itemize}
 \item[(i)] Using the {\em heat kernel}\footnote{As customary, by heat kernel here we mean
the fundamental solution of the heat equation $\partial_t u = \Delta u$ on $M$, where $\Delta$ denotes the Laplace-Beltrami operator on $M$.} $H_M(t,p,q)$ of  $M$, we can put
 \begin{equation}\label{wethiowhoihw2}
K_s(p,q) :=  %\frac{\pi^{n/2}}{4^s \Gamma(s+\frac n2 )}
 \int_{0}^\infty  H_M(p,q,t)\,\frac{dt}{t^{1+s/2}}.
 \end{equation}
We then define
 \begin{equation}\label{wethiowhoihw}
 [u]^2_{H^{s/2}(M)} := \iint_{M\times M}(u(p)-u(q))^2 K_s(p,q) \,dV_p \,d V_q.
 \end{equation}

\item[(ii)]  Following a  {\em spectral approach}, we can set
\begin{equation}\label{ghfghfg}
    [u]^2_{H^{s/2}(M)} = \sum_{k\ge 1} \lambda_k^{s/2} \langle u,\phi_k \rangle^2_{L^2(M)}
\end{equation}
 where $\{\phi_k\}_k$ is an orthonormal basis of eigenfunctions of the Laplace-Beltrami operator $(-\Delta_g)$ and $\{\lambda_k\}_k$ are the corresponding eigenvalues. For $s=2$ this immediately recovers the usual $[u]^2_{H^{1}(M)}$ seminorm.

\item[(iii)] Considering a {\em Caffarelli-Silvestre type extension (cf. \cite{CafSi,BGS})}, namely, a degenerate-harmonic extension problem in one extra dimension, we can set
\begin{equation}\label{cafextintro}
[u]^2_{H^{s/2}(M)} = \inf\left\{ \int_{M\times \R_+} z^{1-s} |\ov{\nabla}  U(p,z)|^2 \, dV_p dz \s {\mbox{s.t.}} \s U(x, 0)=u(x) \right\}\,.
\end{equation}

Here $\widetilde \nabla$ denotes the Riemannian gradient of the manifold $\widetilde M = M\times (0, \infty)$, with respect natural product metric $\widetilde g= g + dz\otimes dz$, and the infimum is taken over all  $U$ belonging to the weighted Hilbert space $\widetilde{H}^1(\widetilde{M})$ (see Definition \ref{weighSobspacemfd} for the precise definition of this space, and we refer to Section \ref{ext section} in general for all the basic properties of this extension characterization).
\end{itemize}
\vsp
We will prove that (i)-(iii) define the same norm (not merely equivalent norms) up to explicit multiplicative constants. We emphasize that this gives a canonical definition of the $H^{s/2}(M)$ seminorm on a closed manifold.\\ %(Notice that, for example, there is no canonical definition of a $W^{s,1}(M)$ seminorm) {\color{green} How do we justify that $\iint |u(x)-u(y)| K_s(x,y)\,dx\,dy$ is not its canonical form?}.

Definition (i) through the expression \eqref{wethiowhoihw} will allow us to control precisely the behaviour of the fractional Sobolev energy. See, for example, Lemmas \ref{enboundslemma} and \ref{enboundslemma2}, which show that the fractional Sobolev energy is smooth with quantitative bounds under inner variations. For that, we will give precise quantitative estimates for the kernel $K_s(p,q)$ (defined in \eqref{wethiowhoihw2}) and its derivatives, depending only on local quantities. In particular, we will show that it is comparable to $\frac{1}{d(p,q)^{n+s}}$ if $p$ and $q$ are contained in a Riemannian ball with controlled geometry. We recall that the two kernels coincide in the case $M=\R^n$ (up to a constant factor).\\
The estimates for $K_s$ will follow from corresponding estimates for the heat kernel $H_M$. Albeit of somewhat ``standard flavour", they are hard to find in the literature with this level of precision, and we give an almost completely self-contained account that we believe to be of independent interest.\\

The extension definition (iii) will be used to give a monotonicity formula for stationary points $u$ of semilinear elliptic functionals, that is, of functionals of the form
\begin{equation*}
    \mathcal E(v)=[v]^2_{H^{s/2}(M)}+\int_M F(v) \, dV ,
\end{equation*}
under the assumption that $F\geq 0$. More precisely, $u$ need only be stationary for $\mathcal E(v)$ under inner variations; in particular, setting $F\equiv 0$ will give a monotonicity formula for nonlocal $s$-minimal surfaces, which we will define in a moment. Up to now, the result was known on $\R^n$ by \cite{CRS}, \cite{CC} and \cite{MSW}.
\begin{definition}
Given $s \in (0, 1)$
and a (measurable) set $E \subset M$, we define the \textit{$s$-perimeter of $E$} as
\begin{equation}\label{wbhtouw1}
\textnormal{Per}_s(E): = [\chi_E]^2_{H^{{s/2}}(M)}= \frac{1}{4}[\chi_E-\chi_{E^c}]^2_{H^{{s/2}}(M)}\,, 
\end{equation}
where $\chi_E$ is the characteristic function of $E$ and $E^c: = M\setminus E$.\\
From the estimates that we will prove for $K_s$, one can see that for every set $E\subset M$ with smooth boundary, one has that
$(1-s)\textnormal{Per}_s(E) \to \textnormal{Per}(E)$  as $s\uparrow 1$ (up to a multiplicative dimensional constant, see \cite{BBM01}
and also~\cite{Dav02, CV11, ADPM11}
for further details on the computation in the case of $\R^n$). 
\end{definition}
\begin{definition}
The boundary $\partial E$ of a set $E\subset M$ is said to be an $s$-\textit{minimal surface} if  $\textnormal{Per}_s(E)<\infty$  and, for every smooth vector field $X$ on $M$, we have 
\begin{equation*}
     \frac{d}{dt}\Big|_{t=0} \textnormal{Per}_s(\psi^t_X(E)) = 0 \, ,
\end{equation*}
where $\psi^t_X : M \times \R \to M$ denotes the flow of $X$ at time $t$.
\end{definition}

\begin{remark}
As we prove\footnote{Notice that for functions taking values in $\{\pm 1\}$ the potential part of the energy vanishes and the Sobolev part of the energy gives the fractional perimeter.} in Lemma \ref{enboundslemma} and Lemma \ref{enboundslemma2}, if $\textnormal{Per}_s(E,\Omega)<\infty$ and $X\in \mathfrak X_c(\mathcal U)$ is such that  ${\rm spt}(X) \subset \overline\Omega$ then the map $t\mapsto \textnormal{Per}_s(\psi^t_X(E), \Omega)$ is well-defined for all $t$ and of class $C^{\infty}$. Thus, the previous definitions are meaningful. 
\end{remark}

In other words, $\partial E$ is an $s$-minimal surface if $u=\chi_E-\chi_{E^c}$ is stationary under inner variations for $\mathcal E(v)=[v]^2_{H^{s/2}(M)}$. The general monotonicity formula claimed before is the following (see Section \ref{GenMonFormula} for the precise definitions and notation):
\begin{theorem}[\textbf{Monotonicity formula}]
Let $(M^n, g)$ be an $n$-dimensional, closed Riemannian manifold. Let $s \in (0,2)$ and
\begin{equation*}
    \mathcal E(v)=[v]^2_{H^{s/2}(M)}+\int_M F(v) \, dV ,
\end{equation*}
 where $F$ is any smooth nonnegative function. Let $u:M\to \R $ be stationary for $\mathcal{E}$ under inner variations, meaning that  $\mathcal E(u)<\infty$ and for any smooth vector field $X$ on $M$ there holds $\frac{d}{dt}\big|_{t=0} \mathcal{E}(u\circ\psi_X^t)=0$, where $\psi_X^t$ is the flow of $X$ at time $t$. For $(p_\circ,0) \in \ov{M} $ and $R>0$ define 
\begin{equation*}
    \Phi(R) := \frac{1}{R^{n-s}} \left( \beta_s \int_{\widetilde{B}^+_R(p_\circ,0)} z^{1-s}| \nabla  U(p,z)|^2 \, dV_p dz +  \int_{B_R(p_\circ)} F(u) \, dV \right) , 
    \end{equation*}
where $U$ is the unique solution (see Theorem \ref{extmfd}) in $\widetilde H^1(\widetilde M)$ to
$$\begin{cases} \widetilde{ {\rm div}}(z^{1-s} \widetilde \nabla U) = 0  &   \mbox{in} \s \widetilde{M} \,,  \\ U(p,0) = u(p) &  \mbox{for} \s p \in \partial \widetilde{M} = M\, .  \end{cases}$$
Then, there exists a positive constant $C$ with the following property: whenever $R_\circ \le \inj_{M}(p_\circ)/4$ and $K$ is an upper bound for all the sectional curvatures of $M$ in $B_{R_\circ}(p_\circ)$, then
\begin{equation*}
    R \mapsto \Phi(R)e^{C \sqrt{K} R } \s \textit{is nondecreasing for} \s R < R_\circ \,,
\end{equation*}
and the inequality
\begin{align*}
    \Phi'(R) \ge - C \sqrt{K} \Phi(R)+\frac{ s }{R^{n-s+1}} \int_{B_R(p_\circ)} F(u) \, dV + \frac{2\beta_s}{R^{n-s}} \int_{\partial^+ \ov{B}_R^+(p_\circ,0)} z^{1-s} \lp \nabla U , \nabla d \rp ^2 \, d\,\widetilde{\sigma}
\end{align*}
holds for all $R < R_0 $, with $d(\cdot) = d_{\ov{g}}((p_\circ,0), \,\cdot\, )$ the distance function on $\ov{M}$ from the point $(p_\circ, 0)$.

\vsp
Moreover, in the particular case where $M=\R^n$, $F\equiv 0$, $s\in(0,1)$, and $u=\chi_E-\chi_{E^c}$ where $E$ is an $s$-minimal surface, there holds
\begin{align*}
    \Phi'(R) = \frac{2\beta_{s}}{R^{n-s}} \int_{\partial^+ \ov{\B}_R^+(p_\circ,0)} z^{1-s} \lp \nabla U , \nabla d \rp ^2 \, dxdz \ge 0 \,,
\end{align*}
which shows that $\Phi$ is nondecreasing and that it is constant if and only if $E$ is a cone. 
\end{theorem}

\subsection{Overview of the kernel estimates}

Here is an overview of the estimates for the heat kernel $H_M$ and the singular kernel $K_s$ (defined in \eqref{wethiowhoihw2}) that will be proved in \autoref{Ker prop section}. In particular, the reader is advised to consult \autoref{prop:kern1}, which records several of the main results for $K_s$, including an explicit asymptotic expansion for short distances.

\begin{center}
\begin{tabular}{ |c|c|c| } 
\hline
 & Heat kernel $H_M$ & Singular kernel $K_s$ \\
\hline
Global comparability on $(\R^n,g)$ & \autoref{globcomparability} & \autoref{globcomparability} \\  \hline
Short distance comparability & \autoref{lemaux0}, \autoref{lemaux2} & \autoref{loccomparability} \\  \hline
Long distance estimates & \autoref{lemaux1}, \autoref{lemaux12}  & \autoref{prop:kern1}, \autoref{prop:kern2}  \\ 
\hline
Precise asymptotics & \autoref{lemaux3} &  \autoref{prop:kern1} \\ 
\hline
\end{tabular}
\end{center}

\section{The fractional Laplacian on a closed manifold \texorpdfstring{$(M,g)$}{}}\label{LapSection}
Throughout the paper (unless otherwise stated) $(M,g)$ will be a closed (i.e. compact and without boundary) Riemannian manifold of dimension $n$.\\
Taking inspiration from the case of $\R^n$, in this section, we give several equivalent definitions for the fractional Laplacian $(-\Delta)^{s/2}$ on a closed Riemannian manifold, with $s \in (0,2)$.

\subsection{Spectral and singular integral definitions}

The fractional Laplacian $(-\Delta)^{s/2}$ can be defined as the $s/2$-th power (in the sense of spectral theory) of the usual Laplace-Beltrami operator on a Riemannian manifold, through Bochner's subordination.

\vsp
Given $\lambda >0$ and $s \in (0,2)$, the following numerical formula holds:
\begin{equation}\label{numeric1}
    \lambda^{s/2} = \frac{1}{\Gamma(-s/2)} \int_{0}^{\infty} (e^{-\lambda t}-1)\frac{dt}{t^{1+s/2}} \, .
\end{equation}
Formally applying the above relation to the operator $L=(-\Delta)$ in place of $\lambda$, one obtains the following definition for the fractional Laplacian.
\begin{definition}[\textbf{Spectral definition}]
Let $s\in(0,2)$. The fractional Laplacian $(-\Delta)^{s/2}$ is the operator that acts on smooth functions $u$ by
\begin{equation}\label{boclap}
   (-\Delta)^{s/2} \, u = \frac{1}{\Gamma(-s/2)} \int_{0}^{\infty} (e^{t\Delta}u-u)\frac{dt}{t^{1+s/2}}\, .
\end{equation}
Here, the expression $e^{t\Delta}u$ is to be understood as the solution of the heat equation on $M$ at time $t$ and with initial datum $u$. 

\vsp 
From now on, to denote the solution of the heat equation with initial datum $u$, we will write $P_t u$ in place of $e^{t\Delta } u $.
%Since the exponential of bounded operators between Banach spaces is always well defined, heuristically it is reasonable that doing so the right hand side of \eqref{boclap} converges. Indeed, the convergence of the right hand side of \eqref{boclap} can be shown for a wide class of $u$,  
\end{definition}
\begin{remark}
On a closed Riemannian manifold a closely related definition of the fractional Laplacian is available: if $\{\phi_k\}_{k=1}^{\infty} $ is an $L^2(M)$ orthonormal basis of eigenfunctions for $(-\Delta) $ with eigenvalues
\begin{equation*}
    0 < \lambda_1 < \lambda_2 \le \dotsc \le \lambda_k \conv{k \to \infty} +\infty
\end{equation*}
and $u$ is a smooth function then
\begin{equation*}
    (-\Delta)^{s/2} u = \sum_{k=1}^\infty \lambda_k^{s/2} \lp u, \phi_k\rp_{L^2(M)} \phi_k \,. 
\end{equation*}
Since the solution to the heat equation on $M$ with initial datum an eigenfunction $\phi_k$ is given by $e^{t\Delta}\phi_k = e^{-\lambda_k t}\phi_k$, the above definition is easily shown to be identical to \eqref{boclap} by first observing that they coincide for eigenfunctions (thanks to \eqref{numeric1}), and then extending the result by approximation. In \cite{Stinga} (see also \cite{CafStinga}) all the details of this equivalence are carried out in the case of certain positive second order operators with discrete spectrum on a domain $\Omega \subset \R^n$. In our case of $(-\Delta)$ on a closed Riemannian manifold, the proof is then completely analogous. Nevertheless, this characterization will not be used in what follows and is given only as complementary information.
\end{remark}

The second definition for the fractional Laplacian, closely related to the spectral one, expresses it as a singular integral. It will be our working definition in a substantial portion of the article.
\begin{definition}[\textbf{Singular integral definition}]
The fractional Laplacian $(-\Delta)^{s/2}$ of order (of differentiation) $s \in (0,2)$ is the operator that acts on a regular function $u$ by
\begin{align}\label{singintlap}
    (-\Delta)^{s/2} u(p) &= p.v. \int_{M} (u(p)-u(q))K_s(p,q)\, dV_q\\
    &:=\lim_{\ep\to 0}\int_{M} (u(p)-u(q))K_s^{\ep}(p,q) \, dV_q \,. \nonumber
\end{align}
Here $K_s(p,q) : M \times M \to \R$ denotes the singular kernel given by\footnote{Note that $\frac{1}{|\Gamma(-s/2)|}= \frac{s/2}{\Gamma(1-s/2)}$. }
\begin{equation}\label{kerneleq}
    K_s(p,q)=\frac{s/2}{\Gamma(1-s/2)} \int_{0}^{\infty} H_M(p,q,t) \frac{dt}{t^{1+s/2}}
\end{equation}
where $H_M:M\times M \times (0,\infty)\to \R$ denotes the usual heat kernel on $M$,
and $K_s^{\ep}(p,q)$ is the natural regularization
\begin{equation}\label{app ker K}
    K_s^{\ep}(p,q)=\frac{s/2}{\Gamma(1-s/2)} \int_{0}^{\infty} H_M(p,q,t) e^{-\ep^2/4t} \frac{dt}{t^{1+s/2}}\, .
\end{equation}
\end{definition}

\begin{remark}\label{cvcvcvcv}
If the compact manifold $M$ is replaced by the Euclidean space $\R^n$ then 
\begin{equation*}
    K_{s}(x,y)=\frac{s/2}{\Gamma(1-s/2)} \int_{0}^{\infty} H_{\R^n}(x,y,t) \frac{dt}{t^{1+s/2}} = \frac{s/2}{\Gamma(1-s/2)} \int_{0}^{\infty} \left(  \frac{1}{(4\pi t)^{\frac{n}{2}}} e^{-\frac{|x-y|^2}{4t}}\right) \frac{dt}{t^{1+s/2}} = \frac{\alpha_{n,s}}{|x-y|^{n+s}} \,,
\end{equation*}
where 
\begin{equation}\label{alphadef}
    \alpha_{n,s}= \frac{2^s \Gamma\Big(\tfrac{n+s}{2}\Big)}{\pi^{ n/2} |\Gamma(-s/2)| } = \frac{s 2^{s-1} \Gamma\Big(\tfrac{n+s}{2}\Big)}{\pi^{n/2}\Gamma(1-s/2)}.
\end{equation}
Hence, we recover the usual form of the fractional Laplacian on $\R^n$.
\end{remark}

Let us briefly comment on our choice of the "natural" regularization $K_s^{\ep}(p,q)$ that we have used to define the principal value $(p.v.)$ in \eqref{singintlap}. First, we will see in the proof of \eqref{lapextmfd} that this approximation naturally appears in the computation. This is because $K_s^{\ep}(p,q)$ is directly related to the fractional Poisson kernel $\mathbb{P}_{\widetilde M}(p,q,z)$ of $\widetilde M := M\times (0, +\infty)$ by the formula
    \begin{equation*}
        K_s^{\ep}(p,q) = \mathbb{P}_{\widetilde M}(p,q,\ep)\ep^{-s} \,,
    \end{equation*}
    and the fractional Poisson kernel is the fundamental solution of the Caffarelli-Silvestre extension problem. 

    \vsp
    Moreover, if the compact manifold $M$ is replaced by the Euclidean space $\R^n$ then 
    \begin{equation*}
        K_s^{\ep}(x,y) = \frac{s/2}{\Gamma(1-s/2)} \int_{0}^{\infty} H_{\R^n}(x,y,t) e^{-\ep^2/4t}\frac{dt}{t^{1+s/2}} = \frac{\alpha_{n,s}}{(|x-y|^2+\ep^2)^{\frac{n+s}{2}}} \,,
    \end{equation*}
    which is arguably a very natural regularization of $\frac{\alpha_{n,s}}{|x-y|^{n+s}}$, and is easily seen to give the same notion of principal value that one would get by integrating $K_s(x,y)$ against a function on $\R^n\setminus B_{\ep}(y)$ and then taking $\ep\to 0^+$. This is also true on a Riemannian manifold, and actually many other desingularizations of the (singular) kernel $K_s(p,q)$ are possible and give the same notion of principal value $(p.v.)$ under mild hypotheses:
\begin{proposition}\label{prop:pvequiv}
Let $(M,g)$ be a closed, $n$-dimensional Riemannian manifold, and let $\{K_s^\ep\}_{\ep >0}$ be a family of nonnegative kernels defined on $L^\infty(M)$. Assume that the following hold:
\begin{itemize}
    \item The $K_s^\ep$ converge uniformly to $K_s(p,\cdot)$ on compact subsets of $M\setminus\{p\}$, as $\ep\to 0^+$.
    \item There exist some $r=r(p)>0$ and some chart parametrization $\varphi:\B_{r}\subset\R^n\to M$ with $\varphi(0)=p$ such that:
    \begin{itemize}
        \item[(i)] The flatness assumptions ${\rm FA}_2(M,g, r,p,\varphi)$ are satisfied (see Definition \ref{flatnessassup}).
        \item[(ii)] Setting $\widetilde K_s^\ep(y):=K_s^\ep(p,\varphi(y))$, there is some positive constant $C$ such that, for all $y\in \B_{r_p}$,
        \begin{equation}\label{kte1}
            \widetilde K_s^\ep(y)\leq \frac{C}{|y|^{n+s}}
        \end{equation}
        and moreover, the symmetry condition
        \begin{equation}\label{kte2}
            \Big|\widetilde K_s^\ep(y)-\widetilde K_s^\ep(-y)\Big|\leq \frac{C}{|y|^{n+s-1}}
        \end{equation}
        is satisfied.
    \end{itemize} 
\end{itemize}
Then, for every $f\in C^\infty(M)$, the limit $\lim_{\ep\to 0}\int_M (f(p)-f(q)) K_s^\ep(p,q) \, dV_q$ exists and is independent of the family $K_s^\ep$. In particular, any such family gives the same value for \eqref{singintlap} as the choice in \eqref{app ker K}.
\end{proposition}
\begin{remark}
    As discussed above, this covers the case of removing a geodesic ball $B_{\ep}(p)$ in the corresponding definition of the fractional Laplacian as an integral and then sending $\ep$ to $0$, as in the usual Euclidean definition of a principal value integral. Indeed, this corresponds to considering $K_s^\ep:=K_s(p,q)\chi_{M\setminus B_{\ep}(p)}$ in the Proposition above. Another reasonable desingularization could be
    \begin{equation*}
    K_s^{\ep}(p,q)=\frac{s/2}{\Gamma(1-s/2)} \int_{\ep}^{\infty} H_M(p,q,t) \frac{dt}{t^{1+s/2}} \,. 
    \end{equation*}
    The fact that both of these choices for the families $\{K_s^\ep\}_\ep$, as well as the choice \eqref{app ker K}, satisfy the hypotheses in Proposition \ref{prop:pvequiv}, can be easily seen using the results that will appear in the next section. More precisely, they follow from the combination of Remark \ref{fbsvdg} and (the proof of) estimate \eqref{remaining0} from Theorem \ref{prop:kern1}. The latter shows that conditions \eqref{kte1} and \eqref{kte2} hold directly for the kernel $K_s$ thanks to precise estimates on the heat kernel $H_M$, and it is then simple to see that they hold for the regularisations $K_s^{\ep}$ as well.
\end{remark}
\begin{proof}[Proof of Proposition \ref{prop:pvequiv}]
    Recall that $\varphi^{-1}(p)=0$. Given $0<\delta<r_p$, we can write
    \begin{align*}
        \int_M (f(p)-f(q))K_s^\ep(p,q)\, dV_q&= \int_{M\setminus \varphi(\B_\delta)} (f(p)-f(q))K_s^\ep(p,q)\, dV_q+\int_{\varphi(\B_\delta)} (f(p)-f(q))K_s^\ep(p,q)\, dV_q\,.
    \end{align*}
    
    By the dominated convergence theorem, the first term on the RHS converges to 
    \begin{equation*}
        \int_{M\setminus \varphi(\B_\delta)} (f(p)-f(q))K_s(p,q)\, dV_q \,.
    \end{equation*}
 Therefore
    \begin{align*}
        \limsup_{\ep\to 0} \Big|\int_M (f(p)-f(q)) & K_s^\ep(p,q)\, dV_q-\int_{M\setminus \varphi(\B_\delta)} (f(p)-f(q))K_s(p,q)\, dV_q\Big| \\ & \leq \limsup_{\ep\to 0}\Big|\int_{\varphi(\B_\delta)} (f(p)-f(q))K_s^\ep(p,q)\, dV_q\Big| \,.
    \end{align*}
    In other words, to conclude our desired result, it suffices to show that $\int_{\varphi(\B_\delta)} (f(p)-f(q))K_s^\ep(p,q)\, dV_q$ is bounded independently of $\ep$ and moreover can be made arbitrarily small by choosing $\delta$ small enough.

    \vsp 
    To check this, we start by changing variables using the coordinates given by $\varphi$, leading to
    \begin{align*}
        \int_{\varphi(\B_\delta)} (f(p)-f(q))K_s^\ep(p,q)\, dV_q&=\int_{\B_\delta} (f(\varphi(0))-f(\varphi(y))) \widetilde K_s^\ep(y)\sqrt{|g|}(y)\,dy\,.
    \end{align*}
    Defining $h(y):=(f(\varphi(0))-f(\varphi(y)))\sqrt{|g|}(y)$, which verifies that $h(y)=y\cdot \nabla h(0)+O(|y|^2)$, and using the symmetry of the Lebesgue measure under the transformation $y\mapsto (-y)$, we can then compute
    \begin{align*}
    \Big|\int_{\varphi(\B_\delta)} (f(p)-f(q))K_s^\ep(p,q)\, dV_q\Big|&=\Big|\int_{\B_{\delta}} h(y) \widetilde K_s^\ep(y)\,dy\Big|\\
    &\leq\Big|\int_{\B_{\delta}} y\cdot \nabla h(0) \widetilde K_s^\ep(y)\,dy\Big|+C\int_{\B_{\delta}} |y|^2 \widetilde K_s^\ep(y)\,dy\\
    &= \Big|\frac{1}{2} \int_{\B_{\delta}} y\cdot \nabla h(0) \widetilde K_s^\ep(y)\,dy+\frac{1}{2} \int_{\B_{\delta}} (-y)\cdot \nabla h(0) \widetilde K_s^\ep(-y)\,dy\Big|\\
    &\hspace{1cm}+C\int_{\B_{\delta}} |y|^2 \widetilde K_s^\ep(y)\,dy\\
    &=\frac{1}{2} \Big|\int_{\B_{\delta}} y\cdot \nabla h(0) (\widetilde K_s^\ep(y)-\widetilde K_s^\ep(-y))\,dy\Big|+C\int_{\B_{\delta}} |y|^2 \widetilde K_s^\ep(y)\,dy
    \end{align*}
    Using the assumptions on the kernel, we conclude that
    \begin{align*}
    \Big|\int_{\varphi(\B_\delta)} (f(p)-f(q))K_s^\ep(p,q)\, dV_q\Big|&\leq C\int_{\B_{\delta}} |y\cdot \nabla h(0)| \frac{1}{|y|^{n+s-1}}\,dy+C\int_{\B_{\delta}} |y|^2\frac{1}{|y|^{n+s}}\,dy\\
    &\leq C\int_{\B_{\delta}} \frac{1}{|y|^{n+s-2}}\,dy\\
    &\leq C\delta^{2-s}\,.
    \end{align*}
    Since $s\in(0,2)$, this quantity can be made arbitrarily small by choosing $\delta$ small enough, independently of $\ep$. This concludes the proof of our result.
\end{proof}

We now show the equivalence between the spectral and singular integral definitions for the fractional Laplacian.
\begin{proposition}
    For every $s\in (0,2)$ definitions \eqref{boclap} and \eqref{singintlap} coincide, meaning that:
    \begin{itemize}
        \item[$(i)$] For $u \in C^\infty(M)$ they coincide pointwise everywhere.
        \item[$(ii)$] For $u\in L^2(M)$ they coincide as distributions (i.e. in duality with $C^\infty(M)$).
    \end{itemize}
\end{proposition}
\begin{proof}
   Let $u\in C^{\infty}(M)$. Expressing the solution $P_t u$ to the heat equation in terms of the initial datum as 
\begin{equation*}
   P_t u(p)=\int_M u(q) H_M(p,q,t) \,dV_q \,,
\end{equation*}
and using that $\int_M H_M(p,q,t) \,dV_q=1$ gives, for every $\ep>0$, that
\begin{equation}\label{qhwerfgwerf}
    \frac{1}{\Gamma(-s/2)} \int_{0}^{\infty} (P_t u-u)\frac{e^{-\ep^2/4t} dt}{t^{1+s/2}} = \int_M (u-u(q))K_s^\ep(p,q) \, dV_q \,.
\end{equation}
Since $u$ is smooth, letting $\ep \to 0^+$ gives convergence of both integrals pointwise everywhere and 
\begin{equation*}
    \frac{1}{\Gamma(-s/2)} \int_{0}^{\infty} (P_t u-u)\frac{dt}{t^{1+s/2}} = p.v. \int_M (u-u(q))K_s(p,q) \, dV_q \,,
\end{equation*}
and this proves $(i)$. 

\vsp
Now to show $(ii)$ take $u\in L^2(M)$ and $\varphi \in C^\infty(M)$. Multiply \eqref{qhwerfgwerf} by $\varphi$ and integrate over $M$ to get
\begin{equation}\label{qhwerfgwerf2}
    \frac{1}{\Gamma(-s/2)} \int_M \int_{0}^{\infty} (P_t u-u) \varphi \frac{e^{-\ep^2/4t}}{t^{1+s/2}} \, dt dV = \iint_{M\times M} (u(p)-u(q))\varphi(p) K_s^\ep(p,q) \, dV_q dV_p \,.
\end{equation}
Note that since $\ep>0$ is fixed and positive, neither of the two integrals above is singular, and they are both absolutely convergent. Hence, we can exchange the other of integration in both integrals. For the left-hand-side using that $P_t$ is self adjoint in $L^2(M)$ we get 
\begin{align*}
    \frac{1}{\Gamma(-s/2)} \int_M \int_{0}^{\infty} (P_t u-u) \varphi \frac{e^{-\ep^2/4t}}{t^{1+s/2}} \, dt dV & = \frac{1}{\Gamma(-s/2)} \int_0^\infty \frac{e^{-\ep^2/4t}}{t^{1+s/2}} \lp  P_t u-u , \varphi \rp_{L^2} \, dt \\ &= \frac{1}{\Gamma(-s/2)} \int_0^\infty \frac{e^{-\ep^2/4t}}{t^{1+s/2}} \lp  P_t \varphi-\varphi , u \rp_{L^2} \, dt \\ &= \int_M \left( \frac{1}{\Gamma(-s/2)} \int_{0}^{\infty} (P_t  \varphi - \varphi ) \frac{e^{-\ep^2/4t}}{t^{1+s/2}} \, dt \right) u \, dV \,.
\end{align*}
Regarding the right-hand side, since $K_s^\ep(p,q)$ is symmetric 
\begin{equation*}
    \iint_{M\times M} (u(p)-u(q))\varphi(p) K_s^\ep(p,q) \, dV_q dV_p = \iint_{M\times M} (\varphi(p)-\varphi(q))u(p) K_s^\ep(p,q) \, dV_q dV_p \,.
\end{equation*}
Thus
\begin{equation*}
    \int_M \left( \frac{1}{\Gamma(-s/2)} \int_{0}^{\infty} (P_t  \varphi - \varphi ) \frac{e^{-\ep^2/4t}}{t^{1+s/2}} \, dt \right) u \, dV = \int_{M} \left( \int_M (\varphi(p)-\varphi(q))K_s^\ep(p,q) \, dV_q \right) u(p) dV_p \,,
\end{equation*}
and letting $\ep \to 0^+$ and using $(i)$ proves $(ii)$.
\end{proof}
\begin{remark}
    On a noncompact Riemannian manifold, the mass preservation property $\int_M H_M(p,q,t) \,dV_q=1$ could fail, leading to undesired phenomena such as the fractional Laplacian of a constant being different from zero. It is thus natural in the noncompact case to assume that $M$ is stochastically complete, i.e. that $\int_M H_M(p,q,t) \,dV_q=1$ for every $t>0$.
\end{remark}

\subsection{Properties of the kernel}\label{Ker prop section}

This section gives important estimates on the singular kernel $K_s(p,q)$. In order to precisely quantify the dependence of the constants in the estimates on the geometry of the ambient manifold, the notion of ``local flatness assumption'' will be very useful. Let us introduce it below.

\vsp 
Here, as in the rest of the paper, $\B_R(0)$ denotes the Euclidean ball of radius $R$ centered at $0$ of  $\R^n$, and $B_R(p) $ denotes the metric ball on $M$ of radius $R$ and center $p$. 

\begin{definition}[\textbf{Local flatness assumption}]\label{flatnessassup}
Let $(M^n,g)$ be an $n$-dimensional Riemannian manifold and $p\in M$. For $R>0$, we say that $ (M,g)$ satisfies the $\ell$-th order \textit{flatness assumption at scale $R$ around the point $p$, with parametrization $\varphi$}, abbreviated as ${\rm FA}_\ell(M,g, R,p,\varphi)$,  whenever there exists an open neighborhood $V$ of $p$ and a diffeomorphism
\begin{equation*}
\varphi:  \B_R(0) \to V, \quad \mbox{with }\varphi(0)=p \,,  
\end{equation*}
such that, letting $g_{ij}=g\left(\varphi_*\left(\frac{\partial}{\partial x^i}\right), \varphi_* \left(\frac{\partial}{\partial x^j} \right) \right)$  
be the representation of the metric $g$ in the coordinates $\varphi^{-1}$, we have
\begin{equation}\label{hsohoh1}
\big( 1-\tfrac{1}{100}\big) |v|^2 \le  g_{ij}(x)v^i v^j \le \big( 1+\tfrac{1}{100}\big)|v|^2 \s\s \forall \, v\in \R^n  \mbox{ and } \forall \, x\in \B_R(0) \,,
\end{equation}
and
\begin{equation}\label{hsohoh2}
R^{|\alpha|}\bigg|\frac{\partial^{|\alpha|} g_{ij} (x)}{\partial x^{\alpha}}\bigg|\le \tfrac{1}{100} \s\s\forall \alpha \mbox{ multi-index  with }1\le|\alpha|\le \ell, \mbox{ and }\forall   x\in \B_R(0).
\end{equation}
\end{definition}

\begin{remark}\label{fbsvdg}
Notice that for any smooth closed Riemannian manifold $(M,g)$, given $\ell \ge 0$, there exists $R_0>0$ for which ${\rm FA}_\ell(M,g,R_0,p,\varphi_p)$ is satisfied for all $p\in M$,  where $\varphi_p$ can be chosen to be the restriction of the exponential map\footnote{That is $\varphi_p = (\exp_p\circ \, i)|_{\B_{R_0}(0)}$ for any isometric identification of $i : \R^n \to TM_p$}(of $M$) at $p$ to the (normal) ball $\B_{R_0}(0)\subset T_p M \cong \R^n$.
\end{remark}

\begin{remark}
    The notion above of local flatness is used in our results to stress the fact that, once the local geometry of the manifold is controlled in the sense of Definition \ref{flatnessassup}, then our estimates are independent of
    $M$. Interestingly, this makes our estimates in the present article and in \cite{CFS} of local nature, even though the equation we deal with is nonlocal.
\end{remark}

\begin{remark}\label{flatscalingrmk}
The following useful scaling properties hold.
\begin{itemize}
\item[(a)] Given $M = (M,g)$ and $r>0$, we can consider the "rescaled manifold" $\widehat M = (M,  r^2g)$. When performing this rescaling, the new heat kernel $H_{\widehat {M}}$ satisfies 
\begin{equation*}
    H_{\widehat{M}}(p,q,t) = r^{-n} H_M(p,q,t/r^2)  \,.
\end{equation*}
As a consequence,  the "rescaled kernel" $\widehat  K_s$ defining the $s$-perimeter on $\widehat M$ satisfies 
\[
\widehat  K_s(p,q) = r^{-(n+s)} K_s(p,q).
\]

\item[(b)] Concerning the flatness assumption, it is easy to show that ${\rm FA}_\ell(M,g, R, p,\varphi) \Rightarrow {\rm FA}_\ell(M,g, R',p,\varphi)$ for all $R'<R$ and  ${\rm FA}_\ell(M,g, R, p,\varphi)\Leftrightarrow {\rm FA}_\ell(M, r^2g, R/r, p,\varphi(r\,\cdot\,))$.

\item[(c)] Similarly, if ${\rm FA}_\ell(M,g, R, p,\varphi)$ holds,  and $q\in \varphi(\B_R(0))$ is such that $\B_\varrho(\varphi^{-1}(q))\subset \B_R(0)$, then 
${\rm FA}_\ell(M,r^2g, \varrho/r, q,\varphi_{\varphi^{-1}(q),r})$ holds, where $\varphi_{x,\,\rho} := \varphi(x+\rho \, \cdot \,)$. 
\end{itemize}

\end{remark}

\vsp

In all the sections, we will use the (standard) multi-index notation for derivatives. A multi-index $\alpha=(\alpha_1,\alpha_2,\dots, \alpha_n)$ will be an $n$-tuple of nonnegative integers (in other words $\alpha \in \N^n$). We define 
\[
|\alpha| : = \alpha_1+ \alpha_2 +\cdots +\alpha_n.
\]
For a function  $f: \R^n \to \R$ is of class $C^\ell$ we shall use the notation
\[
\frac{\partial^{|\alpha|}}{\partial x^{\alpha}} f : = \frac{\partial^{\alpha_1 + \alpha_2 + \cdots \alpha_n}f}{(\partial x^1)^{\alpha_1} (\partial x^2)^{\alpha_2}\cdots  (\partial x^n)^{\alpha_n} }\, .
\]
For $\alpha=0$, we put $\frac{\partial^{|0|}}{\partial x^{0}} f :=f$.

\vsp
The next main theorem gives the precise behavior of the kernel around a point satisfying flatness assumptions, including an explicit approximation in coordinates.

\begin{theorem}\label{prop:kern1} 
Let $(M,g)$ be a Riemannian $n$-manifold, not necessarily closed, $s\in(0,2)$ and let $p\in M$. Assume ${\rm FA_\ell}(M,g,R,p, \varphi)$ holds and denote $K(x,y): = K_s(\varphi(x), \varphi(y))$. 

\vsp
Given $x\in \B_R(0)$, let
$A(x)$ denote the positive symmetric square root of the matrix $(g_{ij}(x))$ ---$g_{ij}$ being the metric in coordinates $\varphi^{-1}$--- and, for $x,z\in \B_{R/2}(0)$, define
\begin{equation*}
    k(x, z): = K(x,x+z) \quad \mbox{and}\quad  \widehat k(x,z) := k(x,z) -  \frac{\alpha_{n,s}}{|A(x)z|^{n+s}}.
\end{equation*}
Then
\begin{equation}\label{remaining0}
\big|\widehat k(x,z)\big| \le  R^{-1}\frac{C(n,s)}{|z|^{n+s -1}}\quad \mbox{for all } x,z \in \B_{R/4}(0) \,,
\end{equation}
and, for every multi-indices $\alpha,\beta$  with $|\alpha|+|\beta| \le \ell$,
we have
\begin{equation}\label{remaining}
  \bigg|\frac{\partial^{|\alpha|}}{\partial x^{\alpha}} \frac{\partial^{|\beta|}}{\partial z^{\beta}}  k(x,z)\bigg| \le \frac{C(n,s, \ell)}{|z|^{n+s+ |\beta|}}\quad \mbox{for all } x,z \in \B_{R/4}(0) .  
\end{equation}
The constants $C(n,s)$ and $C(n,s,l)$ stay bounded for $s$ away from $0$ and $2$.\\

Moreover, for all $x\in \B_{R/4}(0)$ and for all $q\in M\setminus \varphi(\B_R(0))$ we have
\begin{equation}\label{remaining2}
\bigg|\frac{\partial^{|\alpha|}}{\partial x^{\alpha}}  K_s(\varphi(x),q)\bigg| \le \frac{C(n,\ell)}{R^{n+s}} \,,
\end{equation}
and 
\begin{equation}\label{remaining3}
\int_{M\setminus \varphi(\B_R(0))}\bigg|\frac{\partial^{|\alpha|}}{\partial x^{\alpha}}  K_s(\varphi(x),q)\bigg| dV_q \le \frac{C(n,\ell)}{R^{s}}, 
\end{equation}
for every multi-index  $\alpha$ with $|\alpha|\le \ell$.

\end{theorem}

\subsubsection{Heat kernel estimates}
To prove Theorem \ref{prop:kern1} we will need several preliminary lemmas studying the properties of the heat kernel of $M$.

The first result compares locally the heat kernel $H_M(p,q,t)$ or the singular kernel $K_s(p,q)$ on $\R^n$ endowed with a metric $g$ with the standard ones on $\R^n$.
\begin{lemma}\label{globcomparability}
    Let $g$ be a smooth metric on $\R^n$ such that $\frac{|v|^2}{4}\le g_{ij}(x)v^iv^j\le 4|v|^2$ and $|Dg_{ij}(x)|\le 1$ for all $x , v \in \R^n$. Denote $M:= (\R^n,g)$ and let $K_s$ be defined by \eqref{kerneleq}. Then, there exist positive constants $c_i=c_i(n)$ for $1\le i \le 6$ such that 
    \begin{equation*}
    \frac{c_1}{t^{n/2}} e^{- \frac{|x-y|^2}{c_2 t} } \le H_M(x,y,t) \le\frac{c_3}{t^{n/2}}  e^{- \frac{|x-y|^2}{c_4 t} } ,
\end{equation*}
and 
\begin{equation*}
   c_5 \frac{\alpha_{n,s}}{|x-y|^{n+s}} \le  K_s(x,y) \le c_6 \frac{\alpha_{n,s}}{|x-y|^{n+s}}  \,, 
\end{equation*}
for all $(x,y,t ) \in \R^n \times \R^n \times [0, \infty)$.
\end{lemma}
\begin{proof}
    The two-sided estimates for the heat kernel $H_M$ follow directly from the classical parabolic estimates of Aronson \cite{Aronson}. The second inequality follows by integrating the first one, from the definition \eqref{kerneleq} of $K_s(x,y)$.
\end{proof}

The next lemma concerns the concentration of mass of the heat kernel.
\begin{lemma}\label{lemaux0}
Let $(M,g)$ be a Riemannian $n$-manifold,  $p\in M$, and assume ${\rm FA_0}(M,g,1,p, \varphi)$ holds. Then
\[
1-C\exp(-c/t)\le \int_{\varphi(\B_{1/2}(0))} H_M(p,q,t)dV_q  \le 1 , \s\mbox{for all }t>0,
\]
with $C,c>0$ depending only on $n$.
\end{lemma}
\begin{proof}
Put $H(x,y,t): = H_M(\varphi(x), \varphi(y),t)$. 
Let $g_{ij}\in C^0(\B_1(0))$ be the metric coefficients in the chart $\varphi^{-1}$, choose $\xi\in C^0_c(\B_1(0))$ such that $\chi_{\B_{3/4}(0)}\le \xi \le \chi_{\B_{1}(0)}$  and put $g'_{ij} : = g_{ij}\xi + \delta_{ij}(1-\xi)$. By assumption, we have 
$\big| \overline g'_{ij}(x) v^iv^j -|v|^2 \big|\le \tfrac {1}{100} |v|^2$ for all $x,v\in \R^n$. Morevoer, 
$g'_{ij}\equiv g_{ij}$ inside $\B_{3/4}(0)$. 
Consider the complete Riemannian manifold $M':=(\R^n, g')$ and let $H'(x,y,t)$ denote its associated heat kernel.
Then, by Lemma \ref{globcomparability} we have
\begin{equation}\label{whitohw736}
\frac{c_1}{t^{n/2}}  e^{-c_2|x-y|^2/t}\le 
 H'(x,y,t)   \le \frac{c_3}{t^{n/2}} e^{-c_4|x-y|^2/t}.
\end{equation}
Now, since $H'(0,x,t)$ is the heat kernel of  the stochastically complete manifold $M'$  we have 
\begin{equation}\label{intheatdensity1}
\int_M H'(0,x,t) \sqrt {|g'|(x)} dx =1\quad \mbox{for all } t>0.
\end{equation}
On the other hand, for every fixed $\tau>0$ set $h(\tau) :=  \frac{c_3}{\tau^{n/2}} e^{-c_4(1/4)^2/\tau}$. Using \eqref{whitohw736} we have 
 that $u(x,t) = \big(H'(0,x,t)- h(\tau)\big)^+$, where $(\,\cdot\,)^+$ denotes the positive part, is a subsolution of  
 \[
 \begin{cases}
u_t \le \frac{1}{\sqrt{|g'|}} \frac{\partial}{\partial x^i} \bigg( \sqrt{|g'|} (g')^{ij} \frac{\partial}{\partial x^j} u \bigg) 
\quad &\mbox{in } \B_{1/4}(0)\times(0,\tau) \,,\\
u = 0  &\mbox{in } \partial\B_{1/4}(0)\times(0,\tau) \,.
\end{cases}
 \]
 Since $g'\equiv g$ in $\B_{1/4}(0)$, it easily follows (using that both $H(0,x,t)$ and $H'(0,x,t)$ have as initial condition a Dirac delta with respect to the same volume form $\sqrt{|g|} dx$) that $u\le H(0,x,t)$ for all $t\in (0,\tau)$.
This gives, for all $t\in (0, \tau)$
\[
\int_{\B_{1/4}} \big(H'(0,x,t)- h(\tau)\big)^+ \sqrt{|g|} dx =\int_{\B_{1/4}} u(x,t) \sqrt{|g|} dx\le \int_{\B_{1/4}} H(0,x,t)\sqrt{|g|} dx
\]
On the other hand, using \eqref{intheatdensity1} and \eqref{whitohw736} we obtain that for all $t\in (0, \tau)$
\[
\begin{split}
1 - C\exp(-c/\tau) &\le 1 - \int_{\R^n \setminus \B_{3/4}(0)} \frac{c_3}{t^{n/2}} e^{-c_4|x|^2/t} \big(1+\tfrac{1}{100}\big)^{n/2} dx 
\le 1 - \int_{\R^n \setminus \B_{3/4}(0)} H'(0,x,t) \sqrt {|g'|} dx\\
&= \int_{\B_{3/4}(0)}  H'(0,x,t)\sqrt{|g|} dx.
\end{split}
\]
Finally, since also $h(\tau)\le  C\exp(-c/\tau)$ (notice that we can ``absorbe'' $\tau^{-n/2}$ in $Ce^{-c/\tau}$ chosing $c>0$ slightly smaller and a larger $C$), we obtain the desired estimate 
\[
1-C\exp(-c/\tau) \le \int_{\B_{1/4}(0)} H(0,x,t)\sqrt{|g|} dx\le \int_{\varphi(\B_{1/2}(0))} H(p,q,t)dV_q,\quad \forall t\in (0,\tau) \,, \]
and for all $\tau>0$. The bound by above by $1$ of the same quantity follows immediately using that $H$ is a heat kernel, i.e. nonnegative and with total mass bounded by $1$.
\end{proof}

\begin{lemma}\label{lemaux1}
Under the same assumptions as in Theorem \ref{prop:kern1}, for all $q\in M\setminus \varphi(\B_1(0))$ we have
\begin{equation}\label{estder}
    \bigg|\frac{\partial^{|\alpha|}}{\partial x^{\alpha}}  H_M(\varphi(x),q,t)\bigg| \le C \exp(-c/t), \s\mbox{for } (x,t)\in \B_{1/2}(0) \times [0,\infty)
\end{equation}
and for every multi-index $\alpha$ with  $|\alpha|\le \ell$, 
with $C,c>0$ depending only on $n$ and $\ell$.
\end{lemma}

\begin{proof}
Notice that $u(x,t) : = H_M(\varphi(x),q,t)$ satisfies $u_t =Lu$, in $\B_1(0)\times [0,\infty)$ and $u\equiv 0$ at $t=0$, where 
\begin{equation}\label{Lu17486}
     Lu =  \frac{1}{\sqrt{|g|}} \frac{\partial}{\partial x^i} \bigg( \sqrt{|g|} g^{ij} \frac{\partial u }{\partial x^j} \bigg)
\end{equation}
is the Laplace-Beltrami operator with metric $g$.

\vsp
Let us show that 
\begin{equation}\label{wnthiowhoihw2}
    |u|\le  C \exp(-c/t)\quad \mbox{for } (x,t)\in \B_{3/4}(0) \times [0,\infty),
\end{equation}
with $C,c>0$ dimensional constants. 
This follows from the following standard probabilistic consideration. Fix $x_\circ \in \varphi(\B_{3/4}(0))$. By continuity of sample paths, the probability that a Brownian motion started at $q \in M\setminus \varphi(\B_1(0))$ hits $\varphi(\B_{\delta}(x_\circ))$ ($0<\delta\ll1$) within time $\le t$ is less that the supremum among $q' \in \varphi(\partial \B_{8/9}) $ of the probability that a Brownian motion started at a point $q'$ hits $\varphi(\B_{\delta}(x_\circ))$ within time $\le t$. 
This gives
\begin{equation}\label{wnthiowhoihw22}
    u(x_\circ,t) \le \sup_{q' \in \varphi(\partial \B_{8/9})} H_M(\varphi(x_\circ), q',t).
\end{equation}
Let us now use \eqref{wnthiowhoihw22}, Lemma \ref{lemaux0}, and the parabolic Harnack inequality as follows to show \eqref{wnthiowhoihw2}. 

\vsp
For fixed $q' \in \varphi(\partial \B_{8/9}) $ set $v(x,t) := H_M(\varphi(x), q',t) $ and consider the rescaled $\widetilde v(x,t):= v(x_\circ+rx, t_\circ+ r^2t)$ for $r \in (0, 1/10 )$. Then $ \widetilde v\ge0$ satisfies a (uniformly) parabolic equation in $\B_1(0)\times(0,1)$ with smooth coefficients (that only improve as $r$ gets smaller). Thus, by the Harnack inequality for every $x \in \B_{1/2}(0)$ and $t \in (1/4, 1/2)$ we have $ \widetilde v(x,t) \le C \inf_{\B_{1/2}(0)} \widetilde v(\, \cdot \,, 1) \le C \widetilde v(y,1)$ for all $y \in \B_{1/2}(0)$. Integrating
    \begin{align*}
        \widetilde v(0,t)  \le C \int_{\B_{1/2}(0)} \widetilde v(y,1) dy = C \int_{\B_{1/2}(0)} v(x_\circ+ry,t_\circ+r^2) dy = C r^{-n} \int_{\B_{r/2}(x_\circ)} v(z,t_\circ+r^2) dz \,,
    \end{align*}
    for some $C=C(n)>0$. Thus, for all $t \in (t_\circ+r^2/4, t_\circ+r^2/2)$
    \begin{equation*}
        v(x_\circ,t) \le C r^{-n} \int_{\B_{r/2}(x_\circ)} v(z,t_\circ + r^2) dz \,.
    \end{equation*}
    But $\varphi(\B_{r/2}(x_\circ)) \subset M\setminus B_{1/10}(q') $ for every $q' \in \varphi(\partial \B_{8/9}(0))$. Then by Lemma \ref{lemaux0} we get
   \begin{equation*}
        v(x_\circ,t) \le C r^{-n} \int_{ M\setminus \varphi(\B_{1/10}(q')) } H_M(z,q',t_\circ+r^2) dz \le Cr^{-n} \exp\left(-\frac{c}{t_\circ+r^2}\right)  ,
    \end{equation*}
where $C,c>0$ depend only on $n$. Now, for small times $t_\circ \le 1/100 $ choosing $r^2 = t_\circ $ (together with the probabilistic argument above) gives the result, since one can absorb the term $r^{-n}=t_\circ^{-n/2}$ in the exponential up to slightly decreasing the value of $c$. For non-small times $t_\circ > 1/100$, one can just take $r=1/10$ and obtain the upper bound by a constant as desired. This concludes the proof of \eqref{wnthiowhoihw2}.

\vsp 
Now, similarly to above, for all $r\in (0,1/4)$ and $(x_\circ,t_\circ)\in \B_{1/2}(0)\times (0,\infty)$ the rescaled function $\overline u(x,t) = u(x_\circ +rx,t_\circ + r^2t)$ satisfies a (uniformly) parabolic equation with smooth coefficients (since the bounds on every $C^k$ norm of the coefficients only improve as $r$ gets smaller) and, from \eqref{wnthiowhoihw2}, we have $|\overline u|\le C \exp(-c/t)$ in $\B_{1}\times(0,1)$. Hence standard parabolic Schauder estimates give
\[
\bigg|\frac{\partial^{|\alpha|}}{\partial x^{\alpha}} \, \overline u\bigg| \le C \exp(-c/t), \s\mbox{for } (x,t)\in \B_{1/2}(0) \times [1/2,1) \,,
\]
for every multi-index $\alpha$ with  $|\alpha|\le \ell$, 
with $C>0$ depending only on $n$ and $\ell$ and $c>0$ as above.

\vsp
After scaling back the estimate above we obtain, for all $r\in (0,1/4]$
\[
\bigg|\frac{\partial^{|\alpha|}}{\partial x^{\alpha}}  u(x_\circ, t_\circ+ t )\bigg| \le Cr^{-|\alpha|} \exp(-c/r^2), \s\mbox{for } (x_\circ,t )\in \B_{1/2}(0) \times [r^2/2,r^2) \,.
\]

Then, for ``non-small'' times  $t_\circ\ge 1/16$ we notice that \eqref{estder} follows taking $r=1/4$. On the other hand, for small times  $t_\circ\in (0,1/16)$ we obtain \eqref{estder}  taking $r^2=t_\circ$, bounding $r^{-|\alpha|}$ by $t_\circ^{-\ell/2}$, and absorbing (chosing $c>0$ smaller and $C$ larger) this negative power of $t_\circ$ in the exponential.
\end{proof}

\begin{lemma}\label{lemaux12}
Under the same assumptions as in Theorem \ref{prop:kern1},
we have
\begin{equation}\label{estder2}\int_{M\setminus\varphi(\B_1(0))}
    \bigg|\frac{\partial^{|\alpha|}}{\partial x^{\alpha}}  H_M(\varphi(x),q,t)\bigg| dV_q \le C \exp(-c/t), \s\mbox{for } (x,t)\in \B_{1/2}(0) \times [0,\infty)
\end{equation}
and for every multi-index $\alpha$ with  $|\alpha|\le \ell$, 
with $C,c>0$ depending only on $n$ and $\ell$.
\end{lemma}

\begin{proof}
It is similar to the proof of Lemma \ref{lemaux1}. 
Let $\sigma: M\setminus\varphi(\B_1(0))\to \{+1,-1\}$ be any measurable function to be chosen. Consider 
\[u(x,t) : = \int_{M\setminus\varphi(\B_1(0))}H_M(\varphi(x),q,t) \sigma(q)dV_q,\] 

By Lemma \ref{lemaux0} --- since $H_M \ge 0$ and $\int_M H_M(p,q,t)dV_q\le 1$ --- we obtain
\begin{equation}\label{wnthiowhoihw}
    |u(x,t)|\le \int_{M\setminus\varphi(\B_{1/4}(0))} H(\varphi(x),q,t) dV_q \le C\exp(-c/t),\qquad \forall (x,t)\in \B_{3/4}(0)\times [0,\infty).
\end{equation}
Notice that in this estimate, $C$ and $c$ are positive dimensional constants (and in particular, they do not depend on the choice of $\sigma$). Also, by the superposition principle $u$ satisfies $u_t =Lu$, in $\B_1(0)\times [0,\infty)$ and $u\equiv 0$ at $t=0$, where $L$ is as in \eqref{Lu17486}.

\vsp
Now proceeding exactly as in the proof of Lemma \ref{lemaux1} we obtain that 
\[
\bigg|\frac{\partial^{|\alpha|}}{\partial x^{\alpha}} u\bigg| \le C \exp(-c/t), \s\mbox{for } (x,t)\in \B_{1/2}(0) \times [0,\infty)
\]
for $|\alpha|\le \ell$.
Now, for any given  $\alpha$, $x$, and $t$, we can choose $\sigma: M\setminus\varphi(\B_1(0))\to \{+1,-1\}$ so that 
\[
\frac{\partial^{|\alpha|}}{\partial x^{\alpha}}   u(x,t) = 
\int_{M\setminus\varphi(\B_1(0))}\frac{\partial^{|\alpha|}}{\partial x^{\alpha}}H_M(\varphi(x),q,t) \sigma(q)dV_q
=\int_{M\setminus\varphi(\B_1(0))}\bigg|\frac{\partial^{|\alpha|}}{\partial x^{\alpha}}H_M(\varphi(x),q,t)\bigg| dV_q,
\]
and we are done.
\end{proof}

\begin{lemma}[\textbf{Localization principle}]\label{lemaux2}
Let $(M,g)$ and  $(M',g')$ be  two Riemannian $n$-manifolds. Assume that both $M$  and $M'$ satisfy the flatness assumptions ${\rm FA_\ell}(M,g,1,p, \varphi)$ and ${\rm FA_\ell}(M',g,1,p', \varphi')$ respectively, and suppose that $g_{ij}\equiv g'_{ij}$ in $\B_1(0)$ in the coordinates induced by $\varphi^{-1}$ and $(\varphi')^{-1}$. 

Then, letting $H(x,y,t): = H_M(\varphi(x),\varphi(y),t)$ and $H'(x,y,t): = H_{M'} (\varphi'(x),\varphi'(y),t)$, we have that the difference $(H-H')(x,y,t)$ is of class $C^\ell$ in  $\B_{1/2}(0) \times \B_{1/2}(0)\times [0,\infty)$ and 
\[
\bigg|\frac{\partial^{|\alpha|}}{\partial x^{\alpha}} \frac{\partial^{|\beta|}}{\partial y^{\beta}} (H-H')(x,y,t)\bigg| \le C \exp(-c/t)\quad \mbox{for } (x,y,t)\in \B_{1/2}(0) \times \B_{1/2}(0)\times [0,\infty),
\]
whenever $\alpha$ and $\beta$ are multi-indices satisfying $|\alpha|+|\beta| \le \ell$, 
with $C,c>0$ depending only on $n$ and $\ell$.
\end{lemma}

\begin{proof}
Let us show that  
\begin{equation}\label{wntiowoi1}
|H-H'|\le  C \exp(-c/t)\quad \mbox{for } (x,y,t)\in \B_{3/4}(0) \times \B_{3/4}(0)\times [0,\infty),
\end{equation}
with $C,c>0$ dimensional constants. 

Indeed, fix $x_\circ \in \B_{3/4}(0)$ and let us show first that we have 
\begin{equation}\label{wntiowoi2}
\big|(H-H')(x_\circ,y,t)\big| \le C\exp(-c/t) \quad \mbox{for all }  y\in \B_{3/4}(0)\setminus \B_{1/8}(x_\circ)
\end{equation}
Indeed, the $L^\infty$ estimate of Lemma \ref{lemaux1} ---appropriately  rescaled to have $\varphi(B_{1/8}(x_\circ))$  instead of $\varphi(\B_1(0))$--- gives
\begin{equation}
H_M(\varphi(x_\circ), \varphi(y),t) \le C\exp(-c/t)\quad \mbox{for all }  y\in \B_{3/4}(0)\setminus \B_{1/8}(x_\circ),
\end{equation}
and the same estimate with $H_{M}$ replaced by $H_{M'}$. Hence \eqref{wntiowoi1} follows using $|H-H'|\le H+H'$.

Now observing that for all $x_\circ$ as above  $u(y,t) :=(H-H')(x_\circ, y,t)$ solves the heat equation
$
(\partial_t -L_y) u=0
$
with zero intial condition, \eqref{wntiowoi1} easily follows from the maximum principle.

Finally, the estimate for the higher derivatives follows from standard parabolic estimates, noticing that $u(x,y,t) := (H-H')(x,y,t)$ solves 
\[
\partial_t u = \frac 1 2  (L_{x,y}  u)\]
where
\[
L_{x,y} u =  \frac{1}{\sqrt{|g(x)|}} \frac{\partial}{\partial x^i} \bigg( \sqrt{|g(x)|} g^{ij}(x) \frac{\partial}{\partial x^j} u \bigg)+\frac{1}{\sqrt{|g(y)|}} \frac{\partial}{\partial y^i} \bigg( \sqrt{|g(y)|} g^{ij}(y) \frac{\partial}{\partial y^j} u\bigg).
\]
is the sum of the Laplace-Beltrami operators with respect to the variables $x$ and $y$ (or, equivalently, the Laplace-Beltrami operator with respect to the product metric in $\B_1(0)\times \B_1(0)$). 
\end{proof}

\begin{proposition}\label{lemaux3}
Assume that $M  = (\R^n, g)$ with $g = (g_{ij}(x))$ satisfying
\begin{equation}\label{asumpg}
        \tfrac 12 {\rm id} \le (g_{ij}) \le 2 {\rm id} \qquad \mbox{and}\qquad \bigg|\frac{\partial^{|\alpha|}}{\partial x^{\alpha}} g_{ij}\bigg|  \le 1\quad \mbox{for all } |\alpha|\le \ell,
\end{equation}
for some $\ell\ge 1$. Let $H(x,y,t)$ be the heat kernel of $M$. 

For $x\in \R^{n}$ let $A(x)$ denote the (unique) positive definite symmetric square root of the matrix $g(x) =  (g_{ij}(x))$, and define $h(z,x, t)$ by the identity 
\[
H(x,y,t) = \frac{1}{t^{n/2}} \,h\bigg(\frac{A(x)(y-x)}{\sqrt t},x,  t\bigg).
\]

Define also:
\[
h_\circ (x, z, t)  = h_\circ (z) :=  \frac{1}{(4\pi)^{n/2}} e^{-|z|^2/4}\qquad \mbox{ and }\qquad \widehat h  := h-h_\circ.
\]

Then, there are positive dimensional $C$ and $c$ such that 
\[  |\widehat h| \le  C \min(1,\sqrt t) e^{-c|z|^2} \quad \mbox{ for all  $(x,z,t)\in \R^n\times\R^n \times (0,\infty)$}.
\]

Moreover, we have
\[
\bigg|\frac{\partial^{|\alpha|}}{\partial x^{\alpha}} \frac{\partial^{|\beta|}}{\partial z^{\beta}}h\bigg|  \le C e^{-c |z|^2}  \quad \mbox{ for all  $(x,z,t)\in 
\R^n
\times\R^n \times (0,1)$ and $\alpha, \beta$ with $|\alpha|+|\beta|\le \ell$},
\]
for positive constants $C$ and $c$ depending only on $n$ and $\ell$.

\end{proposition}

\begin{proof}[Proof of Proposition \ref{lemaux3}]

Notice first that since $H(x,y,t) = H(y,x,t)$ we have
\[ H(x,y,t)= 
t^{-n/2} h\bigg(\frac{A(x)(y-x)}{\sqrt t},x,  t\bigg)= t^{-n/2} h\bigg(\frac{A(y)(x-y)}{\sqrt t},y,  t\bigg).\]

Let $L_xf := \frac{1}{\sqrt{ |g|(x)}} \tfrac{\partial}{\partial x^i}\bigg(\sqrt{|g|(x)}g(x)^{ij} \tfrac{\partial}{\partial x^j}f\bigg)$ denote the Laplace-Beltrami operator (with respect to $x$). Direct computation shows:
\[
\begin{split}
LH &=  t^{- \frac n2 - 1} \bigg(\sqrt t \bigg(\frac{\partial_i \big(\sqrt{|g|}g^{ij}\big)}{\sqrt{|g|}} \bigg)(x)  A^l_j(y) \frac{\partial}{\partial  z^l }  h(*)  + g^{ij}(x) \big(A_i^k A_j^l\big)(y) \frac{\partial^2}{\partial  z^k\partial z^l }  h(*)\bigg), \\
\partial_t H &=  t^{- \frac n2 - 1} \bigg( -\frac n 2 h(*) - \frac 1 2 \frac{\partial}{\partial z^l }  h(*) \frac{(A(y)(x-y))^l}{\sqrt t}  + t\partial_t h(*)\bigg),
\end{split}
\]
where 
\[(*) \mbox{ means evaluated at } \bigg( \frac{A(y)(x-y)}{\sqrt t} , y, t\bigg).\]

This leads to the equation for $h = h(z,y,t)$, where we denote $\partial_i : = \frac{\partial}{\partial z^i}$ and $\partial_{ij} : = \frac{\partial^2}{\partial z^i\partial z^j}$, 
\[
t\partial_t h =  \overline L h : = a^{ij}(z,y,t) \partial_{ij} h + \big(\sqrt t b^i(z,y,t) + \tfrac{z^i}{2}\big) \partial_i h + \frac{n}{2} h,
\]
where 
\[
 a^{ij}(z,y,t) := g^{kl}\Big(y + \sqrt t z\Big) \big(A_k^i A_l^j\big)(y) 
\]
and 
\[
 b^i(z,y,t) :=   \bigg(\frac{\partial_k \big(\sqrt{|g|}g^{kl}\big)}{\sqrt{|g|}} \bigg)\Big(y +\sqrt t z \Big)A^i_l(y);
\]
with initial condition: 
\[h(z,y,0^+) = h_\circ(z) = (4\pi)^{-\frac n 2} e^{-|z|^2/4}.\] 
(Notice that we defined $h$ so that its initial condition is independent of $y$.) 

We emphasize that, by the assumption \eqref{asumpg}, this equation is uniformly elliptic, and the derivatives of $a^{ij}$, $b^i$ up to order $\ell$ in the variables $z$ and $y$ are uniformly bounded for times $t\in (0,T_\circ)$ by constants depending only on the constants  $n$ and $T_\circ$. 

Let us now compute an equation for $\widehat h= h - h_\circ$. Since  
\[\delta^{ij} \partial_{ij} h_\circ  +  \tfrac{z^i}{2} \partial_i h + \frac{n}{2}  h_\circ =0, \]
we obtain
\[
\begin{split}
t\partial_t \widehat h - \overline L \widehat h &= \overline L h_\circ = (a^{ij}-\delta^{ij})\partial_{ij} h_\circ   + \sqrt t b^i \partial_i h_\circ
\\
& = \big((a^{ij}-\delta^{ij}) (z_i z_j -\delta_{ij}) -  \sqrt t b^i \tfrac{z^j}{2}\delta_{ij} \big) h_\circ\, .
\end{split} 
\]
We emphasize that $\widehat h$ satisfies the initial condition
\[
\widehat h (z,y, 0) \equiv 0
\]

Notice that  (since by definition $A(y)$ is a square root of $g(y)$) we have, for all $y$
\[
g^{kl}(y) \big(A_k^i A_l^j\big)(y) = \delta^{ij}
\]
and hence, since $g^{kl}$ is smooth,
\[
|a^{ij}(z,y,t) -\delta^{ij}| \le  C \sqrt t
\]

Hence, we have 
\begin{equation}\label{wtjhwiohwioh}
|t\partial_t \widehat h - \overline L \widehat h| \le C (1+ |z|^2)\sqrt t\, h_\circ
\end{equation}

Let us now find some barrier allowing us to control $\widehat h$. 
We can use as barrier
\[
b(z,t) : =  \sqrt t e^{ -(1/4 - \kappa )|z|^2 }
\]

Direct computation shows that, for $\sqrt t< \theta \kappa$ (so that $a^{ij}\delta_{ij}\ge  n -C\theta \kappa  $ and $|\delta^{ij}-a^{ij}|z^k z^l\delta _{ik}\delta_{jl}\le C\theta \kappa |z|^2$)
\[
\begin{split}
t\partial_t b -\overline L b
&=   \bigg( \frac {1}{2}  -  4\big(\tfrac 1 4  - \kappa \big)^2  \, (a^{ij})z^k z^l\delta _{ik}\delta_{jl} + 2(\tfrac 1 4  - \kappa \big) \,a^{ij}\delta_{ij}   + (\sqrt t b^i + \frac{z^i}{2}) 2(\tfrac 1 4  - \kappa \big) z^j\delta_{ij}  -\frac{n}{2}   \bigg) b  \\
&\ge \bigg( \frac {1}{2}   + \big(\tfrac 1 4  - \kappa \big)4\kappa |z|^2  - C\theta\kappa |z|^2 - C\kappa - C\theta \kappa |z|   \bigg) b 
\\
& \ge \big(\tfrac {1}{4}  + \tfrac{\kappa}{2}|z|^2\big) b \ge 0,
\end{split} 
\]
provided we chose $\theta>0$ and $\kappa>0$ sufficiently small.

Since clearly $b \ge \sqrt t \, h_\circ$ we obtain that $Cb$ is a supersolution of 
\eqref{wtjhwiohwioh} for $\sqrt t< \theta \kappa$.
This shows that $|\widehat h|\le Cb$ for all $t$ small enough.

\vsp
Notice that the estimate $|\widehat h|\le Cb$ (fixing $\kappa>0$ and $\theta>0$ small dimensional) shows, in particular, that 
\begin{equation}\label{whiorhiowh}
    |\widehat h(z,y,t)|\le C\sqrt{t}\exp(-c|z|^2)
\end{equation}
holds with  $c>0$ dimensional for all ``small'' times $t\in (0,\theta^2\kappa^2)$.
On the other hand, for ``non-small'' times $t\ge\theta^2\kappa^2$, the standard heat kernel estimate \eqref{whitohw736}  for $H$ (which holds with $c_i$ dimensional) immediately yields \eqref{whiorhiowh} with $\sqrt{t}$ replaced by $1$.

\vsp
In order to bound the derivatives of $h$ with respect to $z$ we notice we notice that, in logarithmic time $\tau =\log t$,  the function $h(z,y, e^\tau)$ satisfies, for $y$ fixed, a standard parabolic equation with smooth coefficients in the domain $\R^n \times (-\infty, 0)$. Then, thanks to \eqref{whiorhiowh}, applying standard parabolic estimates in parabolic cylinders $\{|x-x_\circ|<2, |\tau-\tau_\circ|<2 \}$ we easily obtain the claimed bounds for all partial derivatives of $h$ with respect to $z$.

\vsp
In order to show the regularity in $y$, one can then differentiate the equation with respect to $y$ as many times as needed (the coefficients depend in a very smooth way also in $y$) and notice that the initial condition will be zero (since $h_\circ$ is independent of $y$). By standard parabolic regularity arguments (e.g., using a Duhamel-type formula to represent the solutions), we obtain the estimates.
\end{proof}

\subsubsection{Estimates for the singular kernel \texorpdfstring{$K_s(p,q)$}{} }

    As a first consequence of Lemma \ref{lemaux2} we have that the following ``local version'' of Lemma \ref{globcomparability} above also holds. 

\begin{lemma}\label{loccomparability}
    Let $s_0 \in (0,2)$ and $s \in (s_0,2)$. Let $(M,g)$ be a Riemannian $n$-manifold and $ p \in M$. Assume that ${\rm FA}_1(M,g,p,1,\varphi)$ holds. Then 
\begin{equation*}
   c_7 \frac{\alpha_{n,s}}{|x-y|^{n+s}} \le  K_s(\varphi(x),\varphi(y)) \le  c_8 \frac{\alpha_{n,s}}{|x-y|^{n+s}} \,, 
\end{equation*}
for all $x,y \in \B_{1/2}(0)$, where $c_7,c_8>0$ depends on $n$ and $s_0$.
\end{lemma}
\begin{proof}
   Take $\eta \in C^\infty_c(\B_{1}(0))$ with $\chi_{\B_{1/2}(0)} \le \eta \le \chi_{\B_1(0)}$ and let $g'_{ij}:=g_{ij} \eta +(1-\eta)\delta_{ij}$. This is a metric on $\R^n$ with $g'_{ij}=g_{ij}$ in $\B_{1/2}(0)$. Denote by $K_s,K_s'$ and $H,H'$ the singular kernels and heat kernels of $(M,g)$ and $(\R^n, g')$ respectively. Then, by Lemma \ref{lemaux2} applied to the manifolds $(M,g)$ and $(\R^n, g')$ we have, for $x,y \in \B_{1/4}(0)$:
\begin{align*}
    \big| K_s(\varphi(x), \varphi(y))-K_s'(x,y) \big| & \le \frac{s/2}{\Gamma(1-s/2)} \int_{0}^\infty \big| H(\varphi(x), \varphi(y),t)-H'(x,y,t) \big| \frac{dt}{t^{1+s/2}} \\ & \le \frac{C s}{\Gamma(1-s/2)} \int_{0}^{\infty} e^{-c/t} \frac{dt}{t^{1+s/2}} \le C(2-s) \,,
\end{align*}
for some dimensional $C=C(n)$. Then, the result follows directly by Lemma \ref{globcomparability} (and the explicit formula \eqref{alphadef} for $\alpha_{n,s}$) for $x,y \in \B_{1/4}(0)$, and the conclusion also holds for $x,y \in \B_{1/2}(0)$ by a standard covering argument.
\end{proof}

Now, we have all the ingredients to give the proof of Theorem \ref{prop:kern1}.

\begin{proof}[Proof of Theorem \ref{prop:kern1}]

Note that the statement is scaling invariant. Hence, with no loss of generality, we can (and do) assume that $R=1$. Moreover, it suffices to consider the case $M=(\R^n,g)$, $p =0$,  $\varphi = {\rm id}$, and $g_{ij}$ satisfying the assumptions of Proposition \ref{lemaux3}:

\vsp
Indeed, similarly to the proof of Corollary \ref{loccomparability}, in the general case we can fix a radially nonincreasing cutoff function $\eta\in C^\infty_c(\B_1)$ such that $\eta\equiv 1$ in $\B_{7/8}$ and consider the ``extended'' metric $g'_{ij} : = g_{ij}\eta + \delta_{ij}(1-\eta)$. 
Observe that $(M,g)$ and $(\R^n,g')$ the assumptions of Lemma \ref{lemaux2} with $M'=\R^n$ and $\varphi'= {\rm id}$. Let $H(x,y,t)$ and $H'(x,y,t)$ be defined as in Lemma \ref{lemaux2}.

\vsp
Recall that, by definition, for all $x,y\in \B_1(0)$ 
\begin{equation}\label{wjhtihwi9h1}
    K(x,y) = K_s(\varphi(x),\varphi(y)) = c_s\int_0^\infty H_M(\varphi(x),\varphi(y),t) \frac{dt}{t^{1+s/2}} = c_s \int_0^\infty H(x,y,t) \frac{dt}{t^{1+s/2}},
\end{equation}
where $c_s =\frac{s/2}{\Gamma(1-s/2)}$.
Let likewise
\[
K'(x,y) =  c_s \int_0^\infty H'(x,y,t) \frac{dt}{t^{1+s/2}}.
\]
Now, thanks to Lemma \ref{lemaux2} we obtain, for all $x,y\in \B_{1/2}$:
\[
\bigg|\frac{\partial^{|\alpha|}}{\partial x^{\alpha}} \frac{\partial^{|\beta|}}{\partial y^{\beta}} (K-K')(x,y,t)\bigg|\le c_s\int_0^\infty  \bigg|\frac{\partial^{|\alpha|}}{\partial x^{\alpha}} \frac{\partial^{|\beta|}}{\partial y^{\beta}} (H-H')(x,y,t)\bigg|\frac{dt}{t^{1+s/2}} \le Cs \int_0^\infty e^{-c/t} \frac{dt}{t^{1+s/2}} \le C.
\]

So, as claimed, we are left to prove the estimate for the $M=(\R^n,g)$, $p =0$,  $\varphi = {\rm id}$, and $g_{ij}$ satisfying the assumptions of Proposition \ref{lemaux3}. 

Recalling \eqref{wjhtihwi9h1}, notice that
\begin{equation}\label{whtuiewgoug}
k(x,z) = K(x, x+z) = c_s \int_0^\infty H(x,x+z,t) \frac{dt}{t^{1+s/2}} = c_s \int_0^\infty \,h\bigg(\frac{A(x)z}{\sqrt t},x,  t\bigg)\frac{dt}{t^{n/2 +1+s/2}}.
\end{equation}
Also, recalling  that $h_\circ(z) := (4\pi)^{-n/2} e^{-|z|^2/4}$, we have 
\[
\begin{split}
\widehat k(x,z) &=k(x,z)- \frac{\alpha_{n,s}}{|A(x)z|^{n+s}} =  c_s \int_0^\infty \bigg(h\bigg(\frac{A(x)z}{\sqrt t},x,  t\bigg)- h_\circ\bigg(\frac{A(x)z}{\sqrt t}\bigg) \bigg)\frac{dt}{t^{n/2 +1+s/2}} \\ &= c_s 
\int_0^\infty \widehat h\bigg(\frac{A(x)z}{\sqrt t},x ,t\bigg) \frac{dt}{t^{n/2 +1+s/2}} ,
\end{split}
\]

Therefore using the heat kernel estimates from Proposition \ref{lemaux3} (and noticing $|A(x)z|\ge \tfrac{1}{\sqrt 2} |z|$ for all $x, z$ by assumption) we obtain 
\[
\begin{split}
|\widehat k(x,z)| \le c_s
\int_0 \bigg|\widehat h\bigg(\frac{A(x)z}{\sqrt t},x ,t\bigg) \bigg| \frac{dt}{t^{n/2 +1+s/2}} \le Cs \int_0^\infty \sqrt t \exp(-c |z|/\sqrt t) \frac{dt}{t^{n/2 +1+s/2}}  = C|z|^{1-n-s}.
\end{split}
\]
This proves \eqref{remaining0}. Similarly, the estimates \eqref{remaining} follow differentiating \eqref{whtuiewgoug} and using the corresponding estimates for derivatives of the heat kernel from Proposition \ref{lemaux3}.

Finally, \eqref{remaining2} and \eqref{remaining3} follow analogously integrating the heat kernel estimates in Lemmas \ref{lemaux1} and \ref{lemaux12}, respectively.
\end{proof}

The next property concerns the behavior of the kernel when the two points $p$ and $q$ are separated from each other.

\begin{proposition}\label{prop:kern2}
Let $(M,g)$ be a  Riemannian $n$-manifold and $s\in (0,2)$. Assume that for some $p,q\in M$  both ${\rm FA_\ell}(M, g, 1, p, \varphi_p)$  and ${\rm FA_\ell}(M, g, 1, q, \varphi_q)$  hold, and suppose that $\varphi_p(\B_1(0))\cap \varphi_q(\B_1(0))= \varnothing$.  Put $K_{pq}(x,y): = K_s(\varphi_p(x), \varphi_q(y))$. Then
\[
\bigg|\frac{\partial^{|\alpha|}}{\partial x^{\alpha}} \frac{\partial^{|\beta|}}{\partial y^{\beta}}   K_{pq}(x,y)\bigg| \le C(n,\ell)\quad 
\mbox{for all }|x|<\tfrac 1 2 \mbox{ and }|y|< \tfrac 1 2,
\]
whenever $|\alpha|+|\beta| \le \ell$.
\end{proposition}
\begin{proof}
Let $H_*(x,y,t) : =H_M(\varphi_p(x), \varphi_q(y),t)$.
It follows from Lemma \ref{lemaux1} that 
\[
\left|\frac{\partial^{|\alpha|}}{\partial x^{\alpha}} H_*(x,y,t) \right|\le C \exp(-c/t)
\]
for all $|x|<\tfrac 3 4$ and $|y|< \tfrac 3 4$, 
where $C$ and $c$ depend only on $n$, and $|\alpha|$.

We now use that (by the symmetry of the heat kernel in $p$ and $q$), for each $x\in \B_{1/2}$ fixed, the function $u(y,t): = \frac{\partial^{|\alpha|}}{\partial x^{\alpha}} H_*(x,y,t)$  is solution of the heat equation $u_t = L u$, in the ball $|y|<1$, where $L$ denotes the Laplace-Beltrami (with respect to $y$, in local coordinates). Since $|u|\le C\exp(-c/t)$ in $\B_{3/4}\times (0,\infty)$,  reasoning exactly as in the proof of Lemma \ref{lemaux1} (only that now the spatial variables are $y$ instead of $x$) we obtain 
\[
\bigg|\frac{\partial^{|\beta|}}{\partial y^{\beta}} u(y,t)\bigg|\le C \exp(-c/t),
\]
for some new positive constants $C$ and $c$ depending only on $n$, and $|\beta|$. 
This shows: 
\[
\bigg| \frac{\partial^{|\alpha|}}{\partial x^{\alpha}}\frac{\partial^{|\beta|}}{\partial y^{\beta}} H_*(x,y,t)\bigg|\le C\exp(-c/t)
\]
Then the proposition follows immediately after noticing that, by definition,
\[
 K_{pq}(x,y) 
= \frac{s/2}{\Gamma(1-s/2)} \int_0^\infty  H_*(x,y,t) \frac{dt}{t^{1+s/2}},
\]
and hence
\[
\bigg|\frac{\partial^{|\alpha|}}{\partial x^{\alpha}} \frac{\partial^{|\beta|}}{\partial y^{\beta}}   K_{pq}(x,y)\bigg| 
=   \bigg| \frac{s/2}{\Gamma(1-s/2)} \int_0^\infty \frac{\partial^{|\alpha|}}{\partial x^{\alpha}}\frac{\partial^{|\beta|}}{\partial y^{\beta}} H_*(x,y,t) \frac{dt}{t^{1+s/2}}\bigg| \le Cs\int_0^\infty \exp(-c/t)\frac{dt}{t^{1+s/2}} \le C \,,
\]
for some constant $C>0$ that depends only on $n$ and $\ell$, and this concludes the proof.
\end{proof}

\subsection{Extension definition}\label{ext section}
% Let $u:M\to \mathbb{R}$ be a function on $M$. The fractional Laplacian of $u$ can be defined as a Dirichlet-to-Neumann map as follows:\\

% Consider the product manifold $(\widetilde{M}, h) := (M \times [0, \infty), g \times \delta_\R )$, where $\delta_\R$ is the standard metric on $\R$, and denote $(x,y)$ for a point in $M \times [0,\infty)$. Let $U(x,y):M \times [0, \infty) \to \R $ solve 
% \begin{equation*}\label{extproblem}
%     \begin{cases} \textnormal{div}_{M\times\R^+}(z^{1-s} \nabla U) = 0  &   \mbox{in} \s \widetilde{M}=M\times (0,\infty) \,,  \\ U(x,0) = u(x) &  \mbox{on} \s \partial{\widetilde{M}} = M \times \{0\} .  \end{cases}
% \end{equation*}
% Then
% \begin{definition}[\textbf{Caffarelli-Silvestre extension definition}]
% The fractional Laplacian is defined as the operator that acts on $u$ by 
% \begin{equation}\label{lapextmfd}\tag{CS}
%     \beta_s^{-1}(-\Delta)^{s/2} u(x) = \lim_{y \to 0^+} -z^{1-s} \partial_y U(x,y). 
% \end{equation}
% \end{definition}

\begin{definition}\label{weighSobspace}
    We define the weighted Sobolev space
   \begin{equation*}
       \widetilde{H}^1(\R^n \times (0,\infty)) = H^1(\R^n \times (0,\infty), z^{1-s}dxdz) 
   \end{equation*}
   as the completion of $C_c^\infty(\R^n \times [0,\infty))$ with the norm
   \begin{equation*}
       \| U\|^2_{\widetilde H^1} := \| U\|^2_{L^2(\R^n \times (0,\infty), z^{1-s}dxdz)} + \| \widetilde{D} U\|^2_{L^2(\R^n \times (0,\infty), z^{1-s}dxdz)} \,,
   \end{equation*}
   where $\widetilde{D}U =\big(\frac{\partial U}{\partial x^1}, \dotsc, \frac{\partial U}{\partial x^n}, \frac{\partial U}{\partial z} \big)$ denotes the Euclidean gradient in $\R^{n+1}$. This is a Hilbert space with the natural inner product that induces the norm above. It is a known fact that any $U \in \widetilde{H}^1(\R^n \times (0,\infty)) $ has a well defined trace in $L^2(\R^n)$ that we denote by $U(x, \cdot)$. 
\end{definition}

The following essential result by Caffarelli and Silvestre shows that fractional powers of the Laplacian on $\R^n$ can be realized as a Dirichlet-to-Neumann map via an extension problem.

\begin{theorem}[\cite{CafSi}]
Let $s \in (0,2)$ and $ u \in H^{s/2}(\R^n) \cap C^\infty(\R^n)$. Then, there is a unique solution $U = U(x,z) : \R^n \times [0, +\infty) \to \R$ among functions in $\widetilde{H}^1(\R^n \times (0,\infty))$ to the problem
\begin{equation}\label{caffextRn}
    \begin{cases}  \Delta_x U + \frac{\partial^2 U}{\partial z^2} +\frac{1-s}{z} \frac{\partial U}{\partial z} =0\, ,   & \mbox{on} \s \R^n \times (0, \infty) \\ U(x,0) = u(x) &   \mbox{for} \s x \in \R^n\, ,  \end{cases}
\end{equation}
and it satisfies
\begin{equation}\label{DtNeq}
    \lim_{z \to 0^+} z^{1-s} \frac{\partial U}{\partial z}(x,z) = - \beta_s^{-1} \, (-\Delta)^{s/2} u(x) \,,
\end{equation}
where $\Delta_x$ denotes the standard Laplacian on $\R^n$ and $\beta_s$ is a positive constant that depends only on $s$.
\end{theorem}

In \cite{CafSi}, three different proofs of this fact are presented, but each of these proofs relies on some additive structure of the base space. To prove that the extension procedure produces the fractional power of the Laplacian also on a Riemannian manifold, which is the setting we are interested in, one has to rely on different ideas. It was proved by P.R. Stinga in \cite{Stinga} that the unique solution to \eqref{caffextRn} verifying \eqref{DtNeq} admits the explicit representation
\begin{equation}\label{bdfgshdfg}
    U(p,z)= \frac{z^{s}}{2^s\Gamma(s/2)} \int_{0}^{\infty} P_t u (p) \,  e^{-\frac{z^2}{4t}}\frac{dt}{t^{1+s/2}} \, , 
\end{equation}
which expresses $U$ in terms of the solution to the heat equation $P_t u$ (and thus makes sense also on a manifold). The proof of this fact does not strongly rely on the additive structure of $\R^n$ and will be proved now also on closed Riemannian manifolds.

\vsp
First, let us define the weighted Sobolev spaces for the extension on compact manifolds.

\begin{definition}\label{weighSobspacemfd}
    We define the weighted Sobolev space
   \begin{equation*}
       \widetilde{H}^1 (\widetilde M) = \widetilde{H}^1( M \times (0,\infty)) 
   \end{equation*}
   as the completion of $C_c^\infty( M \times [0,\infty))$ with the norm
   \begin{equation*}
       \| U\|^2_{\widetilde H^1} := \| TU \|^2_{L^2(M)} + \| \widetilde \nabla U\|^2_{L^2(\widetilde M, z^{1-s}dVdz)} \,,
   \end{equation*}
   where $TU = U(\cdot, 0)$ is the trace of $U$ and $\widetilde \nabla U =(\nabla U, U_z)$ denotes the gradient in $M\times (0,+\infty)$. 
  This is a Hilbert space with the natural inner product that induces the norm above. Moreover, basically by definition, any $U \in \widetilde{H}^1 (\widetilde M) $ leaves a trace in $L^2(M\times \{0\})$. 
\end{definition}

\begin{theorem}\label{extmfd}
Let $(M^n,g)$ be a closed Riemannian manifold, let $s\in (0,2)$ and $u:M\to\R$ be smooth. Consider the product manifold $\widetilde{M}=M\times (0,+\infty)$ endowed with the natural product metric\footnote{That is, the metric defined by $\widetilde g \big((\xi_1, z_1), (\xi_2,z_2)\big)= g(\xi_1, \xi_2) + z_1z_2$, and where $\widetilde{\rm div}$ and $\widetilde \nabla$ denote the divergence and Riemannian gradient with respect to this product metric respectively.}. Then, there is a unique solution $U:M \times (0, \infty) \to \R $ among functions in $\widetilde H^1(\widetilde M )$ to
\begin{equation}\label{caffextMfd}
    \begin{cases} \widetilde{ {\rm div}}(z^{1-s} \widetilde \nabla U) = 0  &   \mbox{in} \s \widetilde{M} \,,  \\ U(p,0) = u(p) &  \mbox{for} \s p \in \partial \widetilde{M} = M\, , \end{cases}
\end{equation}
given by \eqref{bdfgshdfg}, and it satisfies
\begin{equation}\label{eneqal}
    [u]^2_{H^{{s/2}}(M)} = 2\beta_s \int_{\widetilde{M}} |\widetilde \nabla U|^2 z^{1-s} \, dVdz \,,
\end{equation}
where $[u]^2_{H^{{s/2}}(M)}$ is defined through \eqref{wethiowhoihw} and
\begin{equation}\label{betadef}
    \beta_s = \frac{2^{s-1} \Gamma(s/2)}{\Gamma(1-s/2)}\, .
\end{equation}
Moreover,
\begin{equation}\label{lapextmfd}
    \lim_{z \to 0^+} z^{1-s} \frac{\partial U}{\partial z}(p,z) =  -\beta_s^{-1} (-\Delta)^{s/2} u(p) \,, 
\end{equation}
where the fractional Laplacian on the right-hand side is defined by either \eqref{boclap} or \eqref{singintlap}. 
\end{theorem}

\begin{proof}
Note that functions in $\widetilde H^1(\widetilde M )$ leave a well-defined trace (that is, there exists a continuous trace operator with respect to the norm on $\widetilde H^1(\widetilde M )$) on $M\times \{0\}$. Then, the fact that a solution among functions in $\widetilde H^1(\widetilde M )$ exists follows by direct minimization of the associated energy $v \mapsto\int_{\widetilde{M}} |\widetilde \nabla v|^2z^{1-s} dVdz$ over $\widetilde H^1(\widetilde M)$. Since the energy is convex, the solution is also unique.

\vsp
From here we divide the proof in two steps. 

\vsp
\textbf{Step 1.} We show that the (unique) solution $U \in \widetilde H^1(\widetilde M)$ to \eqref{caffextMfd} is given by \eqref{bdfgshdfg}.

\vsp\noindent
Making the identification $T(M \times (0,+\infty) ) \simeq TM \times (0,+\infty)  $ we have
\begin{align*}
    \widetilde{ {\rm div}}(z^{1-s} \widetilde \nabla U) &= \textnormal{div}_g(z^{1-s} \nabla U) + \frac{d}{dz}(z^{1-s} U_z) \\ &= z^{1-s} \Delta U + (1-s)z^{-s} U_z + z^{1-s}U_{zz} \\ &= z^{1-s}( \Delta U + \frac{1-s}{z}U_z + U_{zz}) \, . 
\end{align*}
Thus, in order to prove that $U$ solves $\widetilde{ {\rm div}}(z^{1-s} \widetilde \nabla U) = 0$ we show that $U$ (weakly) solves
\begin{equation}\label{qwrtqrwerq}
    \mathcal{L}(U) := \Delta U + \frac{1-s}{z}U_z + U_{zz} =0 \,.
\end{equation}
Define 
\begin{equation}\label{defG}
    G(z,t) := \frac{1}{2^s \Gamma(s/2)} \frac{ z^{s}e^{-\frac{z^2}{4t}}}{t^{1+s/2}} \, ,
\end{equation}
so that \eqref{bdfgshdfg} rewrites simply as 
\begin{equation}\label{zzzzzz5}
    U(\cdot,z)=\int_{0}^{\infty} (P_t u)G(z,t)\,dt \, .
\end{equation}
It can be easily checked that $G$ satisfies 
\begin{equation}\label{gkerzero}
    -G_t+\frac{1-s}{z}G_z+G_{zz} =0\,,
\end{equation}
and also
\begin{equation*}
    \lim_{t \to 0^+} \sup_{[z_1,z_2]} G(\cdot,t) = 0 \,, \quad \lim_{t \to \infty} \sup_{[z_1,z_2]} G(\cdot,t) = 0 \,,
\end{equation*}
for every $[z_1,z_2]\subset (0,+\infty)$. Moreover, from the definition of $G$ and the fact that $u$ is smooth we se that the integral in the right-hand side of \eqref{zzzzzz5} is absolutely convergent in $\widetilde H^1(\widetilde M)$. Hence $U \in \widetilde H^1(\widetilde M)$. 

\vsp
Now we check that $U$ weakly solves the desired problem. Let $\varphi \in C_c^\infty(\widetilde M)$, $K:= {\rm supp} (\varphi)$ and $z_1, z_2 \in (0,+\infty)$ such that $K\subset \subset M\times [z_1, z_2]$. Let also 
\begin{equation*}
    \mathcal{L}^*(\varphi) := \Delta \varphi +\partial_z \left(\frac{1-s}{z} \varphi \right) +\varphi_{zz} \,.
\end{equation*}
This is the formal adjoint of the operator in \eqref{qwrtqrwerq}. Clearly $\mathcal{L}^*(\varphi) $ still has compact support in $K\subset \subset \widetilde M$ and is smooth. Then
\begin{equation*}
    \int_{\widetilde M} U \mathcal{L}^*(\varphi) \, dV dz = \int_{\widetilde M} \int_{0}^\infty (P_t u) G(z,t) \mathcal{L}^*(\varphi) \, dt  dV dz \,,
\end{equation*}
and we claim that this integral is absolutely convergent. Indeed
\begin{align*}
    \int_{\widetilde M} \int_{0}^\infty & |(P_t u) G(z,t) \mathcal{L}^*(\varphi)| \, dt \, dV dz  \le \| \mathcal{L}^*(\varphi)\|_{L^\infty} \int_0^\infty \int_K |P_t u ||G(z,t)| \, dV dz  dt\\ &\le C \left( \int_K |P_t u |^2|G(z,t)|^2 z^{1-s} \, dV dz \right)^{\frac{1}{2}} \left( \int_K \frac{1}{z^{1-s}} \, dV dz \right)^{\frac{1}{2}} <+\infty \,,
\end{align*}
since the integral in \eqref{zzzzzz5} is absolutely convergent in $\widetilde H^1(\widetilde M)$. Hence we can exchange the order of integration and we get, integrating by parts in space many times
\begin{align*}
     \int_{\widetilde M} U \mathcal{L}^*(\varphi) \, dV dz & =  \int_{0}^\infty \left( \int_{K} (P_t u)  G(z,t)  \mathcal{L}^*(\varphi) \, dV dz  \right) dt  \\ &= \int_{0}^\infty  \int_{K}  \left( G(z,t) \Delta (P_t u)  + (P_t u) \frac{1-s}{z} G_z(z,t) + (P_t u) G_{zz}(z,t) \right) \varphi \, dVdz   dt \,.
\end{align*}
Since $P_t u $ is smooth and solves the heat equation, the first term equals to
\begin{align*}
    \int_{0}^\infty   \int_{K} G(z,t) \Delta (P_t u)  \, dV dzdt & =  \int_{0}^\infty  \int_{K} G(z,t) \partial_t(P_t u) \, dV dzdt =  \int_{K} \int_{0}^\infty  G(z,t) \partial_t(P_t u) \, dtdVdz \\ & = \int_K (P_t u) G(z,t) \, dVdz \,\bigg|_{0^+}^\infty - \int_{K} \int_{0}^\infty  (P_t u ) G_t(z,t) \, dtdVdz \,.
\end{align*}
The boundary terms vanish since
\begin{align*}
   \left| \int_K (P_t u) G(z,t)  \, dV dz \right| & \le \left( \int_M |P_t u | \, dV \right)\left( \int_{z_1}^{z_2} |G(z,t)| \, dz \right) \\ & \le |M|^{1/2} \| u \|_{L^2(M)}|z_2-z_1| \sup_{[z_1, z_2]} G(\cdot, t) \to 0 \,,
\end{align*}
both as $t\to \infty$ and as $t\to 0^+$. Hence, putting all together and using \eqref{gkerzero}
\begin{align*}
    \int_{\widetilde M} \mathcal{L}(U) \varphi \, dVdz = \int_{\widetilde M} U \mathcal{L}^*(\varphi) \, dV_p dz = \int_K \int_0^\infty (P_t u) \left( -G_t +\frac{1-s}{z} G_z + G_zz \right) \varphi \, dVdzdt = 0 . 
\end{align*}
Hence $U$ given by \eqref{bdfgshdfg} is a weak solution of \eqref{qwrtqrwerq}, and by standard elliptic regularity it is also a classical solution. 

\vsp
Moreover, the fact that $U(\cdot, 0^+)=u$ follows by the explicit formula \eqref{bdfgshdfg}. Indeed, by a simple change of variable in the integral, we have 
\begin{equation*}
    U(p,z) = \frac{1}{\Gamma(s/2)} \int_0^\infty (P_{z^2/4r} u)(p) \frac{e^{-r}}{r^{1-s}} dr \,,
\end{equation*}
and taking $z\to 0^+$ in this formula gives $U(\cdot, 0^+)=u$. This concludes Step 1.

\vsp
\textbf{Step 2.} Proof of \eqref{lapextmfd}. 

\vsp \noindent
Note that by the representation formula we just proved for $U$ we have
\begin{equation*}
   \tau^s_z U (p) := \frac{U(p,z)-u(p)}{z^{s}} = \frac{1}{2^s\Gamma(s/2)} \int_{0}^{\infty} (P_tu(p)-u(p))e^{-\frac{z^2}{4t}} \frac{dt}{t^{1+s/2}} \, .
\end{equation*}
Moreover, by L'Hopital's rule
\begin{equation*}
    \lim_{z\to 0^+} \tau^s_z U = \lim_{z \to 0^+} s^{-1} z^{1-s} \frac{\partial U}{\partial z} \,.
\end{equation*}
Writing $P_t u $ as the convolution against the heat kernel $H_M$ of $M$ we get
\begin{equation*}
   \tau^s_z U (p) = \frac{1}{2^s\Gamma(s/2)} \int_{0}^{\infty} \int_M H_M(p,q,t)(u(q)-u(p)) \frac{e^{-\frac{z^2}{4t}}}{t^{1+s/2}} dV_q dt \, .
\end{equation*}
Since $u$ is smooth, and since $K_s^\ep(p,q)$ is non-singular for $\ep>0$ on the diagonal $\{p=q\}$ (recall \eqref{app ker K}) there holds
\begin{align*}
   \int_M  \frac{s/2}{\Gamma(1-s/2)}\int_{0}^{\infty}  H_M(p,q,t)|u(q)-u(p)| \frac{e^{-\frac{z^2}{4t}}}{t^{1+s/2}}  dt dV_q  = \int_M  |u(q)-u(p)| K_s^z(p,q) \, dV_q  <+\infty \,.
\end{align*}
Thus the integral in $\tau^s_z U (p)$ is absolutely convergent, and we can exchange the order of integration to get
\begin{equation*}
   \tau^s_z U (p) =  \frac{\Gamma(1-s/2)}{s 2^{s-1}\Gamma(s/2)}  \int_M (u(q)-u(p)) K_s^z(p,q) \, dV_q \,.
\end{equation*}
From here, by the very definition of the principal value 
\begin{align*}
      \lim_{z \to 0^+} z^{1-s} \frac{\partial U}{\partial z}  =   \lim_{z\to 0^+} s \cdot \tau^s_z U & = \lim_{z\to 0^+} \beta_s^{-1} \int_M (u(q)-u) K_s^z(p,q) \, dV_q \\ &= - \beta_s^{-1} \left(p.v. \int_M (u-u(q)) K_s(p,q) \, dV_q \right) \\ & =  -\beta_s^{-1} (-\Delta)^{s/2} u \,.
\end{align*}
This finishes Step 2.

\vsp
Before proving \eqref{eneqal} we prove also that the convergence in \eqref{lapextmfd} holds in $L^r(M)$ for every $r\in [1, +\infty)$. Since we have pointwise convergence, we show that the sequence is dominated. In particular, we prove that for $z\le 1$ there holds
\begin{equation}\label{qwrbhfqwekgv}
    z^{1-s} | U_z(\cdot, z)| \le C \,,
\end{equation}
where $C$ depends on $\|\Delta u\|_{L^\infty}$, $\|u\|_{L^\infty}$ and $s$. The proof is a standard barrier argument very similar to the proof of Lemma \ref{linf-grad}. Consider $b(p,z):=u(p)-C(z^2-2z^s)$, for $C>0$ that will be chosen soon. Since 
\begin{equation*}
    \widetilde{ {\rm div}}(z^{1-s} \widetilde \nabla b) = z^{1-s}(\Delta u - (4-2s)C) \,,
\end{equation*}
taking $C=\tfrac{1}{4-2s}\|\Delta u\|_{L^\infty} + 2\|u\|_{L^\infty}$, we see that $b$ is a supersolution of \eqref{caffextMfd} and that $U \le b$ on $M\times \{0, 1\}$. Hence $b$ is barrier for $U$, and by the maximum principle for $z\le 1$ we have 
\begin{equation*}
    U(\cdot,z) \le  u-\left( \frac{1}{4-2s}\|\Delta u\|_{L^\infty} +2 \|u\|_{L^\infty} \right)(z^2-2z^s) \le u+ z^s\left( \frac{1}{2-s}\|\Delta u\|_{L^\infty} + 4 \|u\|_{L^\infty} \right)  ,
\end{equation*}
and this implies 
\begin{equation*}
    \lim_{z \to 0^+} z^{1-s} U_z = \lim_{z\to 0^+} s \frac{U(\cdot, z)-u}{z^s} \le \frac{s}{2-s}\|\Delta u\|_{L^\infty} + 4s \|u\|_{L^\infty}\,.
\end{equation*}
Completely analogously using $-b$ as a barrier for $-U$, one also gets the reverse inequality. Moreover, note that the function $V:=z^{1-s}U_z$ solves $-(\Delta V + V_{zz})+\tfrac{1-s}{z}V_z =0 $, thus by the maximum principle 
\begin{equation*}
    \sup_{M\times (0,1] } |V| \le \max\left\{ \sup_{M} V(\cdot, 0^+) , \sup_{M} V(\cdot,1) \right \} \le \max\left\{ C( \|\Delta u\|_{L^\infty} , \|u\|_{L^\infty}, s) , \sup_{M} V(\cdot,1) \right \} .
\end{equation*}
But since $V(\cdot,1)= U_z(\cdot,1)$ we have by standard interior gradient estimates 
\begin{equation*}
    |U_z(p,1)| \le C \sup_{\widetilde B_{1/10}(p,1)} |U| \le C\|u\|_{L^\infty} \,,
\end{equation*}
for some absolute constant $C>0$ independent of $u$. Putting everything together 
\begin{equation*}
    \sup_{M\times (0,1] } |V| = \sup_{M\times (0,1] } z^{1-s}|U_z| \le C(\|\Delta u\|_{L^\infty} , \|u\|_{L^\infty}, s) \,,
\end{equation*}
and this concludes the proof of \eqref{qwrbhfqwekgv}.

\vsp
We're left with proving \eqref{eneqal}. Integrating by parts, for every $ \delta >0$ we find that
\begin{equation*}
    \int_{M\times [\delta, \infty)} |\widetilde \nabla U|^2 z^{1-s} \, dVdz = \beta_s^{-1}\int_{M} U(\cdot, \delta) \delta^{1-s}U_z(\cdot, \delta)  \, dV \,.
\end{equation*}
Letting now $\delta\to 0^+$, by \eqref{lapextmfd}, \eqref{qwrbhfqwekgv} and dominated convergence on the right-hand side
\begin{equation*}
    \int_{\widetilde M} |\widetilde \nabla U|^2 z^{1-s} \, dVdz = \beta_s^{-1} \int_M u(-\Delta)^{s/2} u \, dV = \frac{\beta_s^{-1}}{2} [u]^2_{H^{s/2}(M)} \,.
\end{equation*}
This concludes the proof.

% {\color{green} Conto mio nuovo per sistemare quello che ho scritto:}
% \begin{align*}
%     0&=\int_{\widetilde M} U\textnormal{div}(z^{1-s}\nabla U)\\
%     &=\int_{\widetilde M} U (z^{1-s}\Delta_M U(x,z)+z^{1-s}\partial_{zz} U(x,z)+(1-s)z^{-s}\partial_z U(x,z))\\
%     &=
% \end{align*}

% {\color{green} Conto mio nuovo per sistemare quello che ho scritto:}
% \begin{align*}
%     0&=\int_{\widetilde M} U\textnormal{div}(z^{1-s}\nabla U)\\
%     &=\int U \textnormal{div}(z^{1-s}\nabla \frac{z^s}{2^s\Gamma(s/2)} \int_{0}^{\infty} \int_M H_M(x,y,t)u(y)e^{-\frac{z^2}{4t}} \frac{dt}{t^{1+s/2}})\\
%     &=\int U \textnormal{div}(z^{1-s}\Big(s \frac{z^{s-1}}{2^s\Gamma(s/2)} \nabla(z)\int_{0}^{\infty} \int_M H_M(x,y,t)u(y)e^{-\frac{z^2}{4t}} \frac{dt}{t^{1+s/2}}+\frac{z^s}{2^s\Gamma(s/2)} \int_{0}^{\infty} \nabla \int_M H_M(x,y,t)u(y)e^{-\frac{z^2}{4t}} \frac{dt}{t^{1+s/2}}\Big)
% \end{align*}
\end{proof}
We have just used the following fact, which is proved with a one-line computation using \eqref{singintlap}.
\begin{proposition}
    For a smooth function, one has that \begin{equation*}
    [u]_{H^{s/2}(M)}^2 = 2\int_{M} u(-\Delta)^{s/2} u \, dV \, .
\end{equation*}
\end{proposition}

\begin{remark} Let us briefly comment on the role played by the energy space for the uniqueness of the extension in Theorem \ref{extmfd}. One can note that, for every $C>0$, the function $V=Cz^s$ on $\widetilde M$ is a solution of $\widetilde{ {\rm div}}(z^{1-s} \widetilde \nabla V) = 0$ with zero trace. Hence, uniqueness outside the energy space $\widetilde H^1(\widetilde M)$ does not hold in general, and every uniqueness result that does not rely on being in $\widetilde H^1(\widetilde M)$ must, in particular, rule out this phenomenon. 

\vsp
A simple uniqueness result that does not rely on the energy space is the following. Let $U$ solve $\widetilde{ {\rm div}}(z^{1-s} \widetilde \nabla U) = 0$, with $U(\cdot, 0)=0$ be such that
\begin{equation*}
    \limsup_{z\to \infty} \sup_{p\in M} |U(p,z)| z^{-s}=0 \,.
\end{equation*}
That is, $U$ grows at infinity slower than any multiple of $z^s$. Then $U\equiv 0$. 

\vsp 
This can be proved using $Cz^s$ as a barrier. Indeed, this is a solution of $\widetilde{ {\rm div}}(z^{1-s} \widetilde \nabla U) = 0$, with $U(\cdot, 0)=0$. By the growth hypothesis on $U$ we have that there exists $C$ large (depending on $U$) such that $Cz^s \ge 10 U$ on $\widetilde M$. In particular, with this $C$, we have $U < Cz^s$. 

\vsp
Now start decreasing $C$. The graph of $Cz^s$ can never touch $U$ from above since this would contradict the maximum principle in the interior. Hence $U < Cz^s$ for every $C>0$ and sending $C\to 0^+$ gives $U\le 0$. 

\vsp
By the same argument from below one also gets $U\ge 0$. Hence $U\equiv 0$ and this concludes the proof.    
\end{remark}

\section{The fractional Sobolev energy}
\subsection{Several definitions and their equivalence}
We recall the definition for the fractional Sobolev seminorm that we have used in the previous section, and we define the associated functional space.
\begin{definition}
    We define the fractional Sobolev seminorm $[u]_{H^{s/2}(M)}$ for $s \in (0,2)$ as  
 \begin{equation}\label{sobint}
     [u]^2_{H^{{s/2}}(M)} =\iint_{M\times M} (u(p)-u(q))^2 K_s(p,q) \, dV_p dV_q \, .
\end{equation} 

The associated functional space $H^{s/2}(M)$ is
\begin{equation}\label{sobspace}
H^{s/2}(M)=\{u\in L^2(M)  \text{ : } [u]^2_{H^{s/2}(M)}<\infty\} \,,
\end{equation}
and it is called the \textit{fractional Sobolev space of order $s/2$}. This is a Hilbert space with norm given by 
\begin{equation*}
    \|u\|_{H^{{s/2}}(M)}^2=\|u\|_{L^2(M)}^2+[u]^2_{H^{{s/2}}(M)}\, .
\end{equation*}
\end{definition}

The fractional Sobolev seminorm can also be expressed using spectral or extension approaches: 

\begin{proposition}\label{afdsgsdgh}
Let $u\in H^{s/2}(M)$. Then, the fractional Sobolev seminorm \eqref{sobint} is equal to
\begin{equation}\label{sobspectr}
    [u]^2_{H^{{s/2}}(M)} = 2\sum_{k=1}^\infty \lambda_k^{s/2} \lp u, \phi_k\rp_{L^2(M)} ^2
\end{equation}
and
\begin{equation}\label{sobext}
   [u]^2_{H^{{s/2}}(M)} = \inf_{v \in \widetilde H^1(\widetilde M)} \left\{ 2\beta_s \int_{\widetilde{M}} |\widetilde \nabla v|^2 z^{1-s} \, dVdz \, : \, v(\cdot,0)=u(\cdot) \mbox{ in } L^2(M) \right\}.
\end{equation}
Moreover, the conclusions of Theorem \ref{extmfd} also hold for $u$ (with the exception of \eqref{lapextmfd}), and the infimum in \eqref{sobext} is attained by the unique $U \in \widetilde H^1(\widetilde M) $ given by Theorem \ref{extmfd}. In particular, we also have that
\begin{equation}\label{asdfasd}
    [u]^2_{H^{s/2}(M)}= 2\beta_s \int_{\widetilde{M}} |\widetilde{\nabla} U|^2 z^{1-s} \, dVdz \,,
\end{equation}
where $\beta_s$ is the constant defined in \eqref{betadef}. 
\end{proposition}
\begin{proof}
\textbf{Step 1.} We show that \eqref{sobint} and \eqref{sobspectr} coincide for a function in $L^2(M)$.\\

Recall the regularised kernel $K_s^\ep$ defined in \eqref{app ker K}, which is bounded, symmetrical and increases monotonically to $K_s$ as $\ep\to 0$. By monotone convergence and these properties, for any function $u\in L^2(M)$ we can write
\begin{align}
    [u]^2_{H^{s/2}(M)}&=\iint_{M\times M} (u(p)-u(q))^2 K_s(p,q) \, dV_p dV_q \nonumber\\
    &= \lim_{\ep\to 0} \iint_{M\times M} (u(p)-u(q))^2 K_s^\ep(p,q)\, dV_p dV_q \nonumber\\
    &= \lim_{\ep\to 0} \,2\iint_{M\times M} (u(p)-u(q))u(p) K_s^\ep(p,q)\, dV_p dV_q \nonumber\\
    &= \lim_{\ep\to 0} \,2\int_{M} ((-\Delta)^{s/2}_\ep u)(p) u(p) K_s^\ep(p,q)\, dV_p\, , \label{cutlapint}
\end{align}
where we have set
\begin{align*}
    ((-\Delta)^{s/2}_\ep u)(p)&:=\int_M (u(p)-u(q)) K_s^\ep(p,q)\,dV_q\\
    &=\frac{s/2}{\Gamma(1-s/2)} \int_M (u(p)-u(q)) \int_{0}^{\infty} H_M(p,q,t) e^{-\ep^2/4t} \frac{dt}{t^{1+s/2}}\,dV_q\\
    &=\frac{s/2}{\Gamma(1-s/2)} \int_{0}^{\infty}(u(p)-P_t u(p))  e^{-\ep^2/4t} \frac{dt}{t^{1+s/2}}\, .
\end{align*}
Now, if $(\phi_k)_{k \ge 0}$ is an orthonormal basis of $L^2(M)$ made of eigenfunctions for $(-\Delta)$, with eigenvalues 
\begin{equation*}
    0 = \lambda_0 < \lambda_1 \leq \lambda_2 \le \dotsc \le \lambda_k \conv{k \to \infty} +\infty\, ,
\end{equation*}
then they are also
eigenfunctions for $(-\Delta)^{s/2}_\ep$ with eigenvalues
$$
\lambda_{k,\ep}^{s/2}:= \frac{s/2}{\Gamma(1-s/2)} \int_{0}^{\infty}(1-e^{-\lambda_k t})  e^{-\ep^2/4t} \frac{dt}{t^{1+s/2}}\, ,$$
which one sees immediately by applying the formula above for $(-\Delta)^{s/2}_\ep$ to $\phi_k$ and using that $P_t \phi_k =e^{-\lambda_k t} \phi_k$.
These eigenvalues are uniformly bounded in $k$ (for a fixed $\ep>0$) and increase monotonically to the $\lambda_{k}^{s/2}$ as $\ep\to 0^+$.

\vsp
Expanding $u=\sum_{k=0}^\infty  a_k \phi_k$, with $a_k:=\lp u, \phi_k\rp_{L^2(M)}$, we deduce that
$$
(-\Delta)^{s/2}_\ep u= \sum_{k=0}^\infty  \lambda_{k,\ep}^{s/2} \lp u, \phi_k\rp_{L^2(M)} \phi_k\, .
$$
We remark that the expression makes sense since the $\lambda_{k,\ep}^{s/2}$ are bounded uniformly in $k$ (for a fixed $\ep$), and thus the sum is absolutely convergent in $L^2(M)$. Using this fact, substituting into \eqref{cutlapint} gives that
\begin{align*}
     [u]^2_{H^{s/2}(M)}&=\lim_{\ep\to 0} \,2\int_{M} ((-\Delta)^{s/2}_\ep u)(p) u(p) K_s^\ep(p,q)\, dV_p\\
     &=\lim_{\ep\to 0} \,2\sum_{k=0}^\infty \lambda_{k,\ep}^{s/2} a_k^2\, .
\end{align*}
Using again the monotone convergence theorem (for sums now), we deduce that
\begin{align*}
     [u]^2_{H^{s/2}(M)}
     &=2\sum_{k=0}^\infty \lambda_{k}^{s/2} a_k^2
\end{align*}
as desired.

\vsp
\textbf{Step 2.} We show that $U$ given by the representation formula \eqref{bdfgshdfg}, which was only used for smooth functions $u$, is valid for $u\in H^{s/2}(M)$ in general and moreover \eqref{asdfasd} still holds with this $U$.

\vsp 
Fix $u \in H^{s/2}(M) $, and let $U$ be defined through the representation formula \eqref{bdfgshdfg}.  We will first show that $U$ has finite $\widetilde H^1(\widetilde M)$ energy, using the spectral expression \eqref{sobspectr} for the energy that we have just proved. Recall that if $\phi_k$ is an eigenfunction of $(-\Delta)$, then $P_t \phi_k =e^{-\lambda_k t} \phi_k$. Therefore, writing $u=\sum_{k=0}^\infty  a_k \phi_k$, where $a_k:= \lp u, \phi_k\rp_{L^2(M)}$, we have that
\begin{align*}
U(p,z)&= \frac{z^{s}}{2^s\Gamma(s/2)} \int_{0}^{\infty} P_t u (p) \,  e^{-\frac{z^2}{4t}}\frac{dt}{t^{1+s/2}}\\
&=\frac{z^{s}}{2^s\Gamma(s/2)} \sum_{k=0}^\infty a_k \phi_k (p) \int_{0}^{\infty}    e^{-\lambda_k t - \frac{z^2}{4t}} \, \frac{dt}{t^{1+s/2}}\, .
\end{align*}
Then, we can compute (recall that $\nabla=\nabla_p$ denotes the gradient on $M$)
\begin{align*}
    \nabla U(p,z)=\frac{z^{s}}{2^s\Gamma(s/2)} \sum_{k=1}^\infty a_k \nabla \phi_k (p) \int_{0}^{\infty}    e^{-\lambda_k t -\frac{z^2}{4t}} \, \frac{dt}{t^{1+s/2}}\, ,
\end{align*}
and
\begin{align*}
    \partial_z U(p,z) & = \frac{1}{2^s\Gamma(s/2)} \sum_{k=1}^\infty  a_k\phi_k (p)\int_{0}^{\infty}  e^{-\lambda_k t -\frac{z^2}{4t}} \Big( sz^{s-1} - \frac{z^{1+s}}{2t} \Big) \frac{dt}{t^{1+s/2}}  \\ &=  \frac{z^{s-1}}{2^s\Gamma(s/2)} \sum_{k=1}^\infty  a_k\phi_k (p)\int_{0}^{\infty}  e^{-\lambda_k t -\frac{z^2}{4t}} \Big( s - \frac{z^2}{2t} \Big) \frac{dt}{t^{1+s/2}}  \,.
\end{align*}
Recall that the $\phi_i$ and $\phi_j$ are orthogonal in $L^2(M)$ and $H^1(M)$ seminorms for $i\neq j$, and that moreover $\int_M \phi_k^2=1$ and $\int_M |\nabla \phi_k |^2=\lambda_k$ for every $k$. Then, given $z\in\R_+$ we find that
\begin{align*}
    \int_{M\times \{z\}}|\nabla U(p,z)|^2dV_p
    &=\frac{z^{2s}}{2^{2s}\Gamma^2(s/2)}\sum_{k=1}^\infty  a_k^2   \Big(\int_{0}^{\infty} e^{-\lambda_k t -\frac{z^2}{4t}} \frac{dt}{t^{1+s/2}}\Big)^2 \int_M |\nabla\phi_k|^2 \,    dV \\
    &=\frac{z^{2s}}{2^{2s}\Gamma^2(s/2)}\sum_{k=1}^\infty  \lambda_ka_k^2   \Big(\int_{0}^{\infty}  e^{-\lambda_k t -\frac{z^2}{4t}} \frac{dt}{t^{1+s/2}}\Big)^2 \\ &= \frac{z^{2s}}{2^{2s}\Gamma^2(s/2)}\sum_{k=1}^\infty  \lambda_k^{1+s}a_k^2   \Big(\int_{0}^{\infty}  e^{-r-\frac{z^2\lambda_k}{4r}}\frac{dr}{r^{1+s/2}}\Big)^2 ,
\end{align*}
where in the last line we have performed the change of variables $r=\lambda_k t$.

\vsp
We can argue analogously for $\partial_z U$, which leads to
\begin{align*}
    \int_{M\times\{z\}} \big(\partial_z U(p,z)\big)^2dV_p & = \frac{z^{2s-2}}{2^{2s}\Gamma(s/2)^2} \sum_{k=1}^\infty  a_k^2 \left( \int_{0}^{\infty}  e^{-\lambda_k t -\frac{z^2}{4t}} \Big( s - \frac{z^2}{2t} \Big)\frac{dt}{t^{1+s/2}} \right)^2 \\ &= \frac{z^{2s-2}}{2^{2s}\Gamma(s/2)^2} \sum_{k=1}^\infty  a_k^2 \lambda_k^s \left( \int_{0}^{\infty}  e^{-r -\frac{z^2\lambda_k}{4r}} \Big( s - \frac{z^2\lambda_k }{2r} \Big)\frac{dr}{r^{1+s/2}} \right)^2.
\end{align*}
Now, multiplying by $z^{1-s}$ and integrating in $z$ over $(0,\infty)$, and then performing the change of variables $z=\lambda_k^{-1/2} w$ (so that $z^2\lambda_k=w^2$), gives that
\begin{align*}
    \int_{M\times \R_+}|\nabla U(p,z)|^2z^{1-s} dV_p\,dz&=\frac{1}{2^{2s}\Gamma^2(s/2)}\sum_{k=1}^\infty  \lambda_k^{1+s}a_k^2   \int _0^\infty  z^{1+s}\Big(\int_{0}^{\infty}  e^{-r -\frac{z^2\lambda_k}{4r}}\frac{dr}{r^{1+s/2}}\Big)^2  dz\\
    &=\frac{1}{2^{2s}\Gamma^2(s/2)}\sum_{k=1}^\infty  \lambda_k^{s/2} a_k^2   \int_0^\infty  w^{1+s}\Big(\int_{0}^{\infty}  e^{-r -\frac{w^2}{4r}}\frac{dr}{r^{1+s/2}}\Big)^2 dw \\ &=2 c_1(s)\sum_{k=1}^\infty  \lambda_k^{s/2}a_k^2\\
    &=c_1(s)[u]_{H^{s/2}(M)}^2 \,, 
\end{align*}
and similarly 
\begin{align*}
    \int_{M\times \R_+}|\partial_z U(p,z)|^2 z^{1-s} dV_p\,dz & = \frac{1}{2^{2s}\Gamma^2(s/2)} \sum_{k=1}^\infty  \lambda_k^{s} a_k^2   \int _0^\infty  z^{s-1} \left( \int_{0}^{\infty}  e^{-r -\frac{z^2\lambda_k}{4r}} \Big( s - \frac{z^2\lambda_k }{2r} \Big)\frac{dr}{r^{1+s/2}} \right)^2  dz\\
    &=\frac{1}{2^{2s}\Gamma^2(s/2)}\sum_{k=1}^\infty  \lambda_k^{s/2} a_k^2   \int_0^\infty  w^{s-1} \left( \int_{0}^{\infty}  e^{-r -\frac{w^2}{4r}} \Big( s - \frac{w^2}{2r} \Big)\frac{dr}{r^{1+s/2}} \right)^2 dw \\ &=2 c_2(s)\sum_{k=1}^\infty  \lambda_k^{s/2}a_k^2\\
    &=c_2(s)[u]_{H^{s/2}(M)}^2 \,.
\end{align*}
Here, we have defined $c_1(s)$ and $c_2(s)$ implicitly as the corresponding constants (which depend only on $s$) resulting from the expression, and we have applied \eqref{sobspectr} in the last line in both computations.

\vsp
Putting everything together, we get that
\begin{equation*}
     \int_{M \times \R_+}|\widetilde \nabla U(p,z)|^2 z^{1-s} \, dV_p dz= \big(c_1(s)+c_2(s) \big) [u]_{H^{s/2}(M)}^2 \, .
\end{equation*}
We could write the constant $\big(c_1(s)+c_2(s) \big)$ explicitly in terms of the resulting complicated integral expressions. On the other hand, thanks to \eqref{eneqal} and \eqref{betadef} from Theorem \ref{extmfd} (which was proved only for smooth functions), we know that $c_1(s)+c_2(s) = (2\beta_s)^{-1}$. This proves \eqref{asdfasd} with $U$ given by the representation formula \eqref{bdfgshdfg}.

\vsp
In particular, we now know that $U$ has finite energy for the extension problem. Moreover, arguing as in Step 1 of the proof of Theorem \ref{extmfd}, it is simple to see that $U$ has $u$ as its trace in $L^2(M)$, and that it is a weak (meaning in duality with $C_c^\infty(\widetilde M)$) solution to $\widetilde{ {\rm div}}(z^{1-s} \widetilde \nabla U) = 0$. Let now $U_{\rm min} \in \widetilde H^1(\widetilde M)$ be defined as the unique minimizer of \eqref{sobext}. The fact that $U_{\rm min}$ exists follows by a standard lower-semicontinuity argument, just as at the beginning of the proof of Theorem \ref{extmfd}, together with the fact that the space of competitors is not empty (which holds since for example $U$ defined above, which has finite $\widetilde H^1(\widetilde M)$ energy, is one such competitor). Clearly, $U_{\rm min}$ is also a weak solution of $\widetilde{ {\rm div}}(z^{1-s} \widetilde \nabla U_{\rm min}) = 0$ with trace $u$.

\vsp
\textbf{Step 3.} $U=U_{\rm min}$.

\vsp 
This follows directly from the uniqueness of weak solutions shown in Lemma \ref{uniq weak}, which we state as a separate result after the present proof. 

\vsp 
With this, we conclude the proof of Proposition \ref{afdsgsdgh}.
\end{proof}

\begin{lemma}(Uniqueness of weak solutions)\label{uniq weak} Let $u\in L^2(M)$, and denote by $T:\widetilde H^1(\widetilde M) \to L^2(M)$ the trace operator. Then, there exists at most one solution $U \in \widetilde H^1(\widetilde M)$ to the problem
\begin{equation*}
    \begin{cases}  \widetilde{ {\rm div}}(z^{1-s} \widetilde \nabla U ) = 0  &   \mbox{in duality with} \,\, C_c^\infty(M\times(0,\infty)) \,,  \\  TU=u \,. \end{cases}
\end{equation*}
\end{lemma}
\begin{proof}
    Suppose $U_1$ and $U_2$ are two such solutions and denote $V:=U_1-U_2$. By hypothesis $TV=0$. 
    
    \vsp
    We claim that there exists a sequence $(V_k)_k \in C^\infty_c(M\times (0,\infty))$ such that 
    \begin{equation}\label{wferfqwd}
        \int_{\widetilde M} |\widetilde \nabla V_k - \widetilde \nabla V |^2 z^{1-s} \, dVdz \to 0 \,, \s \textnormal{as} \,\, k\to \infty \,.
    \end{equation}
The point here being that $V_k$ is zero both in a neighborhood of $M\times \{0\}$ and in a neighborhood of infinity. 
    
    \vsp
    The proof is inspired by (a weighted version of) \cite[Section 5.5, Theorem 2]{Evans}. By the definition of the space $\widetilde H^1(\widetilde M)$ there exists a sequence $(U_k)_k \subset C^\infty(M \times [0,\infty))$ with (as $k\to \infty$)
     \begin{equation*}
        \int_{\widetilde M} |\widetilde \nabla U_k - \widetilde \nabla V |^2 z^{1-s} \, dVdz \to 0 \,, \s \textnormal{and} \,\,\, TU_k = U_k(\cdot, 0) \to 0 \mbox{ in }L^2(M)\,.
    \end{equation*}
    Note also that $V$ is smooth in $M\times (0, \infty)$.  Now, for every $(p,z) \in \widetilde M$, by the fundamental theorem of calculus and Holder's inequality
    \begin{align*}
        |U_k(p,z)|^2 &\le 2|U_k(p,0)|^2 + 2\left(\int_0^z |\widetilde \nabla U_k(p,y)| \, dy \right)^2 \\ &\le C|U_k(p,0)|^2 + Cz^s \int_0^z |\widetilde \nabla U_k(p,y)|^2 y^{1-s} \, dy \,,
    \end{align*}
    and integrating for $p\in M$ gives
    \begin{equation*}
        \int_M |U_k(p,z)|^2 \, dV_p \le C\int_M|U_k(\cdot,0)|^2 + Cz^s \int_M\int_0^z |\widetilde \nabla U_k(p,y)|^2 y^{1-s} \, dy \, dV_p \,.
    \end{equation*}
    Letting $k\to \infty$ we get
    \begin{equation}\label{wvbsdefggf}
       \int_M |V(\cdot,z)|^2 \, dV \le  Cz^s \int_M\int_0^z |\widetilde \nabla V|^2 y^{1-s} \, dy \, dV_p \,.
    \end{equation} 

    \vsp
    Now, for every $k\ge 10$, let $\eta_k \in C^\infty([0,+\infty))$ be a smooth cutoff function with $\eta=0$ on $[0, 1/k]$, $\eta=1$ on $[2/k, \infty)$ and $|\eta'|\le C k $. We claim that the sequence $ V \eta_k = V(p,z)\eta_k(z) \in C^{\infty}_c(M\times (0,\infty))$ has the desired property. We have 
    \begin{align*}
        \int_{\widetilde M} & |\widetilde \nabla (V \eta_k) - \widetilde \nabla V |^2 z^{1-s} \, dVdz \le C \int_{\widetilde M}| \widetilde \nabla V|^2 (1-\eta_k)^2 z^{1-s} + C  \int_{\widetilde M} |V|^2|\eta_k'|^2 z^{1-s}  =: I_{1,k} + I_{2,k} \,.
    \end{align*}
    We estimate the two integrals separately. 

    \vsp
    For the first integral we have
    \begin{equation*}
        I_{1,k} \le C \int_{0}^{2/k} \int_{M} | \widetilde \nabla V|^2 z^{1-s} \, dVdz  \to 0 \,,
    \end{equation*}
    as $k \to \infty$, since $V$ has finite energy. 

\vsp
Moreover, by \eqref{wvbsdefggf}, we have regarding the second integral
\begin{align*}
    I_{2,k} & \le C k^2 \int_{0}^{2/k} \int_M z^{1-s} |V|^2 \, dVdz  \\ &\le C k^2 \int_{0}^{2/k} z^{1-s} \left(z^s \int_M \int_0^z  |\widetilde \nabla V|^2y^{1-s} \, dydV \right)dz  \\ &\le C k^2 \left(\int_{0}^{2/k} z \, dz \right) \left( \int_0^{2/k} \int_M  |\widetilde \nabla V|^2 y^{1-s} \, dVdy \right) \\ & = C \int_0^{2/k} \int_M  |\widetilde \nabla V|^2 y^{1-s} \, dVdy  \to 0
\end{align*}
as $k\to \infty$, again as $V$ has finite energy. 

\vsp
Hence $V_k:= V\eta_k$ has the desired property \eqref{wferfqwd}, and it can be used as a test function in the weak formulation in duality with $C_c^\infty(\widetilde M)$. Multiplying $\widetilde{ {\rm div}}(z^{1-s} \widetilde \nabla V ) = 0$ by $V_k$, integrating on $\widetilde M$ and integrating by parts gives
\begin{equation*}
    \int_{\widetilde M} ( \widetilde \nabla V \cdot \widetilde \nabla V_k) z^{1-s} \, dVdz =0 \,.
    \end{equation*}
Letting $k\to \infty $ and using \eqref{wferfqwd} gives
\begin{equation*}
    \int_{\widetilde M} | \widetilde \nabla V |^2  z^{1-s} \, dVdz =0 \,,
\end{equation*}
hence $V$ is constant, and then (since $TV=0$) it must be $V\equiv 0$. Thus, $U_1=U_2$ coincide, and the proof is complete.
\end{proof}

\subsection{A note on noncompact manifolds}

In this subsection, we briefly describe if and how the given previous definitions of the spaces $H^{s/2}(M)$ generalize to the case of complete, noncompact Riemannian manifolds (without boundary). 

\vsp
First, let us stress that all the properties and estimates for the heat kernel $H_M$, and thus also for the singular kernel $K_s$, in \autoref{Ker prop section} hold for every complete Riemannian manifold (not necessarily compact).

\vsp 
Recall definitions (i)-(iii) from the introduction (see \eqref{wethiowhoihw2}--\eqref{cafextintro}). First, let us rewrite definitions (i) and (ii) of the $H^{s/2}$ seminorm, still on a closed manifold $M$, exploiting the corresponding fractional Laplacians. Indeed, note that definition \eqref{wethiowhoihw} can be rewritten (say, for smooth functions) as
\begin{equation*}
    [u]_{H^{s/2}(M)}^2 = 2\int_M u (-\Delta)^{s/2}_{\rm Si} u \, dV \,,
\end{equation*}
where $(-\Delta)^s_{\rm Si}$ is the singular integral fractional Laplacian given by \eqref{singintlap}. Similarly, the spectral definition \eqref{ghfghfg} of the seminorm can be written as
\begin{equation*}
     [u]_{H^{s/2}(M)}^2 = 2\int_M u (-\Delta)^{s/2}_{\rm Spec} u \, dV \,,
\end{equation*}
where $(-\Delta)^{s/2}_{\rm Spec}$ is the spectral fractional Laplacian given by 
\begin{equation}\label{qergq}
    (-\Delta)^{s/2}_{\rm Spec} u = \sum_{k \ge 0} \lambda_k^{s/2} \langle u, \varphi_k \rangle_{L^2(M)} \varphi_k \,.
\end{equation}
Here, the convergence on the right-hand side is to be understood in $L^2(M)$.

\vsp
Both of these definitions can be generalized to the case of a noncompact manifold, perhaps without equality between them anymore:

\vsp
The singular integral definition \eqref{singintlap} applies verbatim to the case of noncompact manifolds. This requires dealing with the heat kernel on noncompact manifolds. We refer to the survey \cite{Grigsurv} for the construction and properties of the heat kernel on complete, noncompact Riemannian manifolds.
In the case of the Euclidean space $\R^n$, this viewpoint is consistent (i.e. coincides) with the usual definition (see Remark \ref{cvcvcvcv}).

\vsp
Moreover, also the spectral fractional Laplacian expression has an interpretation on noncompact manifolds since actually it is not needed that the spectrum is discrete. Indeed, for every (possibly) noncompact manifold $M$ we can regard $(-\Delta)$ as a densely defined, nonnegative, essentially self-adjoint unbounded operator on $L^2(M)$. Then, by the spectral theorem, there exists a unique \textit{spectral resolution} $E$ of $(-\Delta)$. That is, an operator-valued measure 
\begin{equation*}
    E : \big\{ \textnormal{Borel subsets of }[0, +\infty) \big\} \to \big\{\textnormal{Bounded linear operators on } L^2(M) \big\}
\end{equation*}
supported on the spectrum $\sigma(-\Delta)\subset [0,+\infty)$ of $(-\Delta)$ (which can, in general, be non-discrete) such that, for every $u\in {\rm Dom}(-\Delta)$ and $v\in L^2(M)$
\begin{equation*}
    \lp -\Delta u , v \rp_{L^2(M)} = \int_0^\infty \lambda \, d\lp E_\lambda u, v \rp =\int_{\sigma(-\Delta)} \lambda \, d\lp E_\lambda u, v \rp \,. 
\end{equation*}
Actually $E_\lambda$ is a projector, that is a non-negative and self-adjoint operator with $E_\lambda^2=E_\lambda$, for every $\lambda \in [0, +\infty)$. Then, the spectral fractional Laplacian is defined by the spectral theorem as 
\begin{equation}\label{specdef}
    (-\Delta)^{s/2}_{\rm Spec} u :=  \int_0^\infty \lambda^{s/2} \, d E_\lambda u  \,,
\end{equation}
for every $u$ in its natural domain 
\begin{equation*}
    u \in {\rm Dom}( (-\Delta)^{s/2}_{\rm Spec}) = \left\{ v \in L^2(M) \,\, \Big| \,\, \int_0^\infty \lambda^s \, d \| E_\lambda v \|^2 < +\infty \right\}.
\end{equation*}
This formula coincides with definition \eqref{qergq} we gave for closed manifolds since in the case of closed manifolds, the spectrum $\sigma(-\Delta)$ is discrete, and the spectral measure $E$ is supported on the eigenvalues. 

\vsp
Moreover, it is important to notice that the spectral formula \eqref{specdef} coincides, essentially always (meaning on the natural function space where both formulas make sense and the integrals converge), with the one we gave in \eqref{boclap} using Bochner's integrals. Indeed, for every $u$ that makes the integral in \eqref{boclap} convergent in the sense of Bochner, then $u\in {\rm Dom}((-\Delta)^{s/2}_{\rm Spec})$ and the two Laplacians coincide $(-\Delta)^{s/2}_{\rm Spec} u = (-\Delta)^{s/2}_{B} u $. Making a complete and precise proof of this is beyond the scope of this work, but the proof is essentially as follows. Let $u\in {\rm Dom}((-\Delta)^{s/2}_{\rm Spec})$ so that 
\begin{equation*}
    \int_0^\infty \lambda^s \, d \| E_\lambda u \|^2 < +\infty \,,
\end{equation*}
and recall formula \eqref{numeric1}. Then (see Section A.5.4 in \cite{GriBook} to justify all the steps)
\begin{align*}
 \|  (-\Delta)^{s/2}_{\rm Spec} u \|_{L^2(M)}^2 & = \int_0^\infty \lambda^s \, d \| E_\lambda u \|^2 \\ &= \int_0^\infty \left( \frac{1}{\Gamma(-s/2)} \int_0^\infty
(e^{-\lambda t}-1)\frac{dt}{t^{1+s/2}} \right)^2 \, d \| E_\lambda u \|^2 \\ &= \left \| \frac{1}{\Gamma(-s/2)} \int_0^\infty
(e^{ t \Delta} u - u )\frac{dt}{t^{1+s/2}}  \right \|_{L^2(M)}^2 = \|  (-\Delta)^{s/2}_{B} u \|_{L^2(M)}^2 \,.
\end{align*}
From here, using that both $(-\Delta)^{s/2}_{\rm Spec}$ and $(-\Delta)^{s/2}_{B}$ are self-adjoint one can depolarize the last identity of the norms to get  $(-\Delta)^{s/2}_{\rm Spec} u = (-\Delta)^{s/2}_{B} u $ in $L^2(M)$. 

\vsp
Moreover, it was proved in \cite{CG23} that they also coincide, on a very general class of functions, with $(-\Delta)^{s/2}_{\rm Si} u $ on every stochastically complete Riemannian manifold.

\vsp
Regarding definition (iii), via the extension problem, it still generalizes well in the case of some non-compact manifolds. Some extra assumptions are needed in order to establish the equivalence between (i) and (iii) ---see \cite{BGS} for a related discussion concerning the definition of the fractional Laplacian on noncompact manifolds.

\vsp
It will be clear from our proofs that, in the case of noncompact manifolds for which the equivalence of (i) and (iii) can be established, the fractional Sobolev spaces $H^{s/2}$ will enjoy the same properties as the ones established here in the case of compact manifolds (e.g. the monotonicity formula), with almost identical proofs.

\subsection{Monotonicity formula for stationary points of semilinear elliptic functionals and \texorpdfstring{$s$}{}-minimal surfaces}\label{GenMonFormula}

The monotonicity formula for minimizing $s$-minimal surfaces in $\R^n$ was proved in the seminal article \cite{CRS}, and for Allen-Cahn type critical points, it was first obtained in \cite{CC}. In \cite{MSW}, the monotonicity formula is shown to extend to stationary $s$-minimal surfaces. Here, we prove the analogous (local) monotonicity formula on a Riemannian manifold. The proof holds simultaneously for any $s$-minimal surface, that is, for any stationary point of the fractional perimeter regardless of second variation or regularity, and also for any stationary point of a semilinear elliptic functional with a nonnegative potential term, hence including the fractional Allen-Cahn energy. For $r>0$ and $p \in M $ denote
\begin{equation}\label{notationballs}
\begin{aligned}
     B_r(p) & = \big\{ q \in M \, : \, d_g(q, p) < r \big\} \,, \\[0.5mm]
     \ov{B}_{r}^+(p, 0) & = \big\{ (q,z) \in \ov{M} \, : \, d_{\ov{g}}((q,z), (p, 0) ) < r \big\} \,,  \\
    \partial \ov{B}_{r}^+(p, 0) & = \partial \left( \ov{B}_{r}^+(p, 0) \right) \, \\ \partial^+ \ov{B}_{r}^+(p, 0) & = \partial \ov{B}_{r}^+(p, 0) \cap \{ z>0 \} \,.
\end{aligned} 
\end{equation}
In all this section, since there will be no possible ambiguity, we will use $\nabla$ instead of $\widetilde \nabla$ to denote the gradient in  $\widetilde M$ with respect to the product metric.

\begin{theorem}[\textbf{Monotonicity formula}]\label{monfor}
Let $(M^n, g)$ be an $n$-dimensional, closed Riemannian manifold. Let $s \in (0,2)$ and
\begin{equation*}
    \mathcal E(v)=[v]^2_{H^{s/2}(M)}+\int_M F(v) \, dV ,
\end{equation*}
 where $F$ is any smooth nonnegative function. Let $u:M\to \R $ be stationary for $\mathcal{E}$ under inner variations, meaning that  $\mathcal E(u)<\infty$ and for any smooth vector field $X$ on $M$ there holds $\frac{d}{dt}\big|_{t=0} \mathcal{E}(u\circ\psi_X^t)=0$, where $\psi_X^t$ is the flow of $X$ at time $t$. For $(p_\circ,0) \in \ov{M} $ and $R>0$ define 
\begin{equation*}
    \Phi(R) := \frac{1}{R^{n-s}} \left( \frac{\beta_s}{2}\int_{\widetilde{B}^+_R(p_\circ,0)} z^{1-s}| \nabla  U(p,z)|^2 \, dV_p dz +  \int_{B_R(p_\circ)} F(u) \, dV \right) , 
    \end{equation*}
where $U$ is the unique solution given by Theorem \ref{extmfd}. Then, there exist constants $C=C(n)$ and $R_{\rm max}=R_{\rm max}(M, p_\circ)>0$ with the following property: whenever $R_\circ \le R_{\rm max}$ and $K$ is an upper bound for all the sectional curvatures of $M$ in $B_{R_\circ}(p_\circ)$, then
\begin{equation*}
    R \mapsto \Phi(R)e^{C \sqrt{K} R } \s \textit{is non-decreasing for} \s R < R_\circ \,,
\end{equation*}
and the inequality
\begin{align*}
    \Phi'(R) \ge - C \sqrt{K} \Phi(R)+\frac{ s }{R^{n-s+1}} \int_{B_R(p_\circ)} F(u) \, dV + \frac{2\beta_s}{R^{n-s}} \int_{\partial^+ \ov{B}_R^+(p_\circ,0)} z^{1-s} \lp \nabla U , \nabla d \rp ^2 \, d\,\widetilde{\sigma}
\end{align*}
holds for all $R < R_0 $, with $d(\cdot) = d_{\ov{g}}((p_\circ,0), \,\cdot\, )$ the distance function on $\ov{M}$ from the point $(p_\circ, 0)$.

\vsp
Moreover, in the particular case where $M=\R^n$, $F\equiv 0$, $s\in(0,1)$, and $u=\chi_E-\chi_{E^c}$ is a stationary set for the fractional $s$-perimeter, there holds
\begin{align*}
    \Phi'(R) = \frac{\beta_{s}}{2}{R^{n-s}} \int_{\partial^+ \ov{\B}_R^+(p_\circ,0)} z^{1-s} \lp \nabla U , \nabla d \rp ^2 \, dxdz \ge 0 \,,
\end{align*}
which shows that $\Phi$ is nondecreasing and that it is constant if and only if $E$ is a cone. 
\end{theorem}
\begin{remark}\label{injradrmk}
    It will follow from the proof that the radius $R_{\rm max}$ in Theorem \ref{monfor} can be taken to be $R_{\rm max} = \inj_{M}(p_\circ)/4$. Moreover, since $M$ is compact $R_{\rm max}$ is uniformly bounded below as $ R_{\rm max} (M,p_\circ) \ge \inj_M/4$, for all $p_\circ \in M$.
\end{remark}
 Before proving the monotonicity formula of Theorem \ref{monfor}, we will need two preliminary lemmas from Riemannian geometry, which will allow us to bound the ``Riemannian errors" in two formulas regarding the distance function.
\begin{lemma} \label{lemmaaux11} Let $(M^n,g)$ be an $n$-dimensional Riemannian manifold, $p\in M$, $R_0 < \inj_M(p)$ and let $K$ be an upper bound for all the sectional curvatures in $B_{R_0}(p)$. Denote by $d$ the distance function to the point $p$. Then, for all $R<\min\{ R_0, \tfrac{1}{\sqrt{K}}\}$ there holds in $B_R(p)$:
\begin{equation*}
    |\langle\nabla_{V}(d\nabla d),V\rangle - |V|^2| \le \sqrt{K}R|V|^2  \,,
\end{equation*}
for every vector field $V$ on $M$.
\end{lemma}
\begin{proof}
We can compute
\begin{align*}
    \lp V, \nabla _{V} (d\nabla d) \rp  &=  \lp V, \lp V,\nabla d\rp \nabla d  \rp +  d \lp V, \nabla_{V} (\nabla d) \rp  \\ &  =  \lp V , \nabla d \rp^2+  d \, \nabla^2 d(V , V ) \,.
\end{align*}
On the other hand, the Hessian Comparison theorem---see Lemma 7.1 in \cite{ColdingMin}---gives that
\begin{equation*}
   | d \, \nabla^2 d (V,V)- |V- \lp V, \nabla d \rp \nabla d|^2| \le d \sqrt{K} |V|^2
\end{equation*}
in $B_R(p)$, whenever $ R < \min \{\inj_{M}(p), \frac{1}{\sqrt{K}}\}$. Moreover, since $|\nabla d|^2 =1 $, we also have that
\begin{align*}
    |V- \langle V, \nabla d \rangle \nabla d|^2 = |V|^2-2 \, \lp V, \nabla d \rp ^2 + \lp V, \nabla d \rp ^2 |\nabla d|^2 = |V|^2-  \lp V, \nabla d \rp ^2 .
\end{align*}
Hence 
\begin{equation*}
   | d \, \nabla^2 d (V,V)+ \lp V, \nabla d \rp^2 - |V|^2  | \le d \sqrt{K} |V|^2 \le R\sqrt{K}|V|^2 
\end{equation*}
holds in $B_R(p)$, as long as $ R < \min \{ R_0, \frac{1}{\sqrt{K}}\}$, and this conludes the proof.
\end{proof}

\begin{lemma}\label{divestlem}
Let $(M^n,g)$ be an $n$-dimensional Riemannian manifold, $p\in M$, $R_0 < \inj_M(p)$ and let $K$ be an upper bound for all the sectional curvatures in $B_{R_0}(p)$. Then, there exists $ C=C(n) > 0 $ such that, for all $ R < R_0$, in $B_R(p)$ we have that
\begin{equation*}
    |\textnormal{div}(d\nabla d)-n| \le C K R^2 \,.
\end{equation*}
\end{lemma}
\begin{proof} 
Fix $p\in M$, and denote $d(p,\cdot)$ just by $d(\cdot)$. Observe first that every geodesic $\sigma$ with $\sigma(0)=p$ and contained in $B_{R_0}(p)$ is uniquely minimizing. For any $R< R_0$ and $x \in B_{R}(p)$, let $\gamma : [0,d ] \to M $ be the normalized geodesic with $\gamma(0)=p$ and $\gamma(d) = x$. Note also that 
\begin{align*}
    \dive(d\nabla d)= |\nabla d|^2 + d \Delta d  = 1 +d\Delta d \,.
\end{align*}
Consider $\dot{\gamma}(d) \in T_{x}M$, and complete it to an orthonormal basis $ \{e_1:=\dot{\gamma}(d), e_2, \dotsc , e_n \}$ of $T_x M $. For $i=2, 3, \dotsc,n$, let $\gamma_i $ be the geodesic with $\gamma_i(0)=x$ and $\dot{\gamma}_i(0)=e_i$. We can compute 
\begin{align*}
    \Delta d(x) = \sum_{i=1}^n \nabla^2 d(x)(e_i, e_i) = \sum_{i=1}^n \frac{d^2}{ds^2}\bigg |_{s=0} (d \circ \gamma_i ) = \sum_{i=2}^n \frac{d^2}{ds^2}\bigg |_{s=0} (d \circ \gamma_i ) \,,
\end{align*}
where we have used that $\frac{d^2}{ds^2}\bigg |_{s=0} (d \circ \gamma ) = \frac{d^2}{ds^2}\bigg |_{s=0} (d(x) +s )=0$.\\
Let $J_i$ be the Jacobi field along $\gamma $ with $J_i(0)=0$ and $J_i(d)=e_i$, well defined by uniqueness of geodesics between endpoints. Denote by 
\begin{align*}
    I(X,Y)= \int_0^d \lp D_t X, D_t Y \rp - \textnormal{Rm}(\dot{\gamma}, X, \dot{\gamma}, Y) \, dt
\end{align*}
the index form associated to $\gamma$ on $[0,d]$. Since $\gamma$ is minimizing along all curves with the same endpoints, for every vector field $X$ on $\gamma([0,d])$ orthogonal to $\dot{\gamma}$ and with $X(0)=0$ and $X(d)=e_i$ we must have 
\begin{align*}
    0 \le I(J_i-X, J_i-X) = I(J_i,J_i) -2I(J_i,X) + I(X,X) \,.
\end{align*}
Since $J_i$ is a Jacobi field, one can easily check that $I(J_i,X)=I(J_i,J_i)$, hence $I(J_i, J_i) \le J(X,X) $. Take $X(t)=\frac{t}{d}E_i(t)$, where $E_i(t)$ is the parallel transport of $e_i \in T_x M$ along $\gamma$. From the second variation formula for arc length we get
\begin{align*}
    \frac{d^2}{ds^2}\bigg |_{s=0} (d \circ \gamma_i ) &= \int_{0}^d |D_t J_i|^2 - \textnormal{Rm}(\dot{\gamma}, J_i, \dot{\gamma}, J_i) \, dt = I(J_i, J_i) \\ & \le I(X,X) = \int_{0}^d |D_t X|^2 - \textnormal{Rm}(\dot{\gamma}, X, \dot{\gamma}, X) \, dt \\ &\le \int_{0}^d |D_t X|^2 + K|X|^2 \, dt \,,
\end{align*}
where we have used that $\sup_{p\in B_{R_0}}|\textnormal{Sec}_p| \le K$. Thus 
\begin{align*}
     \frac{d^2}{ds^2}\bigg |_{s=0} (d \circ \gamma_i ) \le \int_{0}^d |D_t X|^2 + K|X|^2 \, dt = \int_{0}^d \frac{1}{d^2}+K\frac{t^2}{d^2} \, dt = \frac{1}{d} \left(1+ K\frac{d^2}{3} \right).
\end{align*}
Hence
\begin{align*}
    d \Delta d = \sum_{i=2}^n \frac{d^2}{ds^2}\bigg |_{s=0} (d \circ \gamma_i ) \le n-1+ K\frac{n-1}{3}d^2 \,,
\end{align*}
or equivalently
\begin{align*}
    |\dive(d\nabla d)(x) -n | = |d(x)\Delta d(x) +1 -n | \le K\frac{n-1}{3}d^2 \le K\frac{n-1}{3} R^2 ,
\end{align*}
and this completes the proof with $C(n)=\frac{n-1}{3}>0$.
\end{proof}

We can now prove the monotonicity formula.

\begin{proof}[Proof of Theorem \ref{monfor}]

Since during the entire proof, the point $ p_\circ \in M$ will be fixed, we will not specify the center of the balls in what follows, as this will always be $(p_\circ,0)$ for balls inside $\ov{M}$ and $p_\circ$ for balls on $M$. We divide the proof into two steps.

\vsp
\textbf{Step 1.} First, we show that if $u$ is stationary for the energy $ \mathcal E(v)=[v]^2_{H^{{s/2}}(M)}+\int_M F(v)$ under inner variations, then its Caffarelli-Silvestre extension $U$ is stationary for the energy 
\begin{equation*}
U \mapsto \frac{\beta_s}{2} \int_{\widetilde M} z^{1-s} |\nabla U|^2 \, dVdz+\int_{M}F(U|_{M})\,dV \,,    
\end{equation*}
under inner variations on $\widetilde M$ given by vector fields $Y$ on $\widetilde M$ such that $Y|_M$ is tangent to $M$.

\vsp 
Recall that the Caffarelli-Silvestre extension of $u$ is given by \eqref{caffextMfd}. 

\vsp
Let $Y$ be a vector field on $\widetilde M$ such that $Y|_{M}$ is tangent to $M$, and let $\psi_Y^t$ denote its flow at time $t$. Let also $V_{t}$ be the Caffarelli-Silvestre extension of $u\circ\psi_{Y|_{M}}^{t}$, for any $t\in \R$. By the minimality of the extension in the energy space, we have 
\begin{align*}
    \frac{d}{dt}\Big{|}_{t=0} \frac{\beta_s}{2} \int_{\widetilde M} z^{1-s} |\nabla (U\circ\psi_Y^{-t})|^2 & = \lim_{t\to 0} \frac{1}{t}\left( \frac{\beta_s}{2} \int_{\widetilde M} z^{1-s} |\nabla  U|^2- \frac{\beta_s}{2} \int_{\widetilde M} z^{1-s} |\nabla (U\circ\psi_Y^{t})|^2 \right) \\
    &\leq \lim_{t\to 0} \frac{1}{t}\left( \frac{\beta_s}{2}\int_{\widetilde M} z^{1-s} |\nabla  U|^2- \frac{\beta_s}{2} \int_{\widetilde M} z^{1-s} |\nabla  V_t|^2 \right) \\
    &=\lim_{t\to 0} \frac{ [u]^2_{H^{{s/2}}(M)}-[ u\circ\psi_Y^{t}]^2_{H^{{s/2}}(M)}}{t}\\
    &=\frac{d}{dt}\Big{|}_{t=0}  [ u\circ\psi_Y^{-t}]^2_{H^{{s/2}}(M)} \,,
\end{align*}
and likewise 
\begin{align*}
    \frac{d}{dt}\Big{|}_{t=0} \frac{\beta_s}{2}\int_{\widetilde M} z^{1-s} |\nabla (U\circ\psi_Y^{-t})|^2 & = \lim_{t\to 0}\frac{1}{t}\left( \frac{\beta_s}{2} \int_{\widetilde M} z^{1-s} |\nabla (U\circ\psi_Y^{-t})|^2- \frac{\beta_s}{2}\int_{\widetilde M} z^{1-s} |\nabla  U|^2\right) \\
    &\geq \lim_{t\to 0} \frac{1}{t}\left( \frac{\beta_s}{2} \int_{\widetilde M} z^{1-s} |\nabla  V_{-t}|^2- \frac{\beta_s}{2} \int_{\widetilde M} z^{1-s} |\nabla  U|^2\right)\\
    &=\lim_{t\to 0} \frac{ [u\circ\psi_Y^{-t}]^2_{H^{{s/2}}(M)}-[u]^2_{H^{{s/2}}(M)}}{t}\\
    &=\frac{d}{dt}\Big{|}_{t=0} [ u\circ\psi_Y^{-t}]^2_{H^{{s/2}}(M)} \,.
\end{align*}
Hence
\begin{align*}
    \frac{d}{dt}\Big{|}_{t=0} \frac{\beta_s}{2} \int_{\widetilde M} z^{1-s} |\nabla (U\circ\psi_Y^{-t})|^2=\frac{d}{dt}\Big{|}_{t=0} [ u\circ\psi_Y^{-t}]^2_{H^{{s/2}}} \, .
\end{align*}
Since $u$ is stationary for the energy $ \mathcal E(v)=[v]^2_{H^{{s/2}}(M)}+\int_M F(v) \,dV$ under inner variations, this shows that $U$ is stationary for the energy $ U \mapsto \frac{\beta_s}{2} \int_{\widetilde M} z^{1-s} |\nabla U|^2 \, dVdz+\int_{M}F(U|_{M}) \,dV$ under inner variations on $\widetilde M$, with vector fields $Y$ as above, and this concludes the first step.

\vsp
\textbf{Step 2.} We now compute such an inner variation for a suitably chosen $Y$. First, the variation of the potential part of the energy is
\begin{align}
    \frac{d}{dt} \Big|_{t=0}\int_M F(u\circ\psi_Y^{-t}) \, dV&=\frac{d}{dt} \Big|_{t=0}\int_M F(u)J_t(p) \, dV_p\nonumber\\
    &=\int_M F(u)\text{div}_g (Y|_M) \, dV \label{divW} .
\end{align}
The quantity $\text{div}_g(Y|_M)$ will be estimated later. We now focus on computing the variation for the Sobolev part of the energy. Once again, we change variables in the integral using the flow $\psi_Y^t$, obtaining 
\begin{align}\label{sjdhfhsd}
\int_{\widetilde M} z^{1-s} |\nabla (U\circ\psi_Y^{-t})|^2 \, dVdz &=\int_{\widetilde M} (z\circ\psi_Y^t )^{1-s} |\nabla (U\circ\psi_Y^{-t})|^2 \circ \psi_Y^{t} \,J_t(p,z)\,dV_pdz \,.
\end{align}
Now, we choose the vector field $Y$. We take $Y=\eta(d) d\nabla  d$, where $d=d_{\widetilde g}((p_\circ, 0), \, \cdot \,)$ is the distance on $\widetilde M$ from the point $(p_\circ,0)$ and $\eta=\eta_\delta$ is a single variable smooth function with $\eta \equiv 1$ on $[0,R]$, decreasing to zero on $[R,R+\delta]$, and $\eta \equiv 0$ on $[R+\delta,+\infty)$. Since the distance $d_{\widetilde g}((p_\circ, 0), \, \cdot \,)$ restricts to the distance $d_g(p_\circ , \, \cdot \,)$ on $M$ when computed on points on $ \widetilde M $ with $z=0$, clearly $Y|_M$ is tangent to $M$. We want to exchange the order of derivation and integration in \eqref{sjdhfhsd}. Hence, we compute separately the three terms that will appear in doing so. For the first term, using that $d_{\widetilde g}^2((p,z), (p_\circ,0)) = d_g^2(p, p_\circ)+z^2$ and the definition of $Y$ we see that 
$$
\frac{d}{dt}\Big|_{t=0}(z\circ\psi_Y^{t})^{1-s}=(1-s)z^{-s}\eta(d) z=(1-s)z^{1-s}\eta(d)\, .
$$
As for the second term that will appear, a simple general computation ---see for example the lines after Lemma 3.1 in \cite{Gas20}--- shows that
\begin{equation*}
    \frac{d}{dt}\Big|_{t=0}|\nabla (U\circ\psi_Y^{-t})|^2(\psi_Y^{t}(x))=-2\langle\nabla _{\nabla  U}Y,\nabla  U\rangle\, .
\end{equation*}
Moreover, using the form chosen for $Y$ we have that
\begin{align*}
    \langle\nabla _{\nabla  U}Y,\nabla  U\rangle&=\langle\nabla _{\nabla  U}(\eta(d) d\nabla  d),\nabla  U\rangle\\
    &=\langle\nabla  U,\nabla \eta(d)\rangle\langle d\nabla  d,\nabla  U\rangle+\eta(d)\langle\nabla _{\nabla  U}(d\nabla  d),\nabla  U\rangle\\
    &=\eta'(d)\langle\nabla  U,\nabla  d\rangle\langle d\nabla  d,\nabla  U\rangle+\eta(d)\langle\nabla _{\nabla  U}(d\nabla  d),\nabla  U\rangle\\
    &=d\eta'(d)|\langle\nabla  U,\nabla  d\rangle|^2+\eta(d)\langle\nabla _{\nabla  U}(d\nabla  d),\nabla  U\rangle\, .
\end{align*}
Notice that $K$ is also an upper bound for all the sectional curvatures on $\ov{M}$ in $\ov{B}_{R_\circ}(p_\circ,0)$ and that $\inj_{M}(p_\circ)=\inj_{\ov{M}}(p_\circ,0)$. Thus, by Lemma \ref{lemmaaux11} applied to $V=\nabla U$,
\begin{align*}
    \langle\nabla _{\nabla  U}Y,\nabla  U\rangle
    &=d\eta'(d)|\langle\nabla  U,\nabla  d\rangle|^2+\eta(d)(1+O(\sqrt{K}R))|\nabla  U|^2
\end{align*} 
for all $R< \min \left\{ R_\circ , \tfrac{1}{\sqrt{K}} \right\} $.
Lastly, for the remaining factor in the integral, Lemma \eqref{divestlem} gives that
\begin{align*}
    \frac{d}{dt}\Big|_{t=0} J_t =\widetilde\dive (Y) &=\eta'(d)d|\nabla  d|^2 + \eta(d)\widetilde \dive (d\nabla  d) \\ 
    & = d\eta'(d)+\eta(d)(n+1)(1+O(\sqrt{K}R))\, ,
\end{align*}
in $B_R(p_\circ)$, for $R<\min \left\{ R_\circ, \tfrac{1}{\sqrt{K}} \right\}$. 

\vsp
Now, analogously applying Lemma \eqref{divestlem} on $M$ instead of $\widetilde M$ to \eqref{divW}, we already find an estimation for the potential energy:
\begin{align*}
    \frac{d}{dt} \Big|_{t=0}\int_M F(u\circ\psi_Y^{-t}) \, dV&=\int_M F(u)\big(d\eta'(d)+\eta(d)n(1+O(\sqrt{K}R))\big) \, dV\, .
\end{align*}
Moreover, it follows from (the local version of) Bonnet-Myers' theorem that $R_\circ < R_{\rm max}:= \inj_{M}(p_\circ)/4 < \min \left\{ \inj_{M}(p_\circ), \tfrac{1}{\sqrt{K}} \right\}$, and this will be our final choice of $R_{\rm max}$ for the statement. From now on, we always consider $R<R_\circ \le R_{\rm max}=\inj_{M}(p_\circ)/4$.\\
Regarding the Sobolev part of the energy, exchanging differentiation and integration and substituting the estimates we have obtained so far gives:
\begin{align*}
    \frac{d}{dt}\Big{|}_{t=0}  & \int_{\widetilde M} z^{1-s} |\nabla (U\circ\psi_Y^{-t})|^2 \\ & = \int_{\widetilde M} (1-s)z^{1-s}\eta(d)|\nabla  U|^2 + z^{1-s} 
    \big(-2d\eta'(d)|\lp \nabla  U, \nabla  d \rp|^2 -2 \eta(d)(1+O(\sqrt{K}R))|\nabla  U|^2 \big) \\ &  \hspace{0.6cm}+ \int_{\widetilde M }z^{1-s} |\nabla  U |^2 \big(d\eta'(d)+\eta(d)(n+1)(1+O(\sqrt{K}R)) \big) \, dVdz \\ &= (n-s)(1+O(\sqrt{K}R))\int_{\widetilde B_{R+\delta}^+ } z^{1-s} |\nabla  U|^2 \eta(d) + \int_{ {\widetilde B_{R+\delta}^+ } \setminus {\widetilde B_{R}^+ } } z^{1-s} d\eta'(d) \big( |\nabla  U|^2-2|\lp \nabla  U, \nabla  d \rp|^2 \big) .
\end{align*} 
Adding the expressions for the potential and Sobolev parts of the energy, we get
\begin{align*}
    \frac{d}{dt}\bigg{|}_{t=0}  \bigg( \frac{\beta_s}{2} \int_{\widetilde M}  z^{1-s} & |\nabla (U\circ\psi_Y^{-t})|^2  + \int_M F(u\circ\psi_Y^{-t}) \bigg) \\ &=(n-s) (1+O(\sqrt{K}R)) \frac{\beta_s}{2} \int_{\widetilde B_{R+\delta}^+ } \eta(d) z^{1-s}|\nabla  U|^2\\
    &\hspace{0.4cm}+ \frac{\beta_s}{2} \int_{{\widetilde B_{R+\delta}^+ } \setminus {\widetilde B_{R}^+ }} d\eta'(d) z^{1-s}(|\nabla  U|^2-2 \langle\nabla  U,\nabla  d\rangle^2 )\\
    &\hspace{0.4cm}+n(1+O(\sqrt{K}R))\int_{B_{R+\delta}} \eta(d) F(u)+\int_{B_{R+\delta} \setminus B_R} d\eta'(d) F(u) \,.
\end{align*}
By stationarity of $u$ and Step 1 we know that the left-hand side is equal to $0$ for every $Y$, thus the right-hand side vanishes for all $\eta = \eta_\delta$ defined as above. Since this holds for all $\delta>0$, we now let $\delta \searrow 0$ so that $\eta_\delta$ converges to the characteristic function of $[0,R]$. This gives (for a.e. $R\in (0, R_\circ)$)
\begin{align*}
    0=&(n-s)(1+O(\sqrt{K}R)) \frac{\beta_s}{2} \int_{\widetilde B ^+_R} z^{1-s}|\nabla  U|^2-R \frac{\beta_s}{2}\int_{\partial \widetilde B ^+_R} z^{1-s}|\nabla  U|^2 +2R\frac{\beta_s}{2}\int_{\partial^+ \widetilde B ^+_R} (\partial_\nu U)^2\\
    &+n(1+O(\sqrt{K}R))\int_{B_R} F(u)-R\int_{\partial B_R} F(u).
\end{align*}
Rearranging the terms and multiplying by $R^{-n+s-1}$, we deduce that

\begin{align*}
     -\frac{(n-s)}{R^{n-s+1}} & \left( \frac{\beta_s}{2} \int_{\widetilde B^+_R} z^{1-s }|\nabla  U|^2 + \int_{B_R} F(u) \right) + \frac{1}{R^{n-s}} \left( \frac{\beta_s}{2} \int_{\partial \widetilde B^+_R} z^{1-s }|\nabla  U|^2 + \int_{\partial B_R} F(u) \right) \\ &\ge -\frac{C\sqrt{K}}{R^{n-s}} \left( \frac{\beta_s}{2} \int_{\widetilde B^+_R} z^{1-s }|\nabla  U|^2 + \int_{B_R} F(u) \right) +\frac{\beta_s}{R^{n-s}}\int_{\partial^+ \widetilde B^+_R} z^{1-s} \langle\nabla  U,\nabla  d\rangle^2 +\frac{s}{R^{n-s+1}} \int_{B_R} F(u) \,,
\end{align*}
for some absolute constant $C > 0$. In other words,
\begin{align*}
    \Phi'(R) \ge - C \sqrt{K} \Phi(R)  + \frac{ \beta_{s}}{R^{n-s}} \int_{\partial^+ \ov{B} _R^+ } z^{1-s} \langle\nabla  U,\nabla  d\rangle^2 +\frac{ s }{R^{n-s+1}} \int_{B_R} F(u) \, dV \,,
\end{align*}
and this implies, in particular, that
\begin{equation*}
\frac{d}{dR} 
\Big(e^{C\sqrt{K}R} \Phi(R)\Big)\geq 0 \quad \mbox{for all} \s R<  R_\circ .
\end{equation*}

\vsp
Lastly, in the case where $M=\R^n$, $F\equiv 0$, $s\in(0,1)$ and $u=\chi_E-\chi_{E^c}$ is a stationary set for the fractional $s$-perimeter, instead of the two bounds used above
\begin{align*}
     \langle\nabla _{\nabla  U}(d\nabla  d),\nabla  U\rangle &=(1+O(\sqrt{K}R))|\nabla  U|^2 \,, \\  \widetilde{\dive}(d\nabla  d) &= (n+1)(1+O(\sqrt{K}R)) \,,
\end{align*}
given respectively by Lemmas \ref{lemmaaux11} and \ref{divestlem}, one has the equalities
\begin{align*}
     \langle\nabla _{\nabla  U}(d\nabla  d),\nabla  U\rangle_{\R^{n+1}} &=|\nabla  U|^2 \,, \\  \dive_{\R^{n+1}}(d\nabla  d) &= n+1 \,,
\end{align*}
where $U$ is the extension of $u=\chi_E-\chi_{E^c}$. Thus, following the proof one finds the exact expression
\begin{equation*}
    \Phi'(R) = \frac{\beta_{s}}{R^{n-s}} \int_{\partial^+ \ov{\B} _R^+ (p_\circ,0)} z^{1-s} \langle\nabla  U,\nabla  d\rangle^2 \, dxdz \ge 0 \,.
\end{equation*}
In particular, $\Phi$ is constant if and only if $\langle\nabla  U,\nabla  d\rangle =0$, that is, if and only if $E$ is dilation-invariant for dilations with center at $p_\circ \in \R^n$. With this, we conclude the proof.
\end{proof}

% We will always use the overline to distinguish balls on $\ov{M}$ from the ones in $M$. The following lemma is an extension to Riemannian manifolds of Proposition 7.1 in [CRS], the proof can be found in the reference in the euclidean case.
% \begin{lemma}\label{controlupdown} Let $u\in H^{s/2}(M)$, and let $U$ be its extension as given by the solution of (\ref{caffextMfd}). There is a constant $R_0$ depending only on $M$ such that, if $p\in M$ and $R<R_0$ then
% \begin{equation*}
%     \int_{\widetilde B_R^+(p, 0)}z^{1-s}|\nabla U|^2\leq C_1\iint_{(M \times M)\setminus (B_{2R}^c(p) \times B_{2R}^c(p))} 
%   |u(x)-u(y)|^2 K_s(x,y) \,dV_xdV_y \,,
%  \end{equation*}
%  and 
%  \begin{equation*}
%   \iint_{(M \times M)\setminus (B_R^c(p) \times B_R^c(p))} 
%   |u(x)-u(y)|^2 K_s(x,y) \,dV_xdV_y \leq     \frac{C_2}{1-s}\Big{(}R^{n-s}+\int_{\widetilde B_{2R}^+(p, 0)} z^{1-s}|\nabla U|^2\Big{)},
% \end{equation*}
% where $C_1=C_1(M)$ and $C_2=(M, \| u\|_{L^{\infty}}) $.
% \end{lemma}

\subsection{The fractional Sobolev energy under inner variations}
    
We next study how the fractional Sobolev energy behaves under inner variations. For this, we need first to study how the singular kernel $K_s$ behaves when translating its arguments under the flow of a vector field.
\begin{proposition}\label{Ker2est} 
Let $(M,g)$ be a closed $n$-dimensional Riemannian manifold and $s\in(0,2)$. Consider any smooth vector field $X \in \mathfrak{X}( M)$, and fix points $p,q\in M$. Writing $\psi^t$ for the flow of $X$ at time $t$, then the kernel satisfies
\begin{equation}\label{kerd2est}
    \bigg| \frac{d^\ell}{dt^\ell}\bigg|_{t=0} K_s(\psi^t(p),\psi^t(q)) \bigg|\leq C(1 +K_s(p,q)).
\end{equation}
for some constant $C=C(M,s,\ell, \max_{0\le k\le \ell} \|\nabla^k X\|_{L^\infty(M)} )$ which stays bounded for $s$ away from $0$ and $2$.
\end{proposition}
\begin{proof}
  This follows from the estimates of Theorem \ref{prop:kern1}, in particular by \eqref{remaining} and \eqref{remaining2}. We prove the result just for $\ell=1$, as the general case just follows by induction by the very same arguments. Let $R=R(M)>0$ be such that the flatness assumption ${\rm FA}_\ell (M,g,16R, p, \varphi_p) $ holds for every $p\in M$; such an $R$ exists by Remark \ref{fbsvdg}.  We split in two cases. 

  \vsp \textbf{Case 1:} $q \in \varphi_p(\B_{4R}(0))$. 

  \vsp \noindent In this case, denoting $K(x,y) := K_s(\varphi_p(x), \varphi_p(y))$ and $k(x, z): = K(x,x+z)$ as in Theorem \ref{prop:kern1}, we have that 
  \begin{equation*}
      K_s(\psi^t\circ\varphi_p(x),\psi^t\circ \varphi_p(y)) = K(\psi_p^t(x) , \psi_p^t(y)) = k (\psi_p^t(x), \psi_p^t(y)-\psi_p^t(x))\,,
  \end{equation*}
  where $\psi_p^t$ is the flow of $\xi = (\varphi_p)^* X$, i.e. the vector field $\xi=\xi_p\in \mathfrak{X} (\B_{16R}(0))$ such that $X\circ \varphi_p = (\varphi_p)_* \xi$.  Then, for all $x,y\in \B_{2R}(0)$ we have:
  \begin{align*}
      \frac{d}{dt}\bigg|_{t=0} K(\psi_p^t(x) , \psi_p^t(y)) &=  \frac{d}{dt}\bigg|_{t=0} k ( \psi_p^t(x), \psi_p^t(y)-\psi_p^t(x)) \\ &= \frac{\partial k}{\partial x^\alpha}(x,y-x) \xi^\alpha (x) + \frac{\partial k}{\partial z^\alpha}(x,y-x) (\xi^\alpha (x)  - \xi^\alpha (y) ) \,,
  \end{align*}
  where sum over repeated indices is assumed. Hence, by \eqref{remaining} of Theorem \ref{prop:kern1} we get
  \begin{align*}
        \left|\frac{d}{dt}\bigg|_{t=0} K(\psi_p^t(x) , \psi_p^t(y)) \right| & \le \frac{C}{|y-x|^{n+s}} \|\xi \|_{L^{\infty}} + \frac{C}{|y-x|^{n+s+1}} \|D \xi \|_{L^{\infty}}|y-x| \\ & \le \frac{C}{|y-x|^{n+s}} \le C K(x,y)
  \end{align*}
  for some $C=C(n,s, \|\xi\|_{C^{0,1}})$, where in the last line we have also used Lemma \ref{loccomparability}. 
  Finally, evaluating this inequality at $x=0$ and $y= \varphi_p^{-1}(q)$ we obtain 
  \[
  \left|\frac{d}{dt}\bigg|_{t=0} K(\psi^t(p) , \psi^t(q))\right|= \left|\frac{d}{dt}\bigg|_{t=0} K(\psi_p^t(0) , \psi_p^t(y)) \right|  \le C K(0,y) = C K_s(p,q),
 \]
as wanted.

  \vsp \textbf{Case 2:} $q \notin \varphi_p(\B_{4R}(0))$. Then ${\rm FA}_\ell (M,g,R,q, \varphi_q)$ holds and the sets $\varphi_p(\B_{R}(0))$ and $\varphi_q(\B_{R}(0))$ are disjoint. Hence, by Proposition \ref{prop:kern2} the kernel $K_{pq}(x,y): =  K_s(\varphi_p(x), \varphi_q(y))$ is smooth (with uniform estimates on all derivatives) in the domain $\B_{R/2}(0)\times \B_{R/2}(0)$. Hence
  \begin{align*}
      \frac{d}{dt}\bigg|_{t=0} K_s(\psi^t\circ\varphi_p(x),\psi^t\circ\varphi_q(x)) & = \frac{d}{dt}\bigg|_{t=0} K_{pq}(\psi_p^t(x),\psi_p^t(y)) \\ &= \frac{\partial K_{pq}}{\partial x^{\alpha}} (x,y) \xi_p^{\alpha}(x) + \frac{\partial K_{pq}}{\partial y^{\alpha}} (x,y) \xi_q^{\alpha}(y) .
  \end{align*}
  Using Proposition \ref{prop:kern2} to bound the derivatives of $K_{pq}$, and then evaluating at $(x,y)=(0,0)$ gives
  \begin{equation*}
      \left|\frac{d}{dt}\bigg|_{t=0} K_s(\psi^t(p),\psi^t(q)) \right| \le \frac{C}{R^{n+s}} \,,
  \end{equation*}
  for some $C=C(n,s, \| \xi_p \|_{L^{\infty}}, \| \xi_q \|_{L^{\infty}})$. 

  \vsp
  Putting together the two cases above, we get
  \begin{equation*}
      \left|\frac{d}{dt}\bigg|_{t=0} K_s(\psi^t(p),\psi^t(q)) \right| \le C(1+K_s(p,q)) \,, 
  \end{equation*}
  for some $C=C(M, n, s, \| X \|_{C^1(M)})$ and conclude the proof.
    
\end{proof}

%Here $C(r,X)=C_0f_r(\| X\|_{L^\infty},\|\nabla X\|_{L^\infty},...,\|\nabla^r X\|_{L^\infty})$, where $C_0$ %depends only on $M$, $r$, $\sg_0$, and $\sg_1$, and $f_r(z_0, z_1, \dotsc, z_r)$ is a polynomial that has degree %$(r+1-p)$ in the variable $z_p$, for all $1\le p \le r$. {\color{green} The correct property is the one that was %written before. It cannot have degree $2=r+1-(r-1)$ on the variable $z_{r-1}$, for instance, as this gives order %$2(r-1)$ when rescaling the vector field $X$, higher than $r$. It can have, on the other hand, degree $1$ on %$z_{r-2}$ and degree $1$ on $z_2$, as this ends up giving a scaling of order $(r-2)+(2)=r$.}\\
%{\color{green}
%The function is $f_r(x_0,..., x_r)=\sum_{1\leq i_0+1i_1+2i_2+...+r i_r \leq r} x_0^{i_0}x_1^{i_1}...x_r^{i_r}$. %The complicated expression we give for $C_r(X)$ simply means that it is of order at most $r$ in the derivatives %of $X$. For example, $C_2(X)$ can have factors $\|\nabla X\|_{L^\infty}^2$ and $\|\nabla^2 X\|_{L^\infty}$, %which scale with order $2$ when rescaling $X$, but it cannot contain a factor $\|\nabla^2 X\|_{L^\infty}^2$ as %this scales with order 4 when rescaling $X$. How to say this in a simpler manner?}

    We also record a version of Proposition \ref{Ker2est}, which depends only on local quantities:

    \begin{proposition}\label{sdfgsdrgs}
        Let $(M,g)$ be a closed $n$-dimensional Riemannian manifold and $s\in(0,2)$. Assume that the flatness assumption ${\rm FA}_\ell (M,g,R,p,\varphi)$ holds, and let $X\in \mathfrak{X}( M)$ be a smooth vector field supported on $\varphi(\B_{R/4})$. Writing $\psi^t$ for the flow of $X$ at time $t$, then for every $x,y \in \B_{R/4}(0)$ we have
        \begin{equation}\label{cdsacsdc1}
    \bigg| \frac{d^\ell}{dt^\ell}\bigg|_{t=0} K_s(\psi^t(\varphi(x)),\psi^t(\varphi(y))) \bigg|\leq C K_s(\varphi(x), \varphi(y)) \le  C\frac{\alpha_{n,s}}{|x-y|^{n+s}}\,,
\end{equation}
for some constant $C=C(n,s, \|X \|_{C^\ell(\varphi(\B_R))})$. Moreover, given $T>0$ we have that, for all $0\le t \le T$,
\begin{equation}\label{cdsacsdc2}
    \bigg| \frac{d^\ell}{dt^\ell}  K_s(\psi^t(\varphi(x)),\psi^t(\varphi(y))) \bigg|\leq C_T K_s(\varphi(x), \varphi(y)) \le  C_T\frac{\alpha_{n,s}}{|x-y|^{n+s}}\,,
\end{equation}
where $C_T=C_T(n,s, T, \|X \|_{C^\ell(\varphi(\B_{R/4}))})$.\\
The constants stay bounded for $s$ away from $0$ and $2$.
    \end{proposition}
    \begin{proof}
By scaling, we can assume $R=1$. The second inequality in both \eqref{cdsacsdc1} and \eqref{cdsacsdc2} then follows from Lemma \ref{loccomparability}. As for the first inequality of \eqref{cdsacsdc1}, it follows from the proof of Case $1$ in Proposition \ref{Ker2est}, since it only depends on local estimates for $X$. Finally, \eqref{cdsacsdc2} can be deduced from \eqref{cdsacsdc1}. Indeed, note that for all $1 \le k \le \ell $ and $0\le t\le T$,
\begin{equation}\label{stghsdfgh}
    \left|\frac{d^k}{d t^k} K_s(\psi^t(\varphi(x)),\psi^t(\varphi(y)))  \right| = \left|\frac{d^k}{d r^k} \bigg|_{r=0} K_s(\psi^{t+r}(\varphi(x)),\psi^{t+r}(\varphi(y))) \right| \leq  C_0 K_s(\psi^t(\varphi(x)),\psi^t(\varphi(y)))  \,,
\end{equation}
with $C_0=C_0(n,s, \|X \|_{C^\ell(\varphi(\B_1))})$. Thus, we are only left with proving that
\begin{equation*}
    K_s(\psi^t(\varphi(x)),\psi^t(\varphi(y))) \le C_T K_s(\varphi(x),\varphi(y))
\end{equation*}
for some $C_T=C_T(n,s, T, \|X \|_{C^\ell(\varphi(\B_1))})$. But this follows itself from \eqref{stghsdfgh}, with $k=1$, since we can write the inequality as
$$
\frac{d}{dt} \big[e^{-C_0t}K_s(\psi^t(\varphi(x)),\psi^t(\varphi(y)))\big]\leq 0\, ,
$$
and integrating we find that
\begin{equation*}
    K_s(\psi^t(\varphi(x)),\psi^t(\varphi(y))) \le e^{C_0 T} K_s(\varphi(x),\varphi(y))
\end{equation*}
for every $0\le t \le T$. 
 \end{proof}

Proposition  \ref{Ker2est} can be used to bound time derivatives of the energy of ``flown objects'' by their energy at time zero. We show this for the fractional Sobolev energy:
\begin{lemma}\label{enboundslemma}
Let $s \in (0,2)$ and $ v\in H^{s/2}(M)$ be a function with $|v|\leq 1$. Let $X \in \mathfrak{X} (M)$ be a smooth vector field and $v_t := v \circ \psi^{-t}$, where $\psi^t$ is the flow of $X$ at time $t$. Then, for all $T>0$ there holds
\begin{equation*}
    \sup_{0<t<T} \bigg| \frac{d^\ell}{d  t^\ell}\mathcal{E}_{M}(v_t) \bigg|\leq C\big(1+ \mathcal{E}_{M}(v)\big)\, ,
\end{equation*}
for some constant $C=C(M,s,\ell,T, \max_{0\le k\le \ell} \|\nabla^k X\|_{L^\infty(M)} )$ which stays bounded for $s$ away from $0$ and $2$..
\end{lemma}
%\begin{lemma}\label{enboundslemmaloc}
%Let $v\in H^{s/2}(M)$ be a function with $|v|\leq 1$. Let $X$ be a vector field on $M$, supported on $\varphi(B_r)\subset\varphi(B_1)$ {\color{green} Flatness assumptions}. Then there holds
%\begin{equation*}
  %  \sup_{0<t<T} \bigg| \frac{d^\ell}{d  t^\ell}\mathcal{E}_{M}(v_t) \bigg|\leq C(r^{n-s}+\mathcal{E}_{\varphi(B_r)}(v))\, ,
%\end{equation*}
%where $C$ depends only on $s$, $k$, and $\max_{0\le \ell\le k} \|\nabla^\ell X\|_{L^\infty(M)}$ (but not on $M$). {\color{green} State dependency as being on $X$ and its derivatives in the chart as well as on the derivatives of the metric??}
%\end{lemma}

\begin{proof}
Let $C$ denote a constant that depends only on $M$, $s$, $\ell$, $T$ and  $ \max_{0\le k\le \ell} \|\nabla^k X\|_{L^\infty(M)}$. 

The idea of the proof is to change variables using the flow $\psi^t$ in the corresponding integrals defining the Allen-Cahn energy, and after that to exchange integration and differentiation. 

Let us start with the Sobolev part of the energy. We have, denoting by $J_t$ the Jacobian of the flow:
\begin{align}
\frac{d^\ell}{d t^\ell}\mathcal{E}^{\rm Sob}_{M}({v_t})&=\frac{d^\ell}{d t^\ell}\iint |v(\psi^{- t}(p))-v(\psi^{- t}(q))|^2K_s(p,q) \, dV_p\, dV_q\nonumber\\
&=\frac{d^\ell}{d t^\ell}\iint |v(p)-v(q)|^2 K_s(\psi^{ t}(p),\psi^{ t}(q))\,J_{ t}(p)J_{ t}(q) \, dV_p\, dV_q \nonumber  \\
&=\iint |v(p)-v(q)|^2\frac{d^\ell}{d t^\ell}\Big[K_s(\psi^{ t}(p),\psi^{ t}(q)) \,J_{ t}(p)J_{ t}(q)\Big]  dV_p\, dV_q\label{eqcvthm}.
\end{align}
Since $0 < t < T$, the derivatives in time of the Jacobians $J_t$ can of course be bounded by a constant $C$ with the right dependencies. What remains in order to bound \eqref{eqcvthm} by $C(1+\mathcal{E}^{\rm Sob}_M(v))$ is to control the first $k$-th derivatives in time of $K_s(\psi^{ t}(p),\psi^{ t}(q))$ by $C (1+K_s(p,q))$, for all $0<t<T$. The main bound is given by Proposition \ref{Ker2est}, which gives for all $1 \le k \le \ell $:
\begin{equation}\label{rderest}
    \left|\frac{d^k}{d t^k} K_s(\psi^{ t}(p),\psi^{ t}(q)) \right| = \left|\frac{d^k}{d r^k} \bigg|_{r=0} K_s(\psi^{ t+r}(p),\psi^{ t+r}(q)) \right| \leq  C \big(1+K_s(\psi^{ t}(p),\psi^{ t}(q))\big)\,.
\end{equation}
Now, integrating this inequality for $k=1$ similarly to how we proceeded in the proof of Lemma \ref{sdfgsdrgs}, we conclude that
\begin{equation}\label{kiniest}
    K_s(\psi^{ t}(p),\psi^{ t}(q))\leq C (1+K_s(p,q)) ,\s \text{ for all } 0< t<T\, .
\end{equation}
We can now go back to \eqref{eqcvthm} and apply the bounds that we just derived. We get that
\begin{align*}
\left| \frac{d^\ell}{d  t^\ell}\mathcal{E}^{\rm Sob}_{M}(v_t) \right|
& \leq \iint |v(p)-v(q)|^2\frac{d^\ell}{d t^\ell}\Big[K_s(\psi^{ t}(p),\psi^{ t}(q)) \,J_{ t}(p)J_{ t}(q)\Big]  dV_p\, dV_q \\ 
&\leq C \iint |v(p)-v(q)|^2 (1+K_s(p,q)) \, dV_p dV_q\\
&= C(1+\mathcal{E}^{\rm Sob}_{M}(v))
\end{align*}
for all $0< t<T$, where $C$ has the right dependencies.

\vsp
The potential part of the energy is simpler to deal with. Indeed, we have 
\begin{equation}\label{potderpart}
\frac{d^\ell}{d t^\ell}\mathcal{E}^{\rm Pot}_{M}({v_t}) =\frac{d^\ell}{d t^\ell}\int \ep^{-s} W(v(\psi^{- t}(p)) dV_p\, =  \int \ep^{-s} W(v(p))\frac{d^\ell}{d t^\ell}\,J_{ t}(p)  dV_p,
\end{equation}
from which we directly conclude that
\[
\left| \frac{d^\ell}{d t^\ell}\mathcal{E}^{\rm Pot}_{M}({v_t}) \right| \le C\mathcal{E}^{\rm Pot}_{M}({v}),
\]
finishing the proof.
\end{proof}

Lemma \ref{enboundslemma} has a local version, which comes from applying local estimates for the kernel instead.
\begin{lemma}\label{enboundslemma2}
Let $M$ satisfy the flatness assumptions ${\rm FA}_\ell(N,g,p,R,\varphi)$. Let $s \in (0,2)$ and $ v\in H^{s/2}(M)$ be a function with $|v|\leq 1$. Let $X \in \mathfrak{X} (M)$ be a smooth vector field supported on $\varphi(\B_{R/2})$, and put $v_t := v \circ \psi_X^{-t}$, where $\psi_X^t$ is the flow of $X$ at time $t$. Then, for all $T>0$ there holds
\begin{equation*}
    \sup_{0< t<T} \left| \frac{d^\ell}{d  t^\ell}\mathcal{E}_{\varphi(\B_{R/2})}(v_{t}) \right|\leq C(1+\mathcal{E}_{\varphi(\B_{R/2})}(v))\, ,
\end{equation*}
for some constant $C=C(s,\ell,T, \max_{0\le k\le \ell} \|\nabla^k X\|_{L^\infty(\varphi(\B_{R/2}))} )$ which stays bounded for $s$ away from $0$ and $2$.
\end{lemma}

\begin{proof}
%We change variables (using the flow $\psi_j^t(x)$) in the corresponding integral for the energy, and after that we exchange integration and differentiation.
We modify the proof of Lemma \ref{enboundslemma} accordingly. First, by scaling, it suffices to prove the Lemma in the case $R=1$. Since $X$ is supported on $\varphi(\B_{1/2})$, the integrand in \eqref{eqcvthm} is supported then on
\begin{align*}
&(N\times N)\setminus (\varphi(\B_{1/2})^c\times\varphi(\B_{1/2})^c)=\Big[(\varphi(\B_{2/3})\times\varphi(\B_{2/3}))\setminus (\varphi(\B_{1/2})^c\times\varphi(\B_{1/2})^c)\Big]\cup\\
&\hspace{6cm} \cup\Big[(\varphi(\B_{1/2})\times (N\setminus\varphi(\B_{2/3})))\cup ((N\setminus\varphi(\B_{2/3}))\times\varphi(\B_{1/2}))\Big]\, ,
\end{align*}
so that
% It suffices to prove the Lemma in the case $R=1$. As usual, $K(p,q)$ will denote the kernel of the fractional Laplacian on $N$. We can write
% \begin{align*}
% \mathcal{E}_{N}^{\rm Sob}(u_{ t})
% &=\iint_{N\times N} |u(\psi_Y^{- t}(p))-u(\psi_Y^{- t}(q))|^2K(p,q) \, dV_p\,dV_q
% \end{align*}
% Changing variables with the flow $\psi_Y^t$ of $Y$, and denoting by $J_t$ the Jacobian of the flow, we get
% \begin{align*}
% \mathcal{E}_{N}^{\rm Sob}(u_{t})&=\iint_{N\times N} |u(p)-u(q)|^2K(\psi_Y^{ t}(p),\psi_Y^{ t}(q)) J_{ t}(p)J_{ t}(q)\, dV_p\,dV_q\, , \nonumber
% \end{align*}
% thus
% \begin{align*}
% \frac{d^k}{dt^k}\mathcal{E}_{N}^{\rm Sob}(u_{t})&=\iint_{N\times N} |u(p)-u(q)|^2\frac{d^k}{dt^k}\Big[K(\psi_Y^{ t}(p),\psi_Y^{ t}(q)) J_{ t}(p)J_{ t}(q)\Big]\, dV_p\,dV_q\, . \nonumber
% \end{align*}

% Given that $Y$ has support in $\varphi(\B_{R})$, the integrand above is supported on
% \begin{align*}
% (N\times N)\setminus (\varphi(\B_{1/2})^c\times\varphi(\B_{1/2})^c)=&[(\varphi(\B_{2/3})\times\varphi(\B_{2/3}))\setminus (\varphi(\B_{1/2})^c\times\varphi(\B_{1/2})^c)]\\
% &\cup[(\varphi(\B_{1/2})\times (N\setminus\varphi(\B_{2/3})))\cup ((N\setminus\varphi(\B_{2/3}))\times\varphi(\B_{1/2}))]\, .
% \end{align*}
\begin{align*}
\bigg| & \frac{d^k}{d t^k}  \mathcal{E}_{M}^{\rm Sob}(v_{t})\bigg|=\bigg|\frac{d^k}{d t^k}\mathcal{E}_{\varphi(\B_{1/2})}^{\rm Sob}(v_{t})\bigg|\\
&=\bigg|\iint_{(\varphi(\B_{2/3})\times\varphi(\B_{2/3}))\setminus (\varphi(\B_{1/2})^c\times\varphi(\B_{1/2})^c)} |v(p)-v(q)|^2\frac{d^k}{d t^k}\Big[K(\psi_Y^{ t}(p),\psi_Y^t(q))J_{ t}(p)J_{ t}(q)\Big]\,dV_p\,dV_q \\
&\hspace{0.6cm}+2\iint_{\varphi(\B_{1/2})\times (N\setminus\varphi(\B_{2/3}))} |v(p)-v(q)|^2\frac{d^k}{d t^k}\Big[K(\psi_Y^{ t}(p),q)J_{ t}(p)\Big]\,dV_p\,dV_q\bigg| \\
&\leq C\iint_{(\B_{2/3}\times\B_{2/3})\setminus (\B_{1/2}^c\times\B_{1/2}^c)} |v(\varphi(x))-v(\varphi(y))|^2\Big|\frac{d^k}{d t^k}\Big[K(\varphi(\psi_X^{ t}(x)),\varphi(\psi_X^{ t}(y)))J_{ t}(\varphi(x))J_{ t}(\varphi(y))\Big]\Big|\,dx\,dy\\
&\hspace{0.6cm}+C\iint_{\B_{1/2}\times (N\setminus\varphi(\B_{2/3}))} |v(\varphi(x))-v(q)|^2\Big|\frac{d^k}{d t^k}\Big[K(\varphi(\psi_X^{ t}(x)),q)J_{ t}(\varphi(x))\Big]\Big|\,dx\,dV_q\, .
\end{align*}
Bounding the derivatives in time of the Jacobians by a constant with the right dependencies, using \eqref{cdsacsdc1} to bound the kernel in the first double integral, and using \eqref{remaining3} to bound the integral in $q$ in the second double integral by a constant, we conclude that
\begin{equation*}
    \left| \frac{d^k}{d  t^k}\mathcal{E}_{\varphi(\B_{1/2})}^{\rm Sob}(v_{t})\right|\leq C(1+\mathcal{E}_{\varphi(\B_{1/2})}^{\rm Sob}(v))\, .
\end{equation*}
Regarding the potential part of the energy, from the computation in \eqref{potderpart} we readily find that
\begin{equation}\label{potabbound}
\left| \frac{d^\ell}{d t^\ell}\mathcal{E}^{\rm Pot}_{\varphi(\B_{1/2})}({v_t}) \right| \le C\mathcal{E}^{\rm Pot}_{\varphi(\B_{1/2})}({v})
\end{equation}
where $C$ has the right dependencies, which completes the proof.
\end{proof}

\subsection{Estimates for the extension problem}\label{AppX}

\begin{lemma}\label{linf-grad} Let $ s\in (0,2)$ and $M$ satisfy the flatness assumption ${\rm FA}_2(M,g,2,p,\varphi)$. Let also $U : \ov{B}^+_2(p,0) \to \R $ be any function solving
\begin{equation*}
    \widetilde{\textnormal{div}}(z^{1-s} \widetilde \nabla U) =0 \s in \s \ov{B}^+_2(p,0) \,,
\end{equation*}
and let $u$ be its trace on $B_2(p)$. Assume also $U \in L^{\infty}(\ov{B}^+_2(p,0))$ and $u=0$ in $B_{3/2}(p)$. Then there exists $C=C(n)>0$ such that
\begin{equation*}
    z^{1-s}| \widetilde \nabla U (q,z) | \le \frac{C}{s} \| U \|_{L^\infty(\widetilde B^+_{6/5}(p,0) )} \,.
\end{equation*}
for every $(q,z) \in \widetilde B_1^+(p,0)$.
\end{lemma}
\begin{proof}
Let $C$ denote a constant that depends only on $n$. This estimate is proved by a barrier argument. Let $\alpha, \beta >0 $ to be chosen later, and for $q_\circ \in B_{11/10}(p) $ define
\begin{equation*} 
    b_{q_\circ}(q,z) := \frac{\alpha}{2} \big|\varphi^{-1}(q) - \varphi^{-1}(q_\circ) \big|^2- \beta ( z^2 - 2 z^s) \,.
\end{equation*}
Denote by $\Delta_g$ the Laplace-Beltrami operator of $(M,g)$. Then, by ${\rm FA}_2(M,g,2,p,\varphi)$, for $(x,z) \in \widetilde B_{6/5}^+$ there holds $| \Delta_g b_{q_\circ} (q,z)| \le C$. Moreover
\begin{equation*}
    \big( \partial_{zz} +\tfrac{1-s}{z} \partial_z \big)z^2 =2s \s\s \textnormal{and} \s\s \big( \partial_{zz} +\tfrac{1-s}{z} \partial_z \big)z^s =0 \,.
\end{equation*}
Hence 
\begin{equation*}
    \widetilde \dive (z^{1-s} \widetilde \nabla b_{q_\circ} ) = z^{1-s} \left( \Delta_g b_{q_\circ} + \partial_{zz} b_{q_\circ} + \frac{1-s}{z} \partial_z b_{q_\circ} \right) \le z^{1-s}(C \alpha - 2 s \beta  ) \le 0 \,,
\end{equation*}
provided we take $\beta = C \alpha/s$.

\vspace{2pt}
Since $U =0$ in $B_{3/2}(p) \times \{0\}$ clearly $  |U| \le b_{q_\circ}(\cdot,0) $ in $B_{6/5}(p) \times \{0\}$ . Moreover, for every $(x,z) \in \partial ^+ \widetilde B^+_{6/5}(p,0) $ there holds
\begin{equation*}
    b_{q_\circ}(q,z) \ge C\alpha - \beta( z^2-2 z^s) \ge C\alpha \ge \| U \|_{L^\infty(\widetilde B^+_{6/5}(p,0) )} \ge U(x,z) \,,
\end{equation*}
provided we choose $\alpha= C\| U \|_{L^\infty(\widetilde B^+_{6/5}(p,0) ) } $. 

\vspace{2pt}
Hence, with this choice of $\alpha$ and $\beta$, $ |U| \le b_{q_\circ}$ on the full boundary $\partial \widetilde B_{6/5}^+(p,0) $. Since also $U$ solves $\widetilde{\textnormal{div}}(z^{1-s} \widetilde \nabla U) =0$ in $\widetilde B_{6/5}^+(p,0)$, by the maximum principle we get
\begin{equation}\label{aqwefdevadrgadf}
   |U (q_\circ,z)| \le b_{q_\circ}(q_\circ,z) \le \frac{Cz^s}{s} \| U \|_{L^\infty(\widetilde B^+_{6/5}(p,0))}\,,
\end{equation}
for $(q_\circ,z) \in \widetilde B^+_{11/10}(p,0)$. Moreover, by standard (interior) gradient estimates for uniformly elliptic equations, for all $(q,z) \in \widetilde B^+_1(p,0)$ we have
\begin{equation*}
    | \widetilde \nabla U (q,z) | \le \|\widetilde \nabla U \|_{L^{\infty}(\widetilde B^+_{z/100}(q,z))} \le \frac{C}{z} \| U\|_{L^{\infty}(\widetilde B^+_{z/50}(q,z))} \,,
\end{equation*}
which, since $\widetilde B^+_{z/50}(q,z) \subset \widetilde B^+_{11/10}(p,0)$, together with \eqref{aqwefdevadrgadf} implies
\begin{equation*}
    | \widetilde \nabla U (x,z) | \le \frac{C}{s} z^{s-1} \| U \|_{L^\infty(\widetilde B^+_{6/5})} \,,
\end{equation*}
and this concludes the proof.
\end{proof}

\begin{lemma}\label{lem:whtorwohh0}
Let $s_0 \in (0,2)$, $s \in (s_0,2)$. Consider the Riemannian manifold $(\R^n,g)$ with $\big(1-\tfrac 1 {100}\big) |v|^2 \le g_{ij}(x)v^i v^j \le \big(1+\tfrac 1 {100}\big)|v|^2$ and  $\|g_{ij}\|_{C^{1,1}(\overline \R^n)} \le 1$.
Let also $u:\R^n\to [-1,1]$ and  $U : \R^{n}\times \R_+ \to [-1,1]$ be the extension of $u$ (in the sense of Theorem \ref{extmfd}). Then
\begin{equation}
 \int_{\B_{1}^+(0,0)}  |\widetilde \nabla  U|^2 z^{1-s}dVdz \le C  \iint_{\R^n\times \R^n \setminus (\B_2^c \times \B_2^c)} \frac{|u(x) -
u(y)|^2}{|x-y|^{n+s}}  \,dxdy,
\end{equation}
with $C$ depending only on $n$ and $s_0$.
\end{lemma}
\begin{proof}
We proceed as in \cite[Proposition 7.1]{CRS}.
Assume without loss of generality that $\int _{\B_2}u(x) dx=0.$
Let $\xi: \R^n \rightarrow \R$ be a cutoff function such that
$\xi = 1$ in  $\B_{3/2}(0)$ and it is compactly supported in
$\B_2(0)$.  We write
$u= u \xi + u(1-\xi)= u_1 + u_2$ and  $U= U_1+ U_2$.

\vsp
On the one hand, since $u_1$ is compactly supported we have
\[
\begin{split}
\beta_s \int_{\R^{n+1}_+}z^{1-s} | \widetilde \nabla  U_1|^2 &  dVdz =   \|u_1\|_{H^{s/2}(\R^n,g)} \\ & = \iint_{\R^n\times \R^n\setminus (\B_2^c\times \B_2^c)} |u_1(x) -
u_1(y)|^2 K(x,y)  \;dV_x dV_y\\
& \le  C \alpha_{n,s} \iint_{\B_2\times \R^n} \frac{|u(x) -
u(y)|^2}{|x-y|^{n+s}} \xi^2(x) \,dxdy + C \alpha_{n,s} \iint_{\B_2\times \R^n} |u(x)|^2 \frac{|\xi(x) - \xi(y)|^2}{|x-y|^{n+s}} \,dxdy\\
& \le C \alpha_{n,s}\iint_{\R^n\times \R^n \setminus (\B_2^c \times \B_2^c)} \frac{|u(x) -
u(y)|^2}{|x-y|^{n+s}}  \,dxdy + C \int_{\B_2} |u(x)|^2 dx.
\end{split}
\] 
Moreover, using the fractional  Poincar\'e inequality (recall $\int _{\B_2}u(x) dx=0$):
\[
\int_{\B_2} |u(x)|^2 dx \le C\alpha_{n,s}\iint_{\B_2\times \B_2} \frac{|u(x) -
u(y)|^2}{|x-y|^{n+s}}  \,dxdy \le C\alpha_{n,s}\iint_{\R^n\times \R^n \setminus (\B_2^c \times \B_2^c)} \frac{|u(x) -
u(y)|^2}{|x-y|^{n+s}}  \,dxdy =: {\rm I}
\]

On the other hand (using again $\int _{\B_2}u(y) dy=0$)
\[\int_{\R^n}\frac{u(x)^2}{(1+|x|^2)^{\frac{n+s}{2}}}dx =
\iint\frac{(u(x)-u(y))^2}{(1+|x|^2)^{\frac{n+s}{2}}}\frac{\chi_{\B_2}(y)}{|\B_2|}
dxdy - \iint\frac{u(y)^2}{(1+|x|^2)^{\frac{n+s}{2}}}\frac{\chi_{\B_2}(y)}{|\B_2|}dxdy
\le C{\rm I},\]
and by Holder's inequality
\begin{equation*}
    \int_{\R^n}\frac{|u(x)|}{(1+|x|^2)^{\frac{n+s}{2}}}dx \le \left(\int_{\R^n} \frac{|u(x)|^2}{(1+|x|^2)^{\frac{n+s}{2}}} dx \right)^{1/2} \left(  \int_{\R^n} \frac{1}{(1+|x|^2)^{\frac{n+s}{2}}} dx \right)^{1/2} \le C
{\rm I}^{1/2}.
\end{equation*}

\noindent 
\textbf{Claim.} There is $C=C(n)>0$ such that for $(x,z) \in \B_1^+$
\begin{equation}\label{bsfyhdfgh}
    z^{1-s} |\widetilde \nabla U_2(x,z)| \le C
\int_{\R^n}\frac{|u_2(y)|}{{(1+|y|^2)^{\frac{n+s}{2}}}}dy \,.
\end{equation}
\noindent
We postpone the proof of this claim and first see how to conclude the proof of Lemma \ref{lem:whtorwohh0} with it. 

\vsp
By the claim, if  $(x,z) \in \B_1^+$ then 
\[z^{1-s} |\widetilde \nabla U_2(x,z)| \le C
\int_{\R^n}\frac{|u_2(y)|}{{(1+|y|^2)^{\frac{n+s}{2}}}}dy \le C
\int_{\R^n}\frac{|u(y)|}{{(1+|y|^2)^{\frac{n+s}{2}}}}dy \le C
{\rm I}^{1/2}\]

But then the inequality
\[
\int_{\B_1^+}z^{1-s} |\widetilde \nabla U_2|^2 dxdz \le   C{\rm I}^{1/2} \int_{\B_1^+} |\widetilde \nabla U_2| dxdz \le C{\rm I}^{1/2} \bigg(\int_{\B_1^+} z^{1-s}|\widetilde \nabla U_2|^2 dxdz\bigg)^{1/2} \bigg(\int_{\B_1^+} z^{s-1} dxdz\bigg)^{1/2},
\]
gives 
\[
\int_{\B_1^+}z^{1-s} |\widetilde \nabla  U_2|^2 dxdz \le   C{\rm I},
\]
and the lemma follows.

\vsp
%\begin{proof}[%Second 
It only remains to prove \eqref{bsfyhdfgh}.
    Let $H_N(x,y,t)$ be the heat kernel of $N:=(\R^n,g)$. By \eqref{bdfgshdfg} the fractional Poisson kernel\footnote{Which equals $\sigma_{n,s} \frac{z^s}{(|x-y|^2+z^2)^{\frac{n+s}{2}}}$ on $\R^n$ with its standard metric, for some normalization constant $\sigma_{n,s}>0$.} $ \mathbb{P}_N : \R^n \times \R^n \times (0,\infty) \to [0, \infty)$ of $N$ can be represented as
\begin{equation*}
    \mathbb{P}_N(x,y,z)= \frac{z^{s}}{2^s\Gamma(s/2)} \int_{0}^{\infty} H_N(x,y,t)  e^{-\frac{z^2}{4t}}\frac{dt}{t^{1+s/2}} \, , 
\end{equation*}
and the solution $U_2$ to the extension problem with trace $u_2$ is $U_2(x,z) = \int_{\R^n} \mathbb{P}_N(x,y,z) u_2(y) \, dV_y $. Then, by Lemma \ref{globcomparability} we have that $\mathbb{P}_N$ is comparable (up to dimensional constants) to the fractional Poisson kernel of $\R^n$ with its standard metric, that is 
\begin{equation*}
    cs \frac{z^s}{(|x-y|^2+z^2)^{\frac{n+s}{2}}}  \le \mathbb{P}_N(x,y,z) \le Cs \frac{z^s}{(|x-y|^2+z^2)^{\frac{n+s}{2}}} \,,
\end{equation*}
for some $C,c>0$ dimensional. Hence, for every $(x,z) \in \widetilde \B_{6/5}^+$
\begin{equation*}
   |U_2 (x,z)| \le Cs  \int_{\R^n} \frac{|u_2(y)|}{(|x-y|^2+z^2)^{\frac{n+s}{2}}} \, dy \le Cs \int_{\R^n \setminus \B_2} \frac{|u_2(y)|}{|x-y|^{n+s}} \, dy \,.
\end{equation*} 
Since $x\in \B_{6/5}$ and $y\in \R^n \setminus \B_2$ there holds $|x-y|\ge \frac{1}{100}\sqrt{1+|y|^2}$, and hence 
\begin{equation*}
   \|U_2\|_{L^\infty(\B^+_{6/5})} \le Cs \int_{\R^n} \frac{|u_2(y)|}{(1+|y|^2)^{\frac{n+s}{2}}} \, dy \,.
\end{equation*} 
From here, the result follows directly by Lemma \ref{linf-grad}. 
\end{proof}

We will now give an interpolation result for the extended energy. We will use the following standard interpolation result on a Euclidean ball:
\begin{proposition} \label{interprop} Let $s\in(0,1)$, and let $u:\B_1\subset\R^n\to[-1,1]$ be a function of bounded variation. Then,
\begin{equation*}
    \iint_{\B_1\times \B_1} \frac{|u(x)-u(y)|^2}{|x-y|^{n+s}}\,dx\,dy \leq \frac{C(n)}{(1-s)s}[u]_{BV(\B_1)}^s\|u\|_{L^1(\B_1)}^{1-s}\, .
\end{equation*}
    
\end{proposition}
\begin{proof}
See, for instance, Proposition 4.2 in \cite{Brasc} for a simple proof; see also \cite{BM}.
\end{proof}

\begin{lemma}\label{lem:whtorwohh}
Let $s_0\in (0,1)$ and $s \in (s_0,1)$. Let $M$ satisfy flatness assumptions ${\rm FA}_1(M,g,1,p,\varphi)$. Let also $U : \ov{B}^+_1(p,0) \to (-1,1)$ be any function solving
\begin{equation}\label{extenpro}
    \widetilde{\textnormal{div}}(z^{1-s} \widetilde \nabla U) =0\, ,
\end{equation} and let $u$ be its trace on $B_1(p)$. Then for all $\varrho>0$, $R\ge1$, $k \in \R$ and  $q\in B_{1/2}(p)$ such that  $B_{R\varrho}(q)\subset B_{3/4}(p)$,  
\[
 \varrho^{s-n} \beta_s \int_{\widetilde B^+_{\varrho}(q,0)} z^{1-s} |\widetilde \nabla U|^2\,dVdz \leq \frac{C}{R^s} + \frac{C}{1-s}\bigg(\varrho^{-n}\int_{B_{R \varrho}(q)}|u + k|\,dV\bigg)^{1-s}  \bigg(\varrho^{1-n}\int_{B_{R \varrho}(q)} |\nabla u|\,dV\bigg)^s,
\]
where the constant $C$ depends only on $n$ and $s_0$.
\end{lemma}
\begin{proof}
Let us show  that if $U$ and $U'$ are two different
solutions $\ov{B}^+_1(p,0) \to (-1,1)$ of \eqref{extenpro} with the same trace $u$ on $B_{3/4}(p)$, then
\begin{equation}\label{whithoiwh12}
\int_{\widetilde B^+_{\varrho}(q,0)} z^{1-s} \big(|\widetilde \nabla U|^2-|\widetilde \nabla U'|^2\big)\,dVdz \le C\int_{\widetilde B^+_{\varrho}(q,0)} z^{1-s} |\widetilde \nabla U'|^2\,dVdz.
\end{equation}

Indeed, by Lemma \ref{linf-grad} (rescaled) we have
$z^{1-s} \big|\widetilde \nabla (U-U')\big| \le C$ in $\widetilde B^+_{1/2}(p,0)$. 
Thus, we obtain 
\[
\begin{split}
\int_{\widetilde B^+_{\varrho}(q,0)} z^{1-s} \big(|\widetilde \nabla U|^2-|\widetilde \nabla U'|^2\big)\,dVdz &= 
\int_{\widetilde B^+_{\varrho}(q,0)} z^{1-s} \big(\widetilde \nabla (U-U')\big)\cdot \big( \widetilde \nabla(U+U'))\,dVdz \\
& \le C \bigg(\int_{\widetilde B^+_{\varrho}(q,0)} z^{s-1} dVdz\bigg)^{\frac 1 2} \bigg(\int_{\widetilde B^+_{\varrho}(q,0)} z^{1-s} \big(|\widetilde \nabla U|^2+|\widetilde \nabla U'|^2 \big) dVdz\bigg)^{\frac 12}
\\& 
\le \frac 12 \int_{\widetilde B^+_{\varrho}(q,0)} z^{1-s} \big(|\widetilde \nabla U|^2-|\widetilde \nabla U'|^2\big)\,dVdz  + C\int_{\widetilde B^+_{\varrho}(q,0)} z^{1-s} |\widetilde \nabla U'|^2\,dVdz
\end{split}
\]
Thus \eqref{whithoiwh12} follows.

\vsp 
Now let $g_{ij}$ be the components of the metric in the coordinates $\varphi^{-1}$, $\eta \in C^\infty_c(\B_1)$ be a nonnegative smooth cut-off function satisfying $\eta\equiv 1$ in $\B_{3/4}$, and put $g'_{ij} = g_{ij}\eta + \delta_{ij}(1-\eta)$, a metric defined in the whole $\R^n$. Thanks to \eqref{whithoiwh12} it is enough to prove the lemma for the manifold $(\R^n, g')$ with $p=0$ and with $U$ replaced by the (unique!) bounded solution $U'$ of \eqref{extenpro} (with respect to the metric $g'$) in all of $\R^n\times \R_+$.
But in this case we can use Lemma \ref{lem:whtorwohh0} (rescaled)
and obtain 
\[
\begin{split}
 \varrho^{s-n} \int_{\widetilde B^+_{\varrho}} z^{1-s} |\widetilde \nabla U'|^2\,dVdz &\leq C \varrho^{s-n}\iint_{(\R^n\times \R^n) \setminus (\B_{2\varrho}^c\times \B_{2\varrho}^c)} \frac{|u(x)-u(y)|^2}{|x-y|^{n+s}} dxdy\\
 &\le 
 C \varrho^{s-n} 2\iint_{ \B_{R\varrho}\times\B_{R\varrho} \, \cup \, \B_{\varrho}\times \B_{R\varrho}^c} \frac{|u(x)-u(y)|^2}{|x-y|^{n+s}} dxdy
 \\
 &\le 
 C \varrho^{s-n} \bigg( \iint_{ \B_{R\varrho}\times\B_{R\varrho} } \frac{|u(x)-u(y)|^2}{|x-y|^{n+s}} dxdy + \frac{C \varrho^{n-s}}{R^s}\bigg) \,,
 \end{split}
\]
where we have used that 
\begin{align*}
    \iint_{\B_\varrho \times \B_{R\varrho}^c} \frac{|u(x)-u(y)|^2}{|x-y|^{n+s}} \, dxdy \le C \int_{\B_\varrho}\left( \int_{\B_{R\varrho}} \frac{1}{(|y|-\varrho)^{n+s}} dy\right)dx \le \frac{C\varrho^{n}}{(R\varrho)^{s}} \,.
\end{align*}
We conclude using the interpolation inequality of Proposition \ref{interprop}, since the modulus of the Euclidean gradient in $\R^{n+1}$ and the metric gradient $|\widetilde \nabla U|$ are comparable.
\end{proof}

%%%%%%%%%%%%%%%%%%% Acknowledgements %%%%%%%%%%%%%%%%%%
\subsection*{Acknowledgements} 
M.C. would like to thank the FIM (Institute for Mathematical Research) at ETH Z\"{u}rich for the wonderful hospitality during his many stays in 2022-2023. J.S. is supported by the European Research Council under Grant Agreement No 948029. E.F-S has been partially supported
by the CFIS Mobility Grant, by the MOBINT-MIF Scholarship from AGAUR, and the support of a fellowship from ”la Caixa” Foundation (ID 100010434)”. The
fellowship code is “LCF/BQ/EU22/11930072”.

\vspace{10pt}\noindent
\textbf{Data availability}: This paper does not use any data set.

\vspace{10pt}\noindent
\textbf{Conflict of interest}: On behalf of all authors, the corresponding author states that there is no conflict of interest.

%%%%%%%%%%%%%%%%%%%%%%%%%%%%%%%%%%%%%%%%%%%%%%%%%%%%%%%


\begin{thebibliography}{99}

% \bibitem{AAC} G.  Alberti, L. Ambrosio, and X.  Cabr\'e, {\em On a long-standing conjecture of E. De Giorgi: symmetry in 3D for general nonlinearities and a local minimality property}, Acta Appl. Math. 65 (2001), 9--33.

% \bibitem{Alm} F. J. Almgren,  {\em Some interior regularity theorems for minimal surfaces and an extension of Bernstein's theorem}, Ann. of Math. 84 (1966),  277--292.

% \bibitem{Alm2} F. J. Almgren, {\em The homotopy groups of the integral cycle groups}, Topology 1 (1962), 257--299.

% \bibitem{AlmNotes} F. J. Almgren, {\em The theory of varifolds}, mimeographed notes, Princeton (1965).

% \bibitem{AmbrC} L. Ambrosio and X. Cabr\'e, {\em Entire solutions of semilinear elliptic equations in $\R^3$ and a conjecture of De Giorgi}, J. Amer. Math. Soc. 13 (2000),  725--739.

\bibitem{ADPM11} L. Ambrosio, G. De Philippis, and L. Martinazzi, {\em Gamma-convergence of nonlocal perimeter functionals}, Manuscripta Math. 134 (2011), 377–-403.


% \bibitem{AFP} L. Ambrosio, N. Fusco, and D. Pallara, {\em Functions of Bounded Variation and Free Discontinuity Problems}, Oxford Science Publications (2000).

\bibitem{Aronson} D. G. Aronson, {\em
Bounds for the fundamental solution of a parabolic equation}, 
Bull. Amer. Math. Soc. 73 (1967), 890-896. 

\bibitem{BGS} V. Banica, M. M. Gonz\'alez, and M. S\'aez, {\em Some constructions for the fractional Laplacian on noncompact manifolds}, Rev. Mat. Iberoam. 31 (2015), no. 2, 681–712.


% \bibitem{BFV} B. Barrios, A. Figalli, and E. Valdinoci, {\em Bootstrap regularity for integro-differential operators and its application to nonlocal minimal surfaces},  Ann. Scuola Norm. Sup. Pisa Cl. Sci. 13 (2014),  609--639.

% \bibitem{Bern} S. N. Bernstein, {\em Sur une th\'eor\`eme de g\'eometrie et ses applications aux \'equations d\'eriv\'ees partielles du type elliptique}, Comm. Soc. Math. Kharkov 15  (1915-1917), 38--45.

\bibitem{BM} H. Br\'ezis and P. Minorescu,  {\em Gagliardo-Nirenberg inequalities and non-inequalities: The full story},  Ann. Inst. H. Poincar\'e Anal. Non Lin\'eaire 35 (2018),  1355--1376.

% \bibitem{BDG} E. Bombieri, E. De Giorgi, and E. Giusti, {\em Minimal cones and the Bernstein problem}, Invent. Math. 7 (1969), 243--268.

\bibitem{BBM01} J. Bourgain, H. Brezis, and P. Mironescu, {\em Another look at Sobolev spaces}, Optimal control
and partial differential equations, 439-455, IOS, Amsterdam, (2001).

% \bibitem{Brasc} L. Brasco, E. Lindgren and E. Parini, {\em The fractional Cheeger problem}, Interfaces Free Bound. 16 (2014) 419–458.

% \bibitem{PalaisSmale} J. Byeon and P. Rabinowitz, {\em A note on mountain pass solutions for a class of Allen-Cahn models}, RIMS Kokyuroku 1881 (2014), 1--17.

% \bibitem{CC1} X. Cabr\'e and E. Cinti, {\em Energy estimates and 1-D symmetry for nonlinear equations involving the half-Laplacian}, Discrete Contin. Dyn. Syst. 28 (2010),  1179--1206.

% \bibitem{CC2} X. Cabr\'e and E. Cinti, {\em Sharp energy estimates for nonlinear fractional diffusion equations}, Calc. Var. Partial Differential Equations 49 (2014),  233--269.

\bibitem{Brasc} L. Brasco, E. Lindgren and E. Parini, {\em The fractional Cheeger problem}, Interfaces Free Bound. 16 (2014) 419–458.

\bibitem{CC} X. Cabr\'e and E. Cinti, {\em Sharp energy estimates for nonlinear fractional diffusion equations}, Calc. Var. Partial Differential Equations 49 (2014),  233--269.

% \bibitem{Stable} X. Cabr\'e, E. Cinti, and J. Serra, {\em Stable solutions to the fractional Allen-Cahn equation in the nonlocal perimeter regime}, arXiv:2111.06285.

% \bibitem{CCS} X. Cabr\'e, E. Cinti, and J. Serra, {\em Stable $s$-minimal cones in $\R^3$ as flat for $s\sim 1$}, J. Reine Angew. Math. 764 (2020), 157--180.

% \bibitem{Cabre-Sanz}
% X. Cabr\'e and T. Sanz-Perela, {\em A universal H\"older estimate up to dimension 4 for stable solutions to half-Laplacian semilinear equations}, 
% J. Differential Equations 317 (2022), 153-195.


% \bibitem{C-Si1} X. Cabr\'e and Y. Sire, {\em Nonlinear equations for fractional Laplacians, I: Regularity, maximum principles, and Hamiltonian estimates}, Ann. Inst. H. Poincar\'e Anal. Non Lin\'eaire 31 (2014),  23--53.

% \bibitem{CSi} X. Cabr\'e and Y. Sire, {\em Nonlinear equations for fractional Laplacians II: Existence, uniqueness, and qualitative properties of solutions}, Trans. Amer. Math. Soc. 367 (2015),  911--941.

% \bibitem{C-SM} X. Cabr\'e and J. Sol\`a-Morales, {\em Layer solutions in a half-space for boundary reactions}, Comm. Pure Appl. Math. 58 (2005), 1678--1732.


\bibitem{CRS} L. Caffarelli, J.-M. Roquejoffre, and O. Savin, \emph{Nonlocal minimal surfaces}, Comm. Pure Appl. Math. 63 (2010), 1111--1144.

% \bibitem{CafSireg} L. Caffarelli and L. Silvestre, {\em Regularity theory for fully nonlinear integro-differential equations.}, Comm. Pure Appl. Math. 62 (2009), 597-638.

% \bibitem{CSapprox} L. Caffarelli and L. Silvestre, {\em Regularity results for nonlocal equations by approximation}, Arch. Ration. Mech. Anal. 200 (2011), no. 1, 59–88. 


\bibitem{CafSi} L. Caffarelli and L. Silvestre, {\em An extension problem related to the fractional Laplacian}, Comm. Partial Differential Equations 32 (2007), 1245--1260.

% \bibitem{CafSti} L. Caffarelli and P. R. Stinga, {\em Fractional elliptic equations, Caccioppoli estimates and regularity}, Ann. Inst. H. Poincar\'e Anal. Non Lin\'eaire 33 (2016), 767--807.

% \bibitem{CV} L. Caffarelli and E. Valdinoci, {\em Regularity properties of nonlocal minimal surfaces via limiting arguments}, Adv. Math. 248 (2013), 843--871.

\bibitem{CafStinga} L. Caffarelli and P. R. Stinga, {\em Fractional elliptic equations, Caccioppoli estimates and regularity}, Ann. Inst. Henri Poincare (C) Anal. Non Lineaire 33 (2016), 767--807.

\bibitem{CV11} L. Caffarelli and E. Valdinoci, {\em Uniform estimates and limiting arguments for nonlocal minimal
surfaces}, Calc. Var. Partial Differential Equations 41 (2011), 203–240.

% \bibitem{MCthesis} M. Caselli, {\em Min-max construction of fractional minimal curves on Riemannian surfaces.}, ETH Zurich Master Thesis collection, 2021.

\bibitem{CG23} M. Caselli and L. Gennaioli, {\em Asymptotics as $s\to 0^+$ of the fractional perimeter on Riemannian manifolds.}, arxiv preprint, 2023.

\bibitem{CFS} M. Caselli, E. Florit-Simon and J. Serra, {\em Yau's conjecture for nonlocal minimal surfaces}, arxiv preprint, 2023.

% \bibitem{CahnHil} J. Cahn and J. Hilliard,  {\em Free energy of a nonuniform system I. Interfacial free energy}, J. Chem. Phys 28 (1958), 258--267.


% \bibitem{ChambLio}
% G. Chambers and Y. Liokumovich, {\em Existence of minimal hypersurfaces in complete manifolds of finite volume},
% Invent. Math. 219 (2020), 179-217.

% \bibitem{CDPLW} H. Chan,  J. D\'avila, M. del Pino, Y. Liu, and J. Wei, \emph{A gluing construction for fractional elliptic equations. Part II: Counterexamples of De Giorgi Conjecture for the fractional Allen-Cahn equation}, in preparation.

% \bibitem{CDSV} H. Chan,  S. Dipierro, J. Serra and E. Valdinoci, \emph{Nonlocal approximation of minimal surfaces: optimal estimates from stability}, preprint.

% \bibitem{CL}
% O. Chodosh and C. Li, {\em Stable minimal hypersurfaces in $\R^4$},
% arXiv:2108.11462v2.

% \bibitem{CM}
% O. Chodosh and C. Mantoulidis, {\em Minimal surfaces and the Allen--Cahn equation on 3-manifolds: index, multiplicity, and curvature estimates},
% Ann. of Math. 191 (2020), 213-328.

% \bibitem{CM2}
% O. Chodosh and C. Mantoulidis, {\em The p-widths of a surface},
% arXiv: 2107.11684.


% \bibitem{CSV} E. Cinti, J. Serra, and E. Valdinoci, {\em Quantitative flatness results and $BV$-estimates for stable nonlocal minimal surfaces}, J. Differential Geom.
% 112 (2019), 447--504.



\bibitem{ColdingMin} T. Colding and W. Minicozzi, {\em A Course in Minimal Surfaces}, Volume 121 di Graduate studies in mathematics, ISSN 1065-7339, American Mathematical Soc., 2011.

\bibitem{Dav02} J. D\'avila, {\em On an open question about functions of bounded variation}, Calc. Var. Partial Differential
Equations 15 (2002), 519-–527.

% \bibitem{DG1} E. De Giorgi, {\em
% Frontiere orientate di misura minima} (Italian),  
% Seminario di Matematica della Scuola Normale Superiore di Pisa, 1960-61 Editrice Tecnico Scientifica, Pisa 1961. 

% \bibitem{DG-Bern} E. De Giorgi,
% {\em Una estensione del teorema di Bernstein} (Italian), 
% Ann. Scuola Norm. Sup. Pisa 19 (1965), 79--85.
% \bibitem{dPKW} M. del Pino, M. Kowalczyk, and J. Wei, {\em  A conjecture by de Giorgi in large dimensions}, Ann. of Math. 174 (2011), 1485--1569.


% \bibitem{Hitch} E. Di Nezza, G. Palatucci, and E. Valdinoci, {\em Hitchhiker's guide to the fractional Sobolev spaces}, Bull. Sci. Math. 136 (2012), 521--573.

 \bibitem{Evans} L.C. Evans , {\em Partial differential equations}, Graduate Studies in Mathematics, 19, American Mathematical Society, 2010. 

% \bibitem{dCP} M. do Carmo and C. K. Peng,  {\em Stable complete minimal surfaces in $\R^3$ are planes}, Bull. Amer. Math. Soc. (N.S.) 1 (1979), 903--906.

% \bibitem{Fed} H. Federer, {\em Geometric Measure Theory}, Springer Berlin Heidelberg (2014) .

% \bibitem{FelSanz} J. Felipe-Navarro, T. Sanz-Perela, {\em Uniqueness and stability of the saddle-shaped solution to the fractional Allen–Cahn equation}, Rev. Mat. Iberoam. 36 (2020), 1887–-1916.

% \bibitem{FFMMM} A. Figalli, N. Fusco, F. Maggi, V. Millot and M. Morini, {\em Isoperimetry and stability properties of balls with respect to nonlocal energies}, Comm. Math. Phys. 336 (2015), 441--507.

% \bibitem{FS} A. Figalli and J. Serra, {\em On stable solutions for boundary reactions: a De Giorgi-type result in dimension 4+1}, Invent. Math. 219 (2020), 153--177.

% \bibitem{FV} A. Figalli and E. Valdinoci, {\em Regularity and Bernstein-type results for nonlocal minimal surfaces}, J. Reine Angew. Math. 729 (2017), 263--273.

% \bibitem{FZ} A. Figalli and Y. R. Zhang, {\em Uniform boundedness for finite Morse index solutions to supercritical semilinear elliptic equations}, Comm. Pure Appl. Math., to appear.

% \bibitem{FrancS} F. Franceschini and J. Serra, {\em Flat nonlocal minimal surfaces are smooth}, forthcoming preprint.

% \bibitem{FR} R.L. Frank and R. Seiringer, {\em Non-linear ground state representations and sharp Hardy inequalities}, J. Funct. Anal. 255 (2008), 3407–-3430.

% \bibitem{Hitchguide} E. Di Nezza, G. Palatucci and E. Valdinoci {\em Hitchhiker's guide to the fractional Sobolev spaces}, Bulletin des Sciences Mathématiques 136 (2012), 521--573.

% \bibitem{FishS} D. Fischer-Colbrie and R. Schoen, {\em The structure of complete stable minimal surfaces in 3-manifolds of nonnegative scalar curvature}, Comm. Pure Appl. Math. 33 (1980), 199--211.

\bibitem{Gas20}  P. Gaspar, {\em The Second Inner Variation of Energy and the Morse Index of Limit Interfaces.}, J. Geom. Anal. 30 (2020), 69–-85.

% \bibitem{GG1} P. Gaspar and M. Guaraco, {\em The Allen-Cahn equation on closed manifolds}, Calc.
% Var. Partial Differential Equations 57 (2018), 101.

% \bibitem{GG2} P. Gaspar and M. Guaraco, {\em The Weyl Law for the phase transition spectrum and the density of minimal hypersurfaces}, Geom. Funct. Anal. 29 (2019), 382--410.

% \bibitem{Gho1} N. Ghoussoub, {\em Duality and perturbation methods in critical point theory}, Cambridge University Press (1993).

% \bibitem{Gho2} N. Ghoussoub, {\em Location, multiplicity and Morse indices of min-max critical points}, J.
% Reine Angew. Math. 417 (1991), 27--76.

% \bibitem{GG} N. Ghoussoub and C. Gui, {\em On a conjecture of De Giorgi and some related problems}, Math. Ann. 311 (1998), 481--491.


\bibitem{Grigsurv} A. Grigor\a'iyan, 
{\em Estimates of heat kernels on Riemannian manifolds.} Spectral theory and geometry (Edinburgh, 1998), 140–225, 
London Math. Soc. Lecture Note Ser., 273, Cambridge Univ. Press, Cambridge, 1999. 

 \bibitem{GriBook} A. Grigor\a'iyan, {\em Heat Kernel and Analysis on Manifolds},  American Mathematical Society, Providence, RIInternational Press, Boston, MA, 2009. xviii+482 pp.


% \bibitem{G} M. Guaraco, {\em Min-max for phase transitions and the existence of embedded minimal hypersurfaces}, J. Differential Geom. 108 (2018), 91--133.

% \bibitem{Guth} L. Guth, {\em Minimax problems related to cup powers and Steenrod squares}, Geom. Funct. Anal. 18 (2009), 1917–1987.

% \bibitem{MdM}
% M. d. M. Gonz\'alez, {\em Gamma convergence of an energy functional related to the fractional Laplacian},
% Calc. Var. Partial Differential Equations 36 (2009), 173--210.

% \bibitem{minspheres} R. Haslhofer and D. Ketover, {\em Minimal 2-spheres in 3-spheres}, Duke Math. J. 168 (2019), 1929--1975.


% \bibitem{IMN} K. Irie, F. C. Marques and A. Neves, {\em Density of minimal hypersurfaces for generic metrics},
% Ann. of Math. 187 (2018), 963--972.

% \bibitem{10equiv} M. Kwaśnicki, {\em Ten Equivalent Definitions of the Fractional Laplace Operator}, Fract. Calc. Appl. Anal. 20 (2017), 7--51.

% \bibitem{LS} A. Lazer and S. Solimini, {\em Nontrivial solutions of operator equations and Morse indices of critical points of min-max type}, Nonlinear Anal. 12 (1988), 761--775.


% \bibitem{LMN} Y. Liokumovich, F. C. Marques and A. Neves, {\em Weyl law for the volume spectrum},
% Ann. of Math. 187 (2018), 1--29.

% \bibitem{MM} L. Modica and S. Mortola, {\em Un esempio di $\Gamma$-convergenza} (Italian), Boll. Un. Mat. Ital. B 14 (1977),  285--299.

% \bibitem{MN} F. C. Marques and A. Neves, {\em Existence of infinitely many minimal hypersurfaces in positive Ricci curvature},
% Invent. Math. 209 (2017), 577--616.

% \bibitem{MN2} F. C. Marques and A. Neves, {\em Morse index and multiplicity of min-max minimal
% hypersurfaces},
% Camb. J. Math. 4 (2016), 463--511.

% \bibitem{MN3} F. C. Marques and A. Neves, {\em Morse index of multiplicity one min-max minimal
% hypersurfaces},
% Adv. Math. 378 (2021), 107527.

% \bibitem{Willmore} F. C. Marques and A. Neves, {\em Min-Max theory and the Willmore conjecture},
% Ann. of Math. 179 (2014), 683--782.

% \bibitem{MNS} F. C. Marques, A. Neves and A. Song, {\em Equidistribution of minimal hypersurfaces for generic metrics}, Invent. Math. 216 (2019), 421--443.

% \bibitem{Moy} J. Moy, {\em $C^{1,\alpha}$ regularity of hypersurfaces of bounded nonlocal mean curvature in Riemannian manifolds}, 2023 (to appear).

% \bibitem{PSV} G. Palatucci, O. Savin and E. Valdinoci, {\em Local and global minimizers for a variational energy involving a fractional norm}, Ann. Mat. Pura Appl. 192, 673–718 (2013).


% \bibitem{Pitts} J. Pitts, {\em Existence and regularity of minimal surfaces on Riemannian manifolds},
% no. 27 in Mathematical Notes, Princeton University Press (1981).

% \bibitem{Pog} A. V. Pogorelov, {\em On the stability of minimal surfaces}, Soviet Math. Dokl. 24 (1981), 274--276.

% \bibitem{silreg} L. Silvestre, {\em H\"{o}lder Estimates for Solutions of Integro-Differential Equations Like The Fractional Laplace}, Indiana University Mathematics Journal 55 (2006), 1155–--1174.

% \bibitem{saloffcoste} L. Saloff-Coste, {\em The heat kernel and its estimates}, Adv. Stud. Pure Math. (2010), 405--436. 

% \bibitem{Savin} O. Savin, {\em Regularity of flat level sets in phase transitions}, Ann. of Math. 169 (2009),  41--78.

% \bibitem{S-new} O. Savin, {\em Rigidity of minimizers in nonlocal phase transitions},  Anal. PDE 11 (2018), 1881--1900.

% \bibitem{S-new2} O. Savin, {\em Rigidity of minimizers in nonlocal phase transitions II}, Anal. Theory Appl. 35 (2019), no. 1, 1--27.


% \bibitem{SV-gamma}
% O. Savin and E. Valdinoci, {\em $\Gamma$-convergence for nonlocal phase transitions},
%  Ann. Inst. H. Poincar\'e Anal. Non Lin\'eaire  29 (2012),  479--500.

% \bibitem{SV} O. Savin and E. Valdinoci, \emph{Regularity of nonlocal minimal cones in dimension 2},  Calc. Var. Partial Differential Equations 48 (2013),  33--39.

% \bibitem{SchoenSimon} R. Schoen and L. Simon, {\em Regularity of stable minimal hypersurfaces}, Comm. Pure Appl. Math. 34 (1981), 741--797.


% \bibitem{Simon} L. Simon, {\em Schauder estimates by scaling},  Calc. Var. Partial Differential Equations 5 (1997),  391--407.

% \bibitem{Simons} J. Simons, {\em Minimal varieties in Riemannian manifolds}, Ann. of Math. 88 (1968),  62--105.

% \bibitem{FSmith} F. Smith, On the existence of embedded minimal 2-spheres in the 3–sphere, endowed
% with an arbitrary Riemannian metric, Ph.D. thesis, supervisor L. Simon, University
% of Melbourne (1982).

% \bibitem{Song} A. Song, {\em Existence of infinitely many minimal hypersurfaces in closed manifolds}, arXiv:1806.08816.

\bibitem{MSW} V. Millot, Y. Sire and K. Wang, {\em Asymptotics for the Fractional Allen–Cahn Equation and Stationary Nonlocal Minimal Surfaces}, Arch. Rational Mech. Anal 231 (2019), 1129–1216.

\bibitem{Stinga} P. R. Stinga, {\em Fractional powers of second order
partial differential operators:
extension problem and regularity theory
}, PhD Thesis, Universidad Autónoma de Madrid (2010).

\bibitem{Yau82} S. T. Yau, {\em Seminar on Differential Geometry}, vol. 102 in Annals of Mathematics Studies, Princeton University Press (1982), 669–706.

% \bibitem{TW} Y. Tonegawa and N. Wickramasekera, {\em Stable phase interfaces in the van der Waals–Cahn–Hilliard theory}, J. Reine
% Angew. Math. 668 (2012), 191-210.

% \bibitem{WangZhou} Z. Wang and X. Zhou, {\em Existence of four minimal spheres in $\Sp^3$ with a bumpy metric}, arXiv:2305.08755



% \bibitem{Wick} N. Wickramasekera, {\em A general regularity theory for stable codimension 1 integral varifolds}, Ann. of Math. 179 (2014), 843-1007.

% \bibitem{Yau82} S. T. Yau, {\em Seminar on Differential Geometry}, vol. 102 in Annals of Mathematics Studies, Princeton University Press (1982), 669–706.

% \bibitem{Zhou} X. Zhou, {\em On the Multiplicity One Conjecture in min-max theory}, Ann. of Math. 192 (2020), 767-820.


\end{thebibliography}
\end{document}